%% file: Main.tex
\documentclass{article}
\usepackage[utf8]{inputenc}
\usepackage[margin=1.5in]{geometry}
\usepackage{amsmath, amsthm, amsfonts, amssymb, slashed, stmaryrd, mathrsfs}
\usepackage{wrapfig}
\usepackage[dvipsnames]{xcolor}
\usepackage{tikz-cd}
\usepackage{aliascnt}
\usepackage{slashed}
\usepackage{graphicx}
\usepackage{svg}
\usepackage{float} 
\graphicspath{ {./figures/} }

\usepackage[
backend=biber,
style=alphabetic,
sorting=ynt, 
isbn=false,
url=false,
sorting=nyt,
]{biblatex}
\AtBeginBibliography{\small}

\usepackage[hidelinks]{hyperref} 
\hypersetup{
    colorlinks=true,
    linkcolor=MidnightBlue,
    filecolor=MidnightBlue,      
    urlcolor=MidnightBlue,
    citecolor=MidnightBlue,
    pdfpagemode=FullScreen,
    }

\input{letters}

\input{macros}

\input{env}

\usepackage[english]{babel}
\addto\extrasenglish{

}

\title{Holomorphic Floer theory and the Fueter equation}
\author{Aleksander Doan, Semon Rezchikov}
\date{~\\ \small\today}

\begin{document}
\maketitle
 
 \abstract{We outline a proposal for a $2$-category $\mathrm{Fuet}_M$ associated to a hyperk\"ahler manifold $M$, which categorifies the subcategory of the Fukaya category of $M$ generated by complex Lagrangian submanifolds. Morphisms in this $2$-category are formally the Fukaya--Seidel categories of holomorphic symplectic action functionals. As such, $\mathrm{Fuet}_M$ is based on counting maps $[0,1]\times\R^2\to M$ satisfying the Fueter equation with boundary values on complex Lagrangians. 
 
 We make the first step towards constructing this category by establishing some basic analytic results about Fueter maps, such as an energy bound and a maximum principle. When $M=T^*X$ is the cotangent bundle of a Kähler manifold $X$ and $L_0, L_1$ are the zero section and the graph of the differential of a holomorphic function $F \colon X \to \C$, we prove that Fueter maps correspond to the complex gradient trajectories of $F$ in $X$, which relates our proposal to the Fukaya--Seidel category of $F$. This is a complexification of Floer's theorem on pseudo-holomorphic strips in cotangent bundles. 
 
 Throughout the paper, we suggest problems and research directions for  analysts and geometers that may be interested in the subject. }

\setcounter{tocdepth}{2}
\tableofcontents
\newpage

\input{Introduction}
\input{Categories}

\input{FueterEquation}
\input{Convexity}
\input{CotangentBundles}

\appendix
\input{TQFT}
\input{CategoricalAppendix}

\input{FloerTheorem}
\input{Conventions}

\newpage
\printbibliography

\end{document}

%% file: letters.tex
\newcommand{\rd}{{\mathrm d}}

\newcommand{\cA}{\mathcal{A}}

\newcommand{\cF}{\mathcal{F}}

\newcommand{\cH}{\mathcal{H}}
\newcommand{\cI}{\mathcal{I}}

\newcommand{\cM}{\mathcal{M}}

\newcommand{\cP}{\mathcal{P}}

\newcommand{\cX}{\mathcal{X}}

\newcommand{\sA}{\mathscr{A}}

\newcommand{\sF}{\mathscr{F}}

\newcommand{\sI}{\mathscr{I}}

\newcommand{\sQ}{\mathscr{Q}}

\newcommand{\fn}{{\mathfrak n}}

%% file: macros.tex
\newcommand{\Z}{\mathbb{Z}}
\newcommand{\R}{\mathbb{R}}
\newcommand{\C}{\mathbb{C}}
\renewcommand{\H}{\mathbb{H}}
\newcommand{\tensor}{\otimes}
\renewcommand{\Re}{\mathrm{Re}}
\renewcommand{\Im}{\mathrm{Im}}

\newcommand{\Hom}{\mathrm{Hom}}
\newcommand{\End}{\mathrm{End}}
\newcommand{\bHom}{\overline{\mathrm{Hom}}}

\newcommand{\del}{\partial}
\newcommand{\delbar}{\overline\partial}
\newcommand{\vol}{\mathrm{vol}}
\newcommand{\Grad}{\nabla}

\newcommand{\loc}{\mathrm{loc}}
\newcommand{\id}{\mathrm{id}}

\newcommand{\Sp}{\mathrm{Sp}}
\newcommand{\SO}{\mathrm{SO}}
\newcommand{\GL}{\mathrm{GL}}
\newcommand{\U}{\mathrm{U}}

\newcommand{\Diff}{\mathrm{Diff}}
\newcommand{\Symp}{\mathrm{Symp}}

%% file: env.tex
\numberwithin{equation}{section}

\theoremstyle{plain}

\newcommand{\mynewtheorem}[2]{
  \newaliascnt{#1}{dummy}
  \newtheorem{#1}[#1]{#2}
  \aliascntresetthe{#1}
  \expandafter\def\csname #1autorefname\endcsname{#2}
}

\theoremstyle{plain}
\mynewtheorem{theorem}{Theorem}
\mynewtheorem{lemma}{Lemma}
\mynewtheorem{proposition}{Proposition}
\mynewtheorem{corollary}{Corollary}
\mynewtheorem{conjecture}{Conjecture}

\theoremstyle{definition}
\mynewtheorem{example}{Example} 
\mynewtheorem{definition}{Definition}
\mynewtheorem{remark}{Remark}
\mynewtheorem{question}{Question}
\mynewtheorem{problem}{Problem}
\mynewtheorem{assumption}{Assumption}

\theoremstyle{plain}


%% file: Introduction.tex
\section{Introduction}

This paper studies a complexification of Lagrangian Floer homology. Like its real counterpart, this theory arises from a $\sigma$-model: it involves the study of maps
\begin{equation*}
    U \colon E \to M
\end{equation*}
from a three-manifold $E$, which we will take to be
\begin{equation*}
    E = [0,1]\times \R^2,
\end{equation*}
to a manifold $M$ of real dimension $4n$ equipped with an almost quaternionic structure $(I,J,K)$, that is: a triple of almost complex structures satisfying the quaternionic relation $IJ = K$. The analog of the pseudo-holomorphic map equation used in Lagrangian Floer theory,
\begin{equation}
\begin{gathered}
    u \colon \R\times[0,1] \to M, \\
    \label{Eq_PseudoHolomorphicIntroduction}
      \partial_t u + J(u)\partial_\tau u = 0,
\end{gathered}
\end{equation}
where $(t,\tau)$ are the coordinates on $\R\times[0,1]$, is the \emph{Fueter equation}:
\begin{equation}
\label{Eq_FueterIntroduction}
\begin{gathered}
    U \colon [0,1]\times\R^2 \to M, \\
     I(U)\partial_\tau U + J(U)\partial_s U + K(U)\partial_t U = 0,
\end{gathered}
\end{equation}
where $(\tau,s,t)$ are the coordinates on $[0,1]\times\R^2$. This equation is perhaps the simplest higher-dimensional generalization of the pseudo-holomorphic map equation which shares its good analytical properties. As such, many analytical aspects of the Fueter equation are tractable and can be developed systematically by analogy with pseudo-holomorphic maps, although some core questions—most importantly, in compactness theory—remain challenging  \cite{Hohloch2009, Walpuski2017c}. On the other hand, due to the kinship of the Fueter equation with the pseudo-holomorphic map equation, it conjecturally gives rise to a host of algebraic structures generalizing the well-known structures appearing in symplectic topology, with rich connections to low-dimensional topology, representation theory, and mathematical physics. The goal of this paper is to explore some of these ideas and connections.

The Fueter equation has been previously studied from many different perspectives. We list some of them below:  
\begin{itemize}
    \item The linear variant of the Fueter equation, with $M$ a quaternionic vector space, is the $3d$ massless Dirac equation. Solutions to the Dirac equation on $\R^3$ are known to obey various quaternionic analogs of phenomena in complex analysis \cite{Fueter1935}. Indeed, finding a quaternionic analog of holomorphic maps was the original motivation for Fueter's paper.
    \item Solutions to the Fueter equation whose domain $E$ is a closed, framed three-manifold are critical points of a functional on the space of maps $E \to M$. Gradient trajectories of this functional are solutions to the four-dimensional version of the Fueter equation for maps $E \times \R \to M$. When $M$ is a flat hyperkähler manifold, the moduli spaces of solutions to these equations can be used to build an invariant of $M$ analogous to Hamiltonian Floer homology \cite{Hohloch2009}.
    \item The Fueter equation arises naturally as a degeneration limit of various equations in gauge theory, such as the instanton equation on Riemannian manifolds with special holonomy \cites{Donaldson1998, Donaldson2011, Haydys2012, Walpuski2017a, Walpuski2017b} and the generalized Seiberg--Witten equations on three- and four-dimensional manifolds \cite{Taubes2013, Taubes2015, Taubes2016, Haydys2015a, Walpuski2021}. In both  situations, $M$ is the hyperkähler quotient of a hyperkähler manifold by an action of a compact Lie group, a typical example being the moduli space of framed instantons on $\R^4$. There is also a parallel picture for calibrated submanifolds of special holonomy manifolds. 
    \item Maps satisfying the Fueter equation appear as a result of applying supersymmetric localization to the physical 3d $N=4$ $\sigma$-model. The corresponding result for the 4d N=2 theory, which reduces to the 3d N=4 theory, dates back at least to \cite{anselmi-fre}. As such, Fueter maps  are a mathematical incarnation of the topological field theory underlying 3d mirror symmetry—an active topic in mathematical physics and representation theory, which thus far has been treated largely from an algebraic perspective. 
\end{itemize}

Our motivations are closely aligned with the last point above, although we learned this physical perspective after thinking about the results of this paper and much of the algebraic structure indicated below from a purely mathematical viewpoint. Heuristically, one can arrive at many expectations regarding the Fueter equation by complexifying known results in symplectic Floer theory, with the proviso that
\begin{center}
    \emph{complexification categorifies symplectic invariants}.
\end{center}

Specifically, we study the analytic underpinnings of the complexification of a classical theorem of Floer on cotangent bundles \cite{Floer1988, Floer1988a}, which we now briefly review. Let $X$ be a closed manifold with a Morse function $f \colon X \to \R$. The cotangent bundle $T^*X$ is a symplectic manifold and the zero section and graph of $\rd f$ are two Lagrangian submanifolds, which we denote by $L_0$ and $L_1$. Floer proved that if the $C^2$ norm of $f$ is sufficiently small (which can be always guaranteed by rescaling $f$), then the Lagrangian Floer homology of $L_0$ and $L_1$ is canonically identified with the Morse homology of $f$:
\begin{equation}
    \label{Eq_MorseEqualsFloer}
    HF_*(L_0,L_1) = HM_*(X,f). 
\end{equation}
The two homology groups have the same generators, as the intersection points of $L_0$ and $L_1$ correspond to the critical points of $f$, so the proof consists in identifying the differentials. In Morse theory they are given by counting gradient trajectories of $f$, while in Floer homology by counting pseudo-holomorphic strips, that is: solutions to \eqref{Eq_PseudoHolomorphicIntroduction} with boundary on $L_0$ and $L_1$ and asymptotic to $L_0\cap L_1$. Floer proved \eqref{Eq_MorseEqualsFloer} by constructing a time-dependent almost complex structure $J_\tau$ on $T^*X$ compatible with the symplectic structure and with the property that $J_\tau$-holomorphic strips in $M=T^*X$ correspond to the gradient trajectories of $f$ in $X$. 

\emph{We prove a complexified version of Floer's correspondence.} On the Morse theory side, we assume that $X$ is a Kähler manifold and $F \colon X \to \C$ is a holomorphic function. Denote by $I$ the complex structure on $X$. Gradient trajectories in Floer's theorem are replaced by the \emph{complex gradient trajectories} of $F$, that is: maps $v \colon \R^2 \to X$ satisfying
\begin{equation*}
    \delbar_{I}(v) = i\nabla F(v)
\end{equation*}
or, in coordinates $(s,t)$ on $\R^2$: 
\begin{equation}
    \label{Eq_ComplexGradient}
    \partial_s v + I(v)(\partial_t v - \nabla\Re(F)(v)) = 0. 
\end{equation}
This equation is also known as the \emph{Witten equation} or \emph{$\zeta$-instanton equation} in the physics literature. Of course, it is the same as the non-homogenous perturbation of the $I$-holomorphic map equation by the Hamiltonian function $\Im(F)$, introduced by Floer. Solutions with $\partial_s v = 0$ are simply the real gradient trajectories of $\Re(F)$. 

On the Lagrangian side, we consider, as before, the cotangent bundle $M=T^*X$. The complex structure on $X$ induces one on $M$, and $M$ has a canonical holomorphic symplectic form $\Omega$. Let $L_0, L_1$ be the zero section and the graph of $\Re(\rd F)$. They are complex submanifolds of $M$ which are Lagrangian with respect to $\Omega$. Pseudo-holomorphic strips \eqref{Eq_PseudoHolomorphicIntroduction} are replaced by \emph{Fueter strips} in $M$. By definition, a Fueter strip is a map $U \colon [0,1]\times \R^2 \to M$ satisfying the Fueter equation \eqref{Eq_FueterIntroduction}, with boundary on $L_0$ and $L_1$, asymptotic to $L_0\cap L_1$ as $t\to\pm\infty$ and to pseudo-holomorphic strips as $s\to\pm\infty$. These asymptotic boundary  conditions are explained in detail in \autoref{Sec_FueterEquation}. 

\begin{theorem}
    \label{Thm_FloerCorrespondence}
    Let $X$ be a compact Kähler manifold with boundary and let $F \colon X \to \C$ be a holomorphic function. If $F$ is $C^2$ small, then there is a $\tau$-dependent almost quaternionic structure $(I_\tau,J_\tau,K_\tau)$ on $M=T^*X$ such that all Fueter strips
    \begin{gather*}
        U \colon [0,1] \times \R^2 \to M, \\
        I_\tau(U)\partial_\tau U + J_\tau(U)\partial_s U + K_\tau(U)\partial_t U = 0,
    \end{gather*}
    with boundary on $L_0$ and $L_1$, the zero section and the graph of $\Re(\rd F)$, correspond to the complex gradient trajectories of $F$ in $X$.
\end{theorem}

\begin{remark}
When restricted to the asymptotic boundary $s\to\pm\infty$, this result recovers Floer's correspondence between the real gradient trajectories of $\Re(F)$ and pseudo-holomorphic strips. 
\end{remark}

\begin{remark}
    While \autoref{Thm_FloerCorrespondence} concerns compact manifolds with boundary and arbitrary holomorphic functions, typically one considers the complex gradient equation for a holomorphic Morse function $F \colon \widehat X \to \C$ on a noncompact, complete Kähler manifolds (or, more generally, a symplectic Lefschetz fibration). Under appropriate assumptions on the growth of $F$ and the geometry of $\widehat X$ at infinity, all critical points and complex gradient trajectories of $F$ lie in a fixed compact subset $X\subset\widehat X$ \cite{Wang2022}, \cite{fan2018fukaya}. There is an analogous maximum principle for Fueter strips  discussed below. 
\end{remark}

\begin{remark}
    This result is a special case of \autoref{Thm_FloerCorrespondencePerturbed} about the cotangent bundles of almost complex manifolds. In the general case, the correspondence holds for a perturbation of the Fueter equation by first order terms which are proportional to the distance from the zero section and the torsion of a natural connection on the almost complex manifold. This perturbation term vanishes if and only if the manifold is Kähler.
\end{remark}

\autoref{Thm_FloerCorrespondence} is a complexification of Floer's result in the sense that a real Morse function and its gradient trajectories have been replaced by their complex analogs. At the same time, our theorem can be seen as the first steps toward categorifying Floer's isomorphism \eqref{Eq_MorseEqualsFloer}. This is best explained by the following three general predictions:
\begin{enumerate}
    \item For a holomorphic Morse function $F \colon X \to \C$ on an Kähler manifold or, more generally, a symplectic Lefschetz fibration, there should be an associated $A_\infty$-category $\mathrm{FS}(X,F)$ constructed using the complex gradient trajectories of $F$. This category is conjecturally quasi-isomorphic to the Fukaya--Seidel category of $F$ defined in \cite{Seidel2008} by counting pseudo-holomorphic polygons with boundaries on the Lefschetz thimbles of $F$.
    Summarizing the analogies with the real case:
    \begin{center}
        \begin{tabular}{ c|c } 
            Morse theory & Complex Morse theory \\
            \hline
            smooth manifold $X$ & Kähler manifold $X$ \\
            Morse function $f \colon X \to \R$ & holomorphic Morse function $F \colon X \to \C$ \\
            gradient trajectories of $f$ & complex gradient trajectories of $F$ \\
            Morse homology $HM_*(X,f)$ & Fukaya--Seidel category $\mathrm{FS}(X,F)$ 
        \end{tabular}
    \end{center}
    \item Let $(M,I,J,K)$ be a hyperkähler manifold (or, more generally, an almost quaternionic manifold with three symplectic forms). For every pair of complex Lagrangians $L_0, L_1 \subset M$, i.e. submanifolds which are holomorphic with respect to $I$ and Lagrangian with respect to the symplectic forms corresponding to $J$ and $K$, there should be an associated $A_\infty$-category $\mathrm{Fuet}(L_0,L_1)$ constructed using Fueter strips in $M$ with boundary on $L_0$ and $L_1$.
    Comparing again to the real case:
    \begin{center}
        \begin{tabular}{ c|c } 
          Floer theory & Complex Floer theory \\
            \hline
            symplectic manifold $M$ & hyperkähler manifold $M$ \\
            Lagrangians $L_0, L_1$ & complex Lagrangians $L_0, L_1$ \\
            pseudo-holomorphic strips in $M$ & Fueter strips in $M$ \\
            Floer homology $HF_*(L_0,L_1)$ & Fueter category $\mathrm{Fuet}(L_0,L_1)$  
        \end{tabular}
    \end{center} 
    
    \item We expect these two categories to be related by the following conjecture, which categorifies \eqref{Eq_MorseEqualsFloer}. We discuss it in detail in \autoref{sec:cotangent-bundles-conjecture}.
    \begin{conjecture}
    \label{conj:fuet-equals-fs-1}
        Let $F \colon X \to \C$ be a holomorphic Morse function on a Kähler manifold. If $M = T^*X$ and $L_0$, $L_1$ are the zero section and graph of $\Re(\rd F)$, then there is an equivalence of $A_\infty$-categories
        \begin{equation*}
            \mathrm{Fuet}(L_0,L_1) = \mathrm{FS}(X,F).
        \end{equation*}
    \end{conjecture}
\end{enumerate}

The first prediction was made independently by Gaiotto--Moore--Witten \cite{Gaiotto2015,Gaiotto2017} and Haydys \cite{Haydys2015}, with some of the ideas going back to the seminal work of Donaldson and Thomas on higher-dimensional gauge theory \cite{Donaldson1998}. See also \cite{Kapranov2016} for a discussion of the algebraic structures arising from the Gaiotto--Moore--Witten proposal. The putative $A_\infty$ category $\mathrm{FS}(X,F)$ associated with a holomorphic function $F \colon X \to \C$ has:
\begin{itemize}
    \item Objects given by the critical points of $F$.
    \item Morphisms given by graded vector spaces generated by gradient flowlines between the critical points of $F$. That is, for two critical points $p_-$ and $p_+$, let $\theta$ be the slope of the line connecting the corresponding critical values. $\Hom(p_-,p_+)$ is the vector space whose basis are the maps $\gamma \colon \R \to X$ satisfying
    \begin{equation*}
        \frac{\rd\gamma}{\rd t} = \nabla \Re(e^{-i\theta}F), \quad \lim_{t \to \pm\infty}\gamma(t) = p_\pm.
    \end{equation*}
    \item The differential operation given by complex gradient trajectories \eqref{Eq_ComplexGradient} of $e^{-i\theta} F$ asymptotic to the gradient flowlines of $\Re(e^{-i\theta} F)$ as $s\to\pm\infty$ and to $p_\pm$ as $t\to\pm\infty$. The higher $A_\infty$-operations are also defined by counting complex gradient trajectories but with more complicated asymptotic boundary conditions.
\end{itemize}
If the critical values of $F$ are in a convex position, this construction is expected to recover the Fukaya--Seidel category of $F$. In general, one must either perform a complex algebraic operation on the counts of the complex gradient trajectories to define the $A_\infty$-operations \cites{Gaiotto2015, Kapranov2016}, or modify the construction to use curved gradient trajectories \cite{Haydys2015}. A rigorous construction of $\mathrm{FS}(X,F)$ involves solving a number of interesting analytical problems concerning the moduli spaces of complex gradient trajectories. Significant progress in this direction following \cite{Haydys2015} was made in \cite{Wang2022}.

The second prediction—that we can associate an $A_\infty$-category $\mathrm{Fuet}(L_0,L_1)$ with a pair of $I$-complex Lagrangians $L_0, L_1$ inside a hyperkähler manifold $(M,I,J,K)$— is an infinite-dimensional version of the first one. To our knowledge, such a proposal has not yet been described in the literature, although it is closely related to ongoing projects of Kontsevich and Soibelman \cite{Kontsevich2020} from an algebraic perspective, and Bousseau \cite{Bousseau2022} who proposed a number of interesting conjectures relating the Fueter $2$-category to Donaldson--Thomas theory. The idea of associating a Lagrangian $2$-category with a hyperkähler manifold have been previously considered in physics by Khan \cite{ahsan-khan}; see also Bullimoore--Gaiotto--Dimofte--Hilburn \cite[Appendix~A]{Bullimore2016} for connections with $3d$ mirror symmetry.

Given $M$ and $L_0, L_1$ as above, let $(\omega_I,\omega_J,\omega_K)$ be the triple of Kähler forms corresponding to $(I,J,K)$. The starting point of holomorphic Floer theory is the observation that the space of paths 
\begin{equation*}
    \cP = \{ \gamma \colon [0,1] \to M \ | \ \gamma(0) \in L_0, \ \gamma(1) \in L_1 \}
\end{equation*}
is an infinite-dimensional Kähler manifold with respect to the complex structure and metric induced by $I$ and $\omega_I$. Moreover, the \emph{holomorphic symplectic action functional}
\begin{equation}
\label{eq:complex-symplectic-action-functional-lagrangian}
    \cA_\C = \cA_J + i\cA_K \colon \cP \to \C
\end{equation}
combining the symplectic action functionals with respect to $\omega_J$ and $\omega_K$, is a holomorphic function on $\cP$. (We ignore here that $\cA_J$ and $\cA_K$ are actually multi-valued functions. For example, we can assume that $\omega_J, \omega_K$ are exact and $L_0,L_1$ are exact Lagrangians.) Therefore, we can formally apply the construction the Fukaya--Seidel category, outlined above, to $\cA_\C \colon \cP \to \C$. This is discussed in \autoref{Sec_CategoricalAspects}. The result should be an $A_\infty$-category $\mathrm{Fuet}(L_0,L_1)$ with
\begin{itemize}
    \item Objects given by the critical points of $\cA_\C$, which correspond to the intersection points of $L_0$ and $L_1$, as in Lagrangian Floer theory.
    \item Morphism spaces between two intersection points $p_-$ and $p_+$ generated by the gradient trajectories of $\Re(e^{-i\theta} \cA_\C)$ where $\theta$ is the slope of the line connecting the corresponding critical values of $\cA_\C$. As in Lagrangian Floer theory, these trajectories can be interpreted as $J_{\theta+\pi}$-holomorphic strips $u \colon \R\times[0,1] \to M$ with boundary on $L_0, L_1$, where 
    \begin{equation}
        \label{Eq_RotatedComplexStructure}
        J_\theta = (\cos\theta) J + (\sin\theta) K.
    \end{equation}
    \item The differential and $A_\infty$-operations defined by counting complex gradient flowlines of $e^{-i\theta}\cA_\C$. As we show in \autoref{Sec_CategoricalAspects}, these correspond to the Fueter strips $U \colon [0,1]\times\R^2 \to M$ with respect to the rotated quaternionic triple $(I,J_\theta,K_\theta)$, where $K_\theta = IJ_\theta$, with boundary on $L_0, L_1$, and asymptotic to $J_\theta$-holomorphic strips as $s\to\pm\infty$ and to the intersection points $p_\pm$ as $t\to \pm\infty$. 
\end{itemize}

This brings us to the third prediction. If $F \colon X \to \C$ is a holomorphic Morse function on a Kähler manifold, then $M = T^*X$ is a holomorphic symplectic manifold, and the zero section and the graph of $\rd(\Re F)$ are complex Lagrangians. While $M$ does not have to carry a complete hyperkähler metric, we expect that $\mathrm{Fuet}(L_0,L_1)$ is still defined by counting solutions to the Fueter equation with respect to the time-dependent almost quaternionic triple from \autoref{Thm_FloerCorrespondence}. If this is the case, \autoref{conj:fuet-equals-fs-1} would be a direct consequence of the correspondence between Fueter strips in $M$ and complex gradient flowlines in $X$ established by \autoref{Thm_FloerCorrespondence}.

In general, defining $\mathrm{Fuet}(L_0,L_1)$ is a challenging task, which involves a number of interesting analytical questions. Some of these foundational questions, such as transversality and compactness for the Fueter equation, are considered in \cite{Hohloch2009, Walpuski2017c, Esfahani2022}. Compactness theory, in particular, requires further study, as it is currently unclear what role, if any, is played by singular solutions to the Fueter equation. This is parallel to the recent developments in gauge theory and calibrated geometry, which suggest that such singular solutions are an unavoidable part of the theory \cite{Taubes2014,Takahashi2018, Doan2021, Donaldson2021, Haydys2022}. In this paper we do not tackle these foundational difficulties. We do, however, establish some basic analytic results on Fueter strips, mimicking well-known theorems in Lagrangian Floer theory. 

\begin{remark}
    The proof of  \autoref{Thm_FloerCorrespondence} excludes the possibility of the local sphere bubbling or the formation of essential singularities described in \cite{Walpuski2017c}, assuming that $X$ is exact and $F$ is such that its complex gradient trajectories obey an a priori $C^0$ bound. This is interesting as the corresponding metrics on $T^*X$ may have complicated curvature, and thus existing methods for excluding bubbling of Fueter maps do not apply to this setting. For more details, see  \autoref{remark:exclude-fueter-bubbling}.
\end{remark}

The first such result is an a priori energy bound for Fueter strips. Let $(M,I,J,K)$ be an almost quaternionic manifold and $L_0,L_1 \subset M$ be two submanifolds. The main examples are, of course, hyperkähler manifolds and complex Lagrangians, but the next two theorems work in greater generality. Let $U \colon [0,1] \times \R^2 \to M$ be a map with boundary on $L_0, L_1$. If $U$ is asymptotic to $J$-holomorphic strips as $s\to\infty$ in the sense explained earlier, then the integral of $|\rd U|^2$ cannot be finite. Instead, by analogy with pseudo-holomorphic strips, we introduce the following notion of energy:
\begin{equation}
\label{eq:energy-introduction}
    E(U) = \int_{[0,1]\times\R^2} |\partial_s U|^2 + |\partial_t U - J(U)\partial_\tau U|^2.
\end{equation}
To state the energy bound, we need two more notions. First, let $[U|_{\R^2}]$ be a homology class in $M$ relative to the limiting $J$-holomorphic strips, represented by the restriction of $U$ to any slice $\{\tau\}\times\R^2$; this is well-defined and independent of $\tau$. Second, let us say that a complex two-form $\eta$ on $M$ is  \emph{$KI$-taming} if it is of type $(2,0)$ with respect to $J$ and $\Re(\eta)$ tames $K$ and $\Im(\eta)$ tames $I$. 

\begin{theorem}[Energy bound]
    \label{Thm_EnergyIdentity}
    Let $(M,I,J,K)$ be an almost quaternionic manifold. Suppose that there is a closed $KI$-taming two-form $\eta$ on $M$ and $L_0, L_1 \subset M$ are Lagrangian with respect to $\Re(\eta)$. There is a constant $C>0$ such that for every Fueter strip $U \colon [0,1]\times \R^2 \to M$ with boundary on $L_0, L_1$,
    \begin{equation*}
        E(U) \leq -C \langle [\Im(\eta)], [U|_{\R^2}]\rangle,
    \end{equation*}
    where the right-hand side denotes the pairing between the second cohomology and homology of $M$ relative to the limiting $J$-holomorphic strips.
\end{theorem}

\begin{remark}
    If $(M,I,J,K)$ is a hyperkähler manifold and $(\omega_I,\omega_J,\omega_K)$ is the corresponding triple of Kähler forms, then $\eta = \omega_K + i\omega_I$ is a $KI$-taming two-form, and  
    \begin{equation*}
        E(U) = - \langle [\omega_K], [U|_{\R^2}]\rangle.
    \end{equation*}
\end{remark}
 
The second result is a maximum principle for Fueter strips. In \autoref{Sec_Convexity} we introduce the notion of an almost quaternionic manifold with an $IJK$-conical end. This is a quaternionic analog of pseudo-convexity in complex geometry. We, moreover, define what it means for submanifolds to be conical with respect to two of the three almost complex structures. As an example, we show that the cotangent bundle of an almost complex manifold equipped with a pseudo-convex function can be made into an almost quaternionic manifolds with an $IJK$-conical end, with the zero section being $JK$-conical and $I$-complex Lagrangian. With these notions, we prove the following. 

\begin{theorem}[Maximum principle]
    \label{Thm_MaximumPrinciple}
    Let $(M,I,J,K)$ be an almost quaternionic manifold with an $IJK$-conical end and let $L_0, L_1 \subset M$ be submanifolds which are either compact or $JK$-conical with compact intersection $L_0 \cap L_1$. There exists a compact subset $M_0 \subset M$ with the following property. If $U \colon [0,1]\times\R^2 \to M$ is a Fueter strip with boundary on $L_0$ and $L_1$, then the image of $U$ is contained in $M_0$. 
\end{theorem}

In symplectic geometry, pseudo-convex functions are often used to confine pseudo-holomorphic curves. In particular, due to the the Weinstein tubular neighborhood theorem, Floer's result on pseudo-holomorphic strips in cotangent bundles gives a description of local pseudo-holomorphic strips between two Lagrangians in any symplectic manifold in terms of Morse-theoretic, or topological data, and thus gives the first approximation to any invariant in Lagrangian Floer theory. We expect that the convexity theory discussed in this paper can be used similarly to compute the local Fueter trajectories between complex Lagrangians in a general hyperkähler manifold.

This completes the summary of the results of this paper. These analytical results are motivated by a rich set of expectations regarding the algebraic and categorical structures that should arise from the study of the Fueter equation. Indeed, Floer's classical theorem mentioned earlier is the basic input used to compute the Fukaya category of the cotangent bundle. The Fukaya category is an $A_\infty$-category associated with a symplectic manifold. From a physical perspective,  it is related to the $A$-twist of the $2$-dimensional $(2,2)$ superconformal $\sigma$-model with target $M$. Floer's isomorphism \eqref{Eq_MorseEqualsFloer} is the simplest computation of the morphism group between Lagrangians $L_0$, $L_1$, thought of as objects in the Fukaya category of $M=T^*X$. 

In the complex setting, all of this structure should be categorified. Namely, we expect that that if $M$ is a hyperkähler manifold (or some weak version of it, such as a holomorphic symplectic manifold or an almost quaternionic manifold equipped with three symplectic forms), the $A_\infty$-categories $\mathrm{Fuet}(L_0,L_1)$ for all pairs of complex Lagrangians assemble into an $(A_\infty, 2)$-category $\mathrm{Fuet}_M$, in the sense of \cite{bottman-carmel}. Its objects are given by complex Lagrangians, morphisms by intersection points, and $2$-morphisms by pseudo-holomorphic strips between the Lagrangians. This $2$-category should categorify the subcategory of the usual Fukaya category of $(M,\omega_I)$ whose objects are complex Lagrangians. From a physical point of view, this $2$-category is the $2$-category of branes in $M$ for the $A$-twist of the $3d$ $N=4$ $\sigma$-model. By analogy with the real case, we expect that extensions of \autoref{Thm_FloerCorrespondence} will allow us to compute the \emph{category of morphisms} $\mathrm{Fuet}_{T^*X}(L_0, L_1)$ between the pair of objects $(L_0, L_1)$ of the Fueter $2$-category $\mathrm{Fuet}_{T^*X}$.

While a rigorous construction of $\mathrm{Fuet}_M$ involves various nontrivial difficulties,  many formal aspects are clear. Under the philosophy of $3d$ miror symmetry, the $2$-category $\mathrm{Fuet}_M$ is conjectured to be equivalent to the Kapustin--Rozansky--Saulinas $2$-category of the mirror hyperk\"ahler manifold $M^\vee$. The motivation for this conjecture comes from the physics literature, where it is postulated that $3d$ $N=4$ theories come in pairs, and the $A$-twist of one model is equivalent with the $B$-twist of the mirror model. Some initial work to model the Fueter $2$-category via perverse sheaves of categories has been pursued in \cite{gammage-hilburn-mazel-gee-2022}.  In the case of $M = T^*(\C^*)^n$, the results presented in this paper seem to be connected to an enhancement of homological mirror symmetry for toric varieties.

We provide much additional context for the conjectural categorical structure of the $3d$ topological field theory, as well for some of its connections to other topics in mathematical physics, in \autoref{Sec_CategoricalAspects}. While \autoref{Sec_CategoricalAspects} does not contain any proofs, it is the core of the paper in the sense that it motivates the results proved later. In \autoref{Sec_FueterEquation} we proceed with the analytical study of the Fueter equation; in particular, this section contains the proof of \autoref{Thm_EnergyIdentity}. In \autoref{Sec_Convexity} we develop a quaternionic analog of pseudo-convexity theory and prove \autoref{Thm_MaximumPrinciple}. Finally, in \autoref{Sec_CotangentBundles} we apply the general theory to the case of cotangent bundles, proving \autoref{Thm_FloerCorrespondence}.

We hope that the differential-geometric study of the $A$-side of $3d$ mirror symmetry gives rise to an active program of research. Throughout the paper, we suggest problems and research directions for analyts and geometers that may be interested in the subject. 
 
\paragraph{Acknowledgements.}

We thank Justin Hilburn and Ben Gammage for introducing us to the physical aspects of $3d$ mirror symmetry, reviewed in \autoref{sec:3d-mirror-symmetry}, and for discussions regarding algebraic aspects in \autoref{Sec_CategoricalAppendix}. We are grateful to Ahsan Khan, who has independently developed categorical aspects of the Fueter TQFT from the perspective of quantum field theory \cite{PIRSA-khan, ahsan-khan}, for pointing out the potential importance of twisted complexes in the composition in the Fueter $2$-category. We thank Andrew Hanlon for clarifying some errors in our earlier statement of Theorem \autoref{thm:abouzaids-toric-variety-theorem} as well as for interesting discussions on this topic. Finally, we express gratitude to to Mohammed Abouzaid, Denis Auroux, Andriy Haydys,  Yan Soibelman, and Donghao Wang for discussions which contributed to this paper. 

The first author is supported by the National Science Foundation, Simons Society of Fellows, and Trinity College, Cambridge. The second author is supported by the National Science Foundation and by Simons Foundation Collaboration grant “Homological Mirror Symmetry and Applications” (Award \# 385573).  

%% file: Categories.tex
\section{Formal aspects}
\label{Sec_CategoricalAspects}
We expect that the topological field theory based on the Fueter equation assigns a $2$-category to a hyperkähler manifold $X$. In \autoref{sec:aspects-of-tqft} we provide an introduction to the connection between categorical structures and topological field theories. In this informal section, we outline the expected behavior of the Fueter $2$-category. We hope that it will be useful to those interested in the geometry and algebra of $3d$ mirror symmetry. Throughout this section, as well as in \autoref{Sec_CategoricalAppendix},  we point out problems which we see as essential to the rigorous development of the $3d$ $A$-model.

\subsection{The Fueter $2$-category}
\label{sec:fueter-tqft}
Let $M$ be a hyperkähler manifold with complex structures $I,J,K$. We consider $M$ as a complex manifold with the distinguished complex structure $I$. The remaining part of the hyperkähler structure is encoded in the Riemannian metric and holomorphic symplectic form 
\begin{equation*}
    \Omega = \omega_J + i \omega_K
\end{equation*}
where $\omega_J$, $\omega_K$ are the Kähler forms corresponding to $J$, $K$. Fix also a background angle $\hat\theta \in [0,2\pi)$; the final structure will not depend on that choice up to isomorphism. To a very initial approximation, we expect to define a $2$-category $\mathrm{Fuet}_M$ with the following features.
\begin{itemize}
    \item The objects of $\mathrm{Fuet}_M$ are complex Lagrangians $L \subset M$, i.e. submanifolds which are $I$-holomorphic and Lagrangian with respect to $\Omega$ and thus Lagrangian with respect to all real symplectic forms
    \begin{equation*}
        \omega_\theta = \Re(e^{-i\theta}\Omega) = (\cos\theta)\omega_J + (\sin\theta)\omega_K \quad\text{for } \theta \in [0,2\pi).
    \end{equation*}
    \item A \emph{hom-category} $\mathrm{Fuet}_M(L_0, L_1)$ is, formally, the Fukaya--Seidel category of the \emph{holomorphic symplectic action functional} $\cA_\C$, defined in \eqref{eq:complex-symplectic-action-functional-lagrangian}, on the space of paths from $L_0$ to $L_1$
    \begin{equation*}
    \cP = \cP(M; L_0,L_1) = \{ \gamma \colon [0,1] \to M \ | \ \gamma(0) \in L_0, \ \gamma(1) \in L_1 \}.
    \end{equation*}
    Strictly speaking, what is well-defined is the differential of $\cA_\C$, which is a closed holomorphic $(1,0)$-form on $\cP$ given by
    \begin{equation}
    \label{eq:holomorphic-action-1-form}
        \rd\cA_\C(\gamma) v = \int_0^1 \Omega\left( v(\tau), \frac{\rd \gamma}{\rd \tau} \right) \rd \tau
    \end{equation}
    for $\gamma \in \cP$ and $v \in T_\gamma \cP = \Gamma([0,1], \gamma^*TM)$, with $\tau$ denoting the coordinate on $[0,1]$. 
    As usual \cite{Hofer1995}, this one-form may or may not be the differential of globally defined holomorphic function. We can pass to the covering space $\widetilde \cP \to \cP$ with deck transformation group given by 
    \begin{equation*}
        \frac{\pi_2(M, L_0 \cup L_1)}{\ker( \alpha \mapsto \langle [\Omega], \alpha \rangle) }
    \end{equation*}
    so that  the pullback of $\rd\cA_\C$ to $\widetilde \cP$ is the differential of a function $\cA_\C$. The simplest setting is the \emph{exact} setting, in which the deck transformation group is trivial. In this setting, we require that there is a holomorphic one-form $\Lambda$ such that $\Omega = \rd \Lambda$, and $\Lambda|_{L_i} = \rd H_i$ for some holomorphic functions $H_i: L_i \to \C$. In that case,
    \begin{equation}
    \label{eq:exact-holomorphic-symplectic-action-functional}
    \cA_\C(\gamma) = \int_\gamma \Lambda - H_1(\gamma(1)) + H_0(\gamma(0)). 
    \end{equation}
    The usual definition of the Fukaya--Seidel category using Lefschetz thimbles is not available in this infinite-dimensional setting. Instead, here $\mathrm{FS}(\cA_\C)$ is meant to be defined via an infinite-dimensional version of the complex Morse theory model of the Fukaya-Seidel category, which we review in the next section.
    \item In particular, the objects of $\mathrm{Fuet}_M(L_0, L_1)$, i.e. the \emph{1-morphisms} of $\mathrm{Fuet}_M$, are the critical points of $\cA_\C$, which are in bijection with the constant paths mapping to the intersection $L_0 \cap L_1$. When $L_0$ and $L_1$ intersect transversely, these critical points are nondegenerate, which is a requirement to make sense of the Fukaya--Seidel category. Heuristically, these objects correspond to infinite-dimensional Lefschetz thimbles of the holomorphic Morse function $\cA_\C$.
    \item The morphism cochain complex $\mathrm{Fuet}_M(L_0, L_1)(p_-, p_+)$, for $p_-, q_+ \in L_0 \cap L_1$ is generated by the gradient trajectories of $\Re(e^{-i\theta}\cA_\C)$, where  
    \begin{equation*}
        \theta = \arg(\cA_\C(p_-) - \cA_\C(p_+))
    \end{equation*}
    As in Lagrangian Floer theory, they correspond to $J_{\theta+\pi}$-holomorphic strips, or $J_\theta$-antiholomorphic strips, with $J_\theta$ denoting the rotated complex structure \eqref{Eq_RotatedComplexStructure}.  Explicitly, the cochain complex $\mathrm{Fuet}_M(L_0, L_1)(p_-, p_+)$ is freely generated over $\Z$ by the solutions to 
    \begin{gather}
    u: \R \times [0,1] \to M, \nonumber \\
    \label{eq:2-morphism-J-curve-equation}
        \partial_t u + J_{\theta+\pi}(u)\partial_\tau u = 0, 
    \end{gather}
    with boundary condition $u(t, i) \in L_i$ for $i=0,1$ and asymptotic boundary conditions $\lim_{t \to \pm \infty} u(t,\tau) = p_\pm$. As usual, we only consider solutions with finite Floer energy, which implies exponential converges with derivatives a $t \to \pm\infty$.
    \item The differential on  $\mathrm{Fuet}_M(L_0, L_1)(p_-, p_+)$ counts solutions to the complex gradient flow equation for $e^{-i\theta}\cA_\C$. This is an equation for maps $U \colon \R^2 \to \cP$, which we interpret as maps $U \colon [0,1]\times\R^2 \to M$. As we will see in \autoref{Subsec_FromMorseToFueter}, $U$ is formally a complex gradient trajectory if it solves the Fueter equation
    \begin{equation}
    \label{eq:sec2-fueter-eq}
    \begin{gathered}
        U: [0,1] \times \R^2 \to M, \\
        I(U) \partial_\tau U + J_{\theta}(U) \partial_s U + K_{\theta}(U) \partial_t U = 0, \\
        \end{gathered}
    \end{equation}
    and boundary conditions
    \begin{gather*}
        u(0, s, t) \in L_0, \quad u(1,s,t) \in L_1, \\
        \lim_{s \to \pm \infty} U(\tau, s, t) = u_\pm(t, \tau), \\
        \lim_{t \to \pm \infty} u(\tau, s, t) = p_\pm,
    \end{gather*}
    where again the convergence is assumed to be exponential. Here $J_\theta$ is as in \eqref{Eq_RotatedComplexStructure} and $K_\theta = I J_\theta$,  and $u_\pm$ are  solutions to \eqref{eq:2-morphism-J-curve-equation} corresponding to the gradient trajectories of $\Re(e^{-i\theta}\cA_\C)$. Writing $\mathcal{M}_3(u_-, u_+)$ for the moduli space of solutions to \eqref{eq:sec2-fueter-eq} and $\mathcal{M}_2(p_-, p_+)$ for the moduli space of solutions to \eqref{eq:2-morphism-J-curve-equation},  we have that the differential should formally be given by 
    \begin{equation}
    \label{eq:fueter-differential}
        \rd u_- = \sum_{u_+ \in \cM_2(p_-, p_+)} \#\mathcal{M}_3(u_-, u_+) u_+,
    \end{equation}
    where $\#$ denotes a signed count with respect to an orientation on  $\mathcal{M}_3(u_-, u_+)$ defined by means of index theory. 
    \item The compositions between $1$-morphisms \begin{equation*}
        \mathrm{Ob}(\mathrm{Fuet}_M(L_0, L_1)) \tensor \mathrm{Ob}(\mathrm{Fuet}_M(L_1, L_2)) \to \mathrm{Ob}(\mathrm{Fuet}_M(L_0, L_2))
    \end{equation*}
    are defined by counting $J_{\hat\theta}$-holomorphic triangles in \autoref{fig:fueter-domains-1}. Recall that $\hat\theta$ is the fixed background angle and does not depend on $L_0,L_1,L_2$; the role of $\hat\theta$ is discussed in the next section. 
\begin{figure}[H]
    \centering
    \includegraphics[width=10cm]{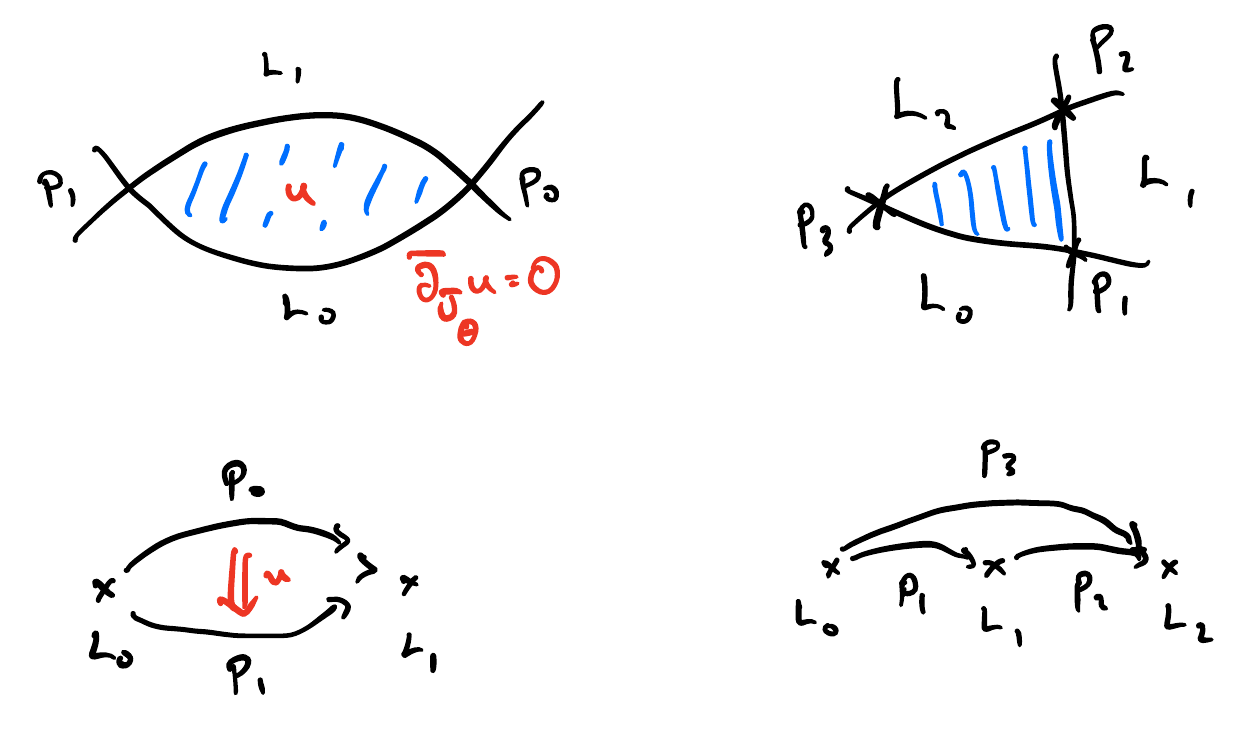}
    \caption{\small The Fueter $2$-category. Left: Objects are holomorpic Lagrangians, $1$-morpisms are intersection points, $2$-morphisms are $J_{\theta+\pi}$-holomorphic strips. Right: Composition of $1$-morphisms is given by the count of $J_{\hat{\theta}}$-holomorphic triangles.  }
    \label{fig:fueter-domains-1}
\end{figure}    
    \item The horizontal compositions and vertical compositions are given by counts of solutions to the Fueter equations with boundary conditions as in  \autoref{fig:fueter-domains-2} and \autoref{fig:fueter-domains-3}, respectively.
\end{itemize}

\begin{figure}[H]
    \centering
    \includegraphics[width=12cm]{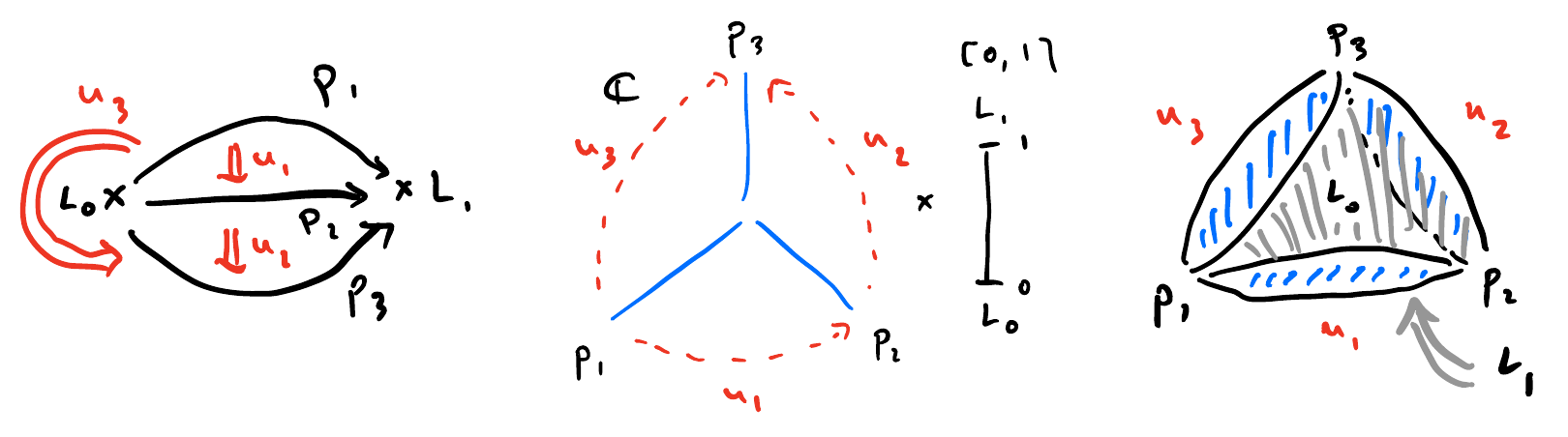}
    \caption{\small Vertical compostion in the Fueter $2$-category. Domain is $\C \times [0,1]$, boundary conditions as on the right. Grey denotes Lagrangian boundary conditions, blue denotes asymptotic pseudoholomorphic strips.  }
    \label{fig:fueter-domains-2}
\end{figure}

\begin{figure}[H]
    \centering
    \includegraphics[width=12cm]{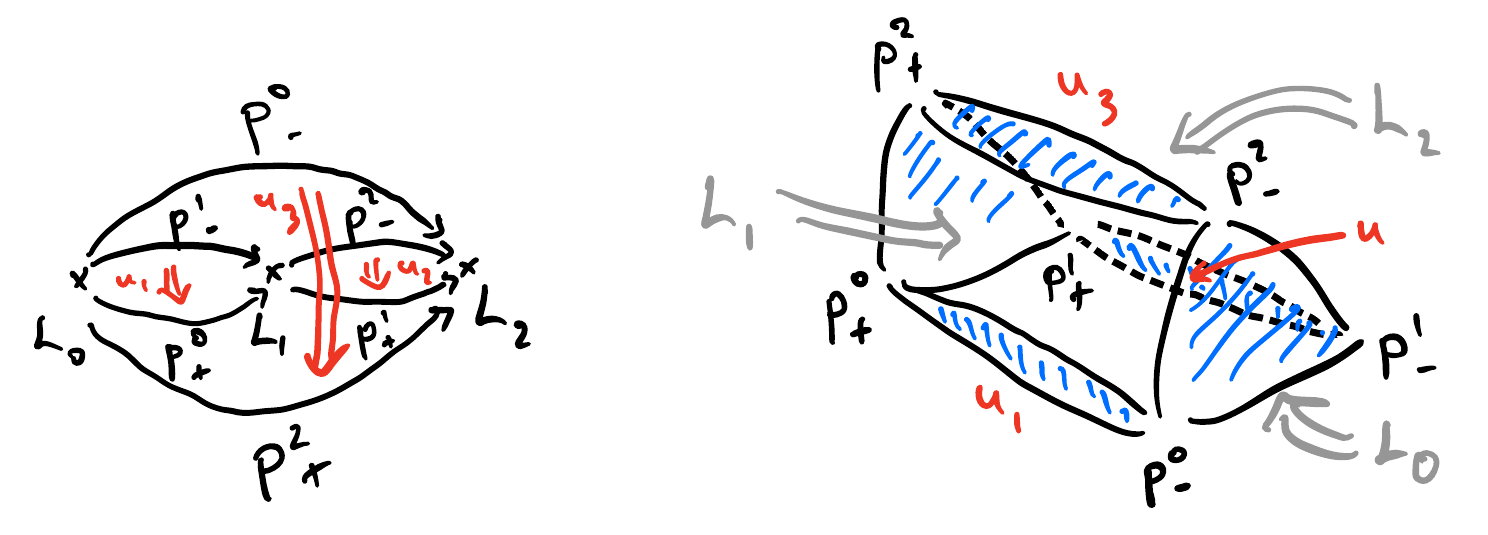}
    \caption{\small Horizontal composition in the Fueter $2$-category. Domain is $\R \times (\mathbb{D}^2 \setminus \{3 \text{ points}\})$. As before, Lagrangians are boundary conditions while pseudo-holomorphic triangles and strips are asymptotic conditions.  }
    \label{fig:fueter-domains-3}
\end{figure}

The most basic problems are then to develop the elementary theory of solutions to the Fueter equation with these types of boundary conditions. One expects the linearization of \eqref{eq:sec2-fueter-eq} to be a Fredholm operator admitting an index formula.
\begin{problem}
Develop the basic Fredholm theory of the Fueter equation  \eqref{eq:sec2-fueter-eq} with the above boundary conditions. Characterize algebro-topological data needed to assign orientations to the moduli spaces of solutions and gradings to $1$-morphisms in $\mathrm{Fuet}_M$, such that the index of the linearization of \eqref{eq:sec2-fueter-eq} at a solution $U$ is the difference of the gradings of $u_-$ and  $u_+$. 
\end{problem}

\begin{problem}
As usual in Floer homology, the above counts of solutions are only correct if the linearizations of the defining equations \eqref{eq:2-morphism-J-curve-equation}, \eqref{eq:sec2-fueter-eq} are surjective on solutions, and only if the resulting solutions are isolated. This may not occur for a fixed hyperk\"ahler triple, so a perturbation scheme, either virtual or geometric, is required to make sense of this prescription in general. Develop such a perturbation scheme.
\end{problem}

To define the compositions of $1$-morphisms as in  \autoref{fig:fueter-domains-2} and \autoref{fig:fueter-domains-3}, one must carefully specify the boundary conditions of the Fueter equation, and one must specify how the angle $\theta$ should vary over the domain because it is different for different strips. More generally, to show that composition is associative, one should compute the counts of solutions to the Fueter equation on domains with boundary conditions labeled by points in the associahedron, so that $\mathrm{Fuet}_M(L_0, L_1)$ is the cohomology category of an $A_\infty$-category. This category is meant to be the complex Morse theory model of the Fukaya-Seidel category of $\cA_\C$. The complex Morse theory model is necessary to make sense of the Fukaya--Seidel category in the infinite-dimensional setting. As such, the problem of the precise definition of the Fueter equation to be solved to define $\mathrm{Fuet}_M(L_0, L_1)$ is to be taken by analogy to a better-understood finite-dimensional construction, which we review below.

\subsection{Complex Morse theory}
\label{sec:zeta-instanton-model}

\paragraph{Fukaya--Seidel category.}

Let $F \colon X \to \C$ be a holomorphic Morse function on a Kähler manifold. We assume that $X$ is an exact symplectic manifold with a pseudo-convex end, and that $F$ has a controlled behavior at infinity. For example, $X$ could be a Stein manifold with the growth of $F$ controlled by an exhausting plurisubharmonic function. Seidel \cite{Seidel2008} defined an $A_\infty$-category $\mathrm{FS}(X,F)$, the Fukaya--Seidel category of the pair $(X,F)$. This category enlarges the Fukaya category of closed Lagrangian submanifolds of $X$ by adding in certain non-compact Lagrangians: the \emph{Lefschetz thimbles} $L_p$ associated to critical points $p$ of $F$. Up to exact Lagrangian isotopy, the thimbles are given by the unstable manifolds of the critical points of $F$ for the gradient flow of $\Re(F)$. It is a remarkable insight of Seidel that this enlargement of the Fukaya category of $X$ is in fact \emph{simpler} than the original category, in the sense that it can be computed entirely from the Floer cohomology of the thimbles themselves \cite{SeidelVanishingCycles}. 

\begin{figure}[H]
    \centering
    \includegraphics{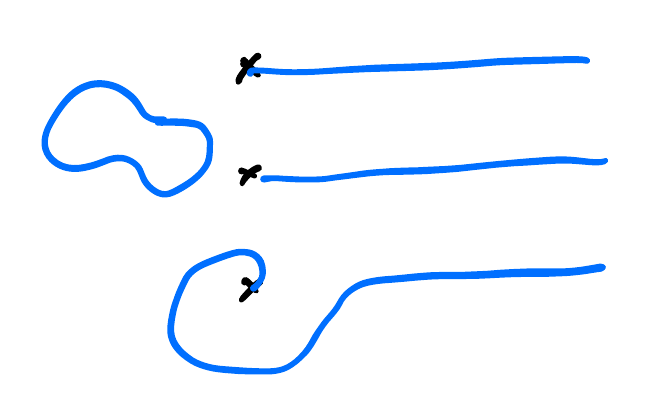}
    \caption{The Fukaya-Seidel category. Objects are closed Lagrangians and non-compact Lagrangians with prescribed asymptotic behavior. All objects can be computed from the distinguished objects corresponding to the Lefschetz thimbles.}
    \label{fig:Fukaya-Seidel-category}
\end{figure}

\paragraph{Morse theory construction.}

It has been proposed, simultaneously in the physics \cite{GMW} and mathematics \cite{Haydys2015} literature, to push this observation of Seidel further in order to construct an alternative model of $\mathrm{FS}(X,F)$, which we call the \emph{complex Morse theory model}, that never refers to any Lagrangian submanifolds of $X$. The motivation for this problem is the study of the Fukaya--Seidel categories of infinite-dimensional holomorphic functionals, especially the holomorphic Chern--Simons functional \cite{Witten-Analytic-Chern-Simons}. 

The basic problem with applying any of the usual models of the Fukaya--Seidel category to the infinite-dimensional setting is that working with infinite-dimensional Lagrangian submanifolds as boundary conditions for certain partial differential equations is typically not amenable to standard analytic methods. Instead, in the complex Morse theory model, one is expected to define the category $\mathrm{FS}(X,F)$ in terms of the gradient flows of $\Re(e^{-i \theta}F)$ and the solutions to the \emph{complex gradient equation}
\begin{equation}
    \label{eq:complex-gradient-sec-2}
    \begin{gathered}
        u: \R^2 \to X, \\
        \partial_s u + I(u) \partial_t u = \Grad \Im(e^{-i \theta} F)(u), 
    \end{gathered}
\end{equation}
where $(s,t)$ are the coordinates on $\R^2$, $I$ is the complex structure on $X$, and we are interested in solutions which converge exponentially to the real gradient trajectories of $\Re(e^{-i\theta} F)$ as $s \to \pm\infty$. Studying the complex gradient equation when $X$ is a certain space of connections then reproduces previous elliptic differential equations of interest, such as the Kapustin--Witten equation \cite{Haydys2015}.

The basic idea of the complex Morse theory model is that Lefschetz thimbles for $F$ are in bijection with the critical points of $F$. To define the thimble $L_p$ associated to a critical point $p \in \mathrm{Crit}(W)$, one fixes a \emph{background angle} $\hat{\theta} \in S^1$ which is not the difference of arguments of any pair of critical values of $F$. For simplicity, we will take $\hat \theta = \pi$. One then defines the thimble $L_p$ to be the unstable manifold of the gradient flow of $\Re(e^{- i\hat{\theta}}F)$. As such, a pair of Lefschetz thimbles $(L_p, L_q)$ will not intersect, as each thimble will project to parallel lines in the $F$ plane. Instead, one defines the intersections of $L_p$ and $L_q$  by rotating the angle at which $L_p$ flows out of $p$ slightly, or alternatively, flowing $L_p$ by the the Hamiltonian flow of a function which is zero on the interior of $X$ and agrees with $\Re(e^{- i\hat{\theta}}F)$ on the ends of $X$. As usual, in the Fukaya category these intersections are the generators of the cochain complex $\mathrm{FS}(X,F)(L_p, L_q)$. 

\begin{figure}[H]
    \centering
    \includegraphics{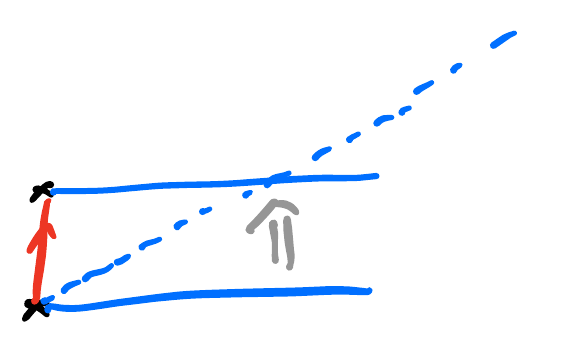}
    \caption{\small The complex Morse theory model $\mathrm{FS}(X,F)$ of the Fukaya-Seidel category does not involve Lagrangian submanifolds of $X$. The generating objects are taken to be critical points of $F$, and the morphisms are generated by gradient flows of $\Re(e^{-i \theta} F)$ between these critical points. In simple cases, these are naturally in bijection with intersection points between the thimbles, as shown above. }
    \label{fig:fukaya-seidel-instanton-comparison}
\end{figure}

In the complex Morse theory model, one requires a definition this complex which does not rely on intersecting Lagrangian submanifolds. To accomplish this, one uses the fact that given two critical points $p, q$, with no critical points $r$ such that 
\begin{equation*}
    \Im(F(p)) < \Im(F(r)) < \Im(F(q)),
\end{equation*} 
the intersection points between the Lefschetz thimbles $L_p, L_q$ of $p$ and $q$ are in bijection with the gradient flow lines of $-\Re(e^{-i \theta}F)$ from $p$ to $q$, where $\theta = \arg(F(p) - F(q))$; one sets these to be the generators of $\mathrm{FS}(X,F)(p, q)$ in the complex Morse theory model. 

\begin{remark}
\label{remark:background-angle-and-standard-vertical-position}
The picture to keep in mind is when $F(p)$ and $F(q)$ are on a vertical line in $\C=\R^2$, with $F(q)$ above $F(p)$. In that case, the gradient flow of $-\Re(e^{- i \theta}F)$ agrees with the Hamiltonian flow of $\Re(e^{-i \hat{\theta}}F)$ where the background angle $\hat{\theta} = \pi$ in our convention. See  \autoref{fig:fukaya-seidel-instanton-morphisms}.
\end{remark}

In the Fukaya--Seidel category, the differential on $\mathrm{FS}(X,F)(L_{p_-}, L_{p_+})$ is given by counting holomorphic strips. These intersections are the critical points of the symplectic action functional $\mathcal{A}_{L_{p_-}, L_{p_+}} : \mathcal{P}_{L_{p_-}, L_{p_+}}\to \R$ on the space of paths from $L_{p_-}$ to $L_{p_+}$, and the holomorphic strips are the gradient flows of this action functional. To write a corresponding holomorphic curve equation for the differential in the complex Morse theory model, one uses the observation that because $F$ is holomorphic, the gradient flows of $-\Re(e^{-i \theta}F)$ from ${p_-}$ to ${p_+}$ also Hamiltonian flows of $-\Im(e^{-i \theta}F)$. Thus, instead of the space of paths $[0,1]\to X$ from $L_{p_-}$ to $L_{p_+}$, one studies the space of \emph{infinite paths} with  exponential decay condition near the critical points:
\begin{equation}
    \mathcal{P}_{p_-, p_+} = \{ \gamma: \R \to X ; \lim_{t \to \pm\infty} \gamma(t) = p_\pm \text{ exponentially}\}
\end{equation}
and defines the symplectic action functional $\mathcal{A}_{p_-,p_+}$ perturbed by this Hamiltonian:
\begin{equation}
\label{eq:symplectic-action-functional-on-chords}
    \mathcal{A}_{p_-, p_+}(\gamma) = \int_\gamma \lambda + \int_\R \Im(e^{-i\theta} (F(\gamma(t)) - F(p_\pm) ) ) \rd t.
\end{equation}
where $\lambda$ is a primitive for the symplectic form $\omega$. 
Equation \eqref{eq:complex-gradient-sec-2} is then the formal gradient flow of $\mathcal{A}_{p_-, p_+}$, if we impose the conditions 
\begin{equation}
\label{eq:asymptotic-conditions-complex-gradient}
    \lim_{t \to \pm \infty} u(s,t) = p_\pm, \quad\text{and}\quad \lim_{s \to \pm \infty} = \gamma_\pm
\end{equation}
where $\gamma_\pm$ are gradient trajectories of $-\Re(e^{-i \theta F})$ from $p_-$ to $p_+$. 
The differential on $\mathrm{FS}(X,F)(p_-, p_+)$ is then given by counting solutions to \eqref{eq:complex-gradient-sec-2} with conditions \eqref{eq:asymptotic-conditions-complex-gradient}; writing $\mathcal{M}(p_-, p_+)$ for this space of solutions the differential is formally given by 
\begin{equation}
\label{eq:complex-morse-diferential}
    \rd \gamma_+ = \sum_{\gamma_- \in \mathrm{Crit}(\mathcal{A}_{p_-, p_+})} \# \mathcal{M}(\gamma_-, \gamma_+) \gamma_-. 
\end{equation}

\begin{remark}The equation that we have called the complex gradient flow equation recurs repeatedly in the literature under different names, as the Witten equation \cite{witten1993algebraic, Fan2013}, the $\zeta$-instanton equation \cite{GMW}, and, of course, it is the inhomogeneous pseudo-holomorphic map equation introduced by Floer. We suggest the moniker \emph{complex gradient flow} for \eqref{eq:complex-gradient-sec-2}, and the \emph{complex Morse theory model} for its application to Fukaya--Seidel categories, in order not to privilege any particular application of this basic differential equation. This terminology goes back to \cite{Donaldson1998}.
\end{remark}

Continuing onwards, the Fukaya--Seidel category has compositions given by counts of pseudo-holomorphic triangles, and more generally  $A_\infty$-operations  given by counts of pseudo-holomorphic polygons, with boundaries on the thimbles and corners on the intersections of the thimbles, see \autoref{fig:Fukaya-Seidel-morphisms}. 
\begin{figure}[H]
    \centering
    \includegraphics{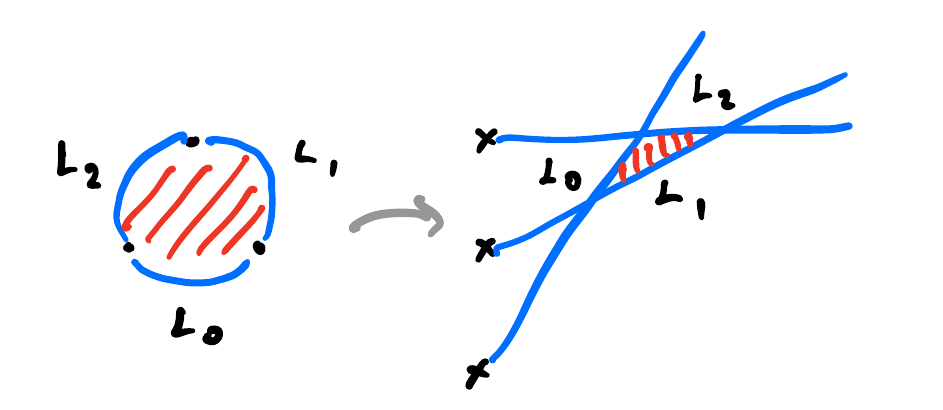}
    \caption{\small Composition in the Fukaya--Seidel category is given by counts of pseudoholomorphic triangles.}
    \label{fig:Fukaya-Seidel-morphisms}
\end{figure}

The basic corresponding ansatz in the complex Morse theory model is to count solutions to the complex gradient equation with asymptotics given by critical points and gradient flows of $F$, as in  \autoref{fig:fukaya-seidel-instanton-morphisms}. 

\begin{figure}[H]
    \centering
    \includegraphics{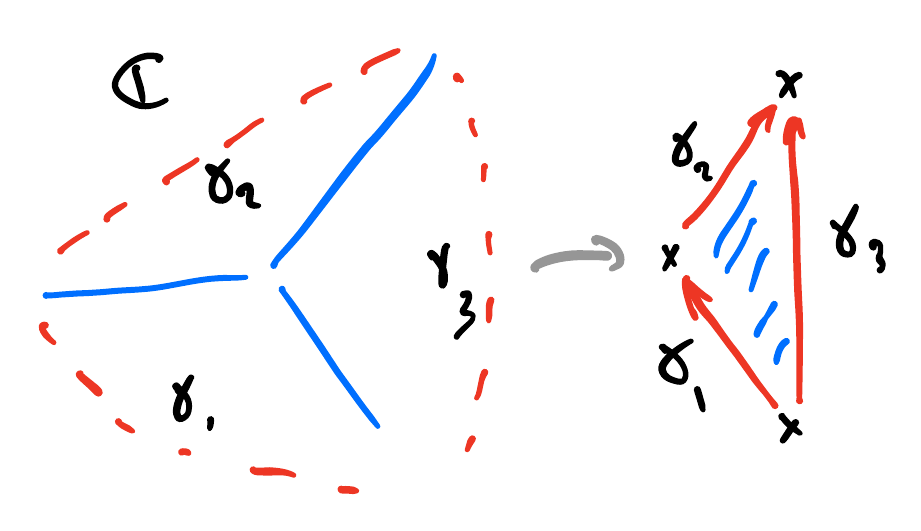}
    \caption{\small Composition in the complex Morse theory model of the Fukaya-Seidel category is given by solutions to the complex gradient equation \eqref{eq:complex-gradient-sec-2} with prescribed asymptotics. Formally, the above complex gradient trajectory corresponds to the pseudo-holomorphic triangle of \autoref{fig:Fukaya-Seidel-morphisms}.}
    \label{fig:fukaya-seidel-instanton-morphisms}
\end{figure}

\begin{figure}[H]
    \centering
    \includegraphics[width=11cm]{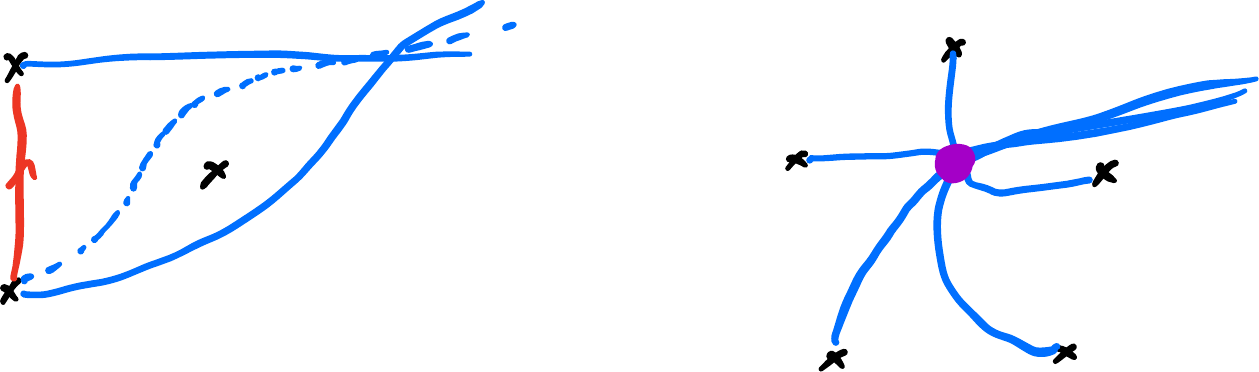}
    \caption{\small Left: Morphisms in the Fukaya--Seidel category cannot be identified with gradient flows between corresponding critical points when an intervening critical point exists. As such there is no natural identification between objects in the complex Morse model, defined naively, and thimbles in the Fukaya--Seidel category. Right: When critical points are in convex position, one can straight-line gradient flows with intersection points of thimbles which are not given by the standard unstable manifolds of $\Re(e^{i \hat{\theta}}F)$. }
    \label{fig:critical-point-problem}
\end{figure}

Beyond this, the existing approaches become less precise. The essential challenge is that while there is a natural bijection between the gradient flow lines from $p_-$ to $p_+$ and the intersection of the corresponding thimbles $(L_{p_-}, L_{p_+})$ when there are no critical values in between the images of the thimbles, this identification fails for general pairs of critical points. Instead, the number of gradient trajectories  from $p_-$ to $p_+$ is in bijection with the intersection points of a pair of $L_{p_-}$ and a different Lagrangian  $\tilde{L}_{p_+}$, which is related to $\tilde{L}_{p_+}$ by a series of \emph{Hurwitz moves}; see \autoref{fig:critical-point-problem}. Unfortunately, this Lagrangian $\tilde{L}_{p_+}$ depends on the choice of the background angle $\hat\theta$ in a strong sense: these Lagrangians are not isomorphic in $\mathrm{FS}(X,F)$ as one varies $\hat{\theta}$. Thus, the naive definition of the morphisms in the complex Morse theory model cannot be correct, as this precludes any natural map from the complex Morse theory model to the usual model of the Fukaya--Seidel category, which must send the critical points to well-defined objects of $\mathrm{FS}(X,F)$. 

There are two proposed approaches to resolve this problem, which are schematically compared in  \autoref{fig:haydys-witten}. The first, proposed by Haydys  \cite{Haydys2015} suggests to count broken Morse trajectories and broken complex gradient trajectories which first flow out from $p_-$ to a point with $-\Re(e^{-i\hat{\theta}}F)$ very large and then back to $p_-$, in order to define morphisms in the complex Morse theory model. As such, these have are always in bijection with the intersection points of the corresponding thimbles. This approach has recently been developed further by \cite{Wang2022}, and shown to produce a category equivalent to the Fukaya--Seidel category. Crucially, the combinatorics of the $A_\infty$ operations are the familiar ones from the Fukaya category.  However, this approach has the downside of introducing many nontrivial perturbations to the complex and real gradient flow equation, which should make computations more difficult, especially when applying the complex Morse theory model to the infinite-dimensional setting; see \autoref{remark:model-comparison}. 

The work of Gaiotto--Moore--Witten \cite{GMW} takes a different approach, and provides a combinatorial formula for defining an $A_\infty$-category in terms of standard, straight-line gradient trajectories between critical points together with solutions to the unperturbed complex gradient flow equation \eqref{eq:complex-gradient-sec-2}. The price of this advantageous geometric setup is a different proposed compactification of solutions to the complex gradient flow equation, which is not based on $A_\infty$ combinatorics but instead produces the $A_\infty$ category $\mathrm{FS}(X,F)$ via a highly involved process. A description of this formula in more traditional mathematical language was given in \cite{Kapranov2016}.

\begin{figure}[H]
    \centering
    \includegraphics[width=12cm]{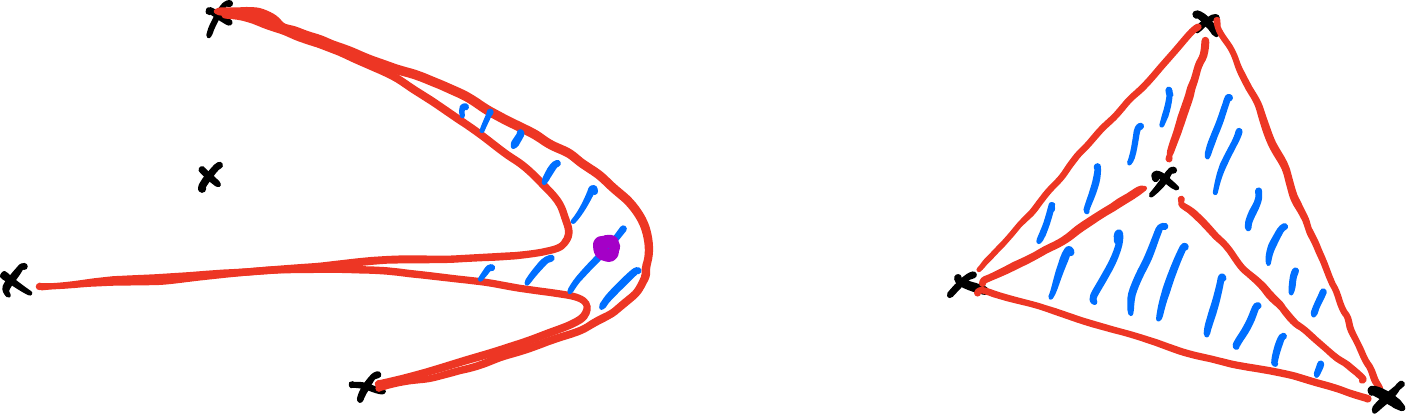}
    \caption{\small Two natural solutions to the problems posed by intervening critical points in \autoref{fig:critical-point-problem}). Left: Proposal of Haydys \cite{Haydys2015}, made rigorous by \cite{Wang2022}. Instead of straight-line gradient flows, one can use time-dependent gradient flows of $\Re(e^{-i \theta(t)}F)$. This simplifies many issues by may make computations in the infinite-dimensional context of \autoref{sec:fueter-tqft} significantly more difficult.  Right: Proposal of Gaiotto--Moore--Witten \cite{GMW}. An alternative compactification of the moduli space of solutions to the complex gradient flow equation should make it possible to use straight-line gradient flows, at the cost of novel analysis and novel algebraic constructions. }
    \label{fig:haydys-witten}
\end{figure}

Most aspects of the analytical setup underlying the Gaiotto--Moore--Witten (GMW) approach remain entirely open. In particular, the basic compactness result conjectured in \cite{GMW} has not been established.
\begin{problem}
\label{problem:gmw-compactification}
Define an analytical setup for the complex gradient flow equation and prove the correctness of the combinatorial compactifications of the spaces of solutions to \eqref{eq:complex-gradient-sec-2} as proposed in \cite{GMW}. 
\end{problem}

In order to make sense of this problem, the precise conditions and perturbation scheme for the complex gradient equation in the Gaiotto--Moore--Witten approach remain to be specified. 

\begin{remark}
Another nontrivial aspect of the compactness theory for the complex gradient flow equation is that the the energy identities of Haydys \cite[Section~2.4]{Haydys2015} (and the corresponding formulae in \cite{Kapranov2016}) are difficult to make sense of, as the values of the action functional, as written in these references, are \emph{infinite}. For the purpose of defining the differential, this can be fixed by renormalizing the action functional as in \eqref{eq:symplectic-action-functional-on-chords}. For defining higher $A_\infty$-operations, a nontrivial renormalization procedure must be performed in order to extract a finite energy identity controlling the $L^2$ energy of solutions to the complex gradient equation \cite{Wang2022}). 
\end{remark}

\begin{remark}
In the case where the critical values of $F$ are the vertices of a convex polytope, there is a natural identification between thimbles and critical points such that the intersection points of the thimbles are in bijection with the counts of gradient flows; see \autoref{fig:critical-point-problem}. 
It is a nontrivial fact that the combinatorics of Gaiotto--Moore--Witten then simplify dramatically in this context and reduce to the typical $A_\infty$ combinatorics of the Fukaya category; this is a consequence of the fact that the secondary polytope of a convex polytope is an associahedron, and can be extracted from \cite{Kapranov2016}. 
\end{remark}

\paragraph{Hochschild homology.}
There is a natural way to produce a vector space from an $A_\infty$-category $\mathcal{F}$, namely, its \emph{Hochschild homology} $HH_*(\mathcal{F})$. We will not go into this construction in detail, except to say that there is an explicit chain complex computing $HH_*(\mathcal{F})$ written in terms of the $A_\infty$-operations on $\mathcal{F}$ called the \emph{cyclic bar complex}, and that when $\mathcal{F}$ is a Fukaya category of a symplectic manifold $X$, there is an \emph{open-closed map} \cite{SeidelSymplecticHomology}
\begin{equation}
    \mathcal{OC}: HH_*(\mathcal{F}) \to HF^*(X,H)
\end{equation}
for some appropriate Hamiltonian $H: X \to \R$, which depends on the variant of the Fukaya category being considered. The Hamiltonian Floer cohomology group $HF^*(X,H)$ is object assigned to a circle by the open-closed field theory  \autoref{sec:aspects-of-tqft}) which assigns $\mathcal{F}(L_0, L_1)$ to the interval with endpoints labeled by $(L_0, L_1)$ , and the open-closed map is a chain-level version of all the cobordism maps associated to genus-zero cobordisms with one circle boundary and many interval boundaries. In many cases, including the case of the Fukaya--Seidel category, the open-closed map is known to be an isomorphism \cite{AbouzaidGeneration, GPS1}. 

Recall that to define the thimbles $L_p$ generating the Fukaya--Seidel category, we had to choose a background angle $\hat{\theta}$. Denote by $\mathrm{FS}_{\hat{\theta}}(X,F)$ the category defined for $\hat\theta$. The appropriate Hamiltonian to take for $\cF = \mathrm{FS}_{\hat{\theta}}(X,F)$  is $H = -\Re(e^{-i \hat{\theta}} F)$, or rather a small real rescaling of this Hamiltonian. As we take this Hamiltonian to zero (on some compact subdomain of $X$ containing all critical points, which arises in the technical setup for the Fukaya--Seidel category), $HF^*(X,H)$ is identified with the Morse homology group $HM_*(-\Re^{-i \hat \theta}F)$. This,  in turn, agrees with the homology of $X$ relative to the subset $\{x \in X | -\Re^{-i \hat \theta}F(x) \ll 0)$. For the conventional choice $\hat{\theta} = \pi$, this is simply $H_*(X, \{\Re(F) \gg 0\})$. 

In the construction of $\mathrm{FS}_{\hat{\theta}}(X,F)$ of $F$, which a priori depends on the choice of $\hat{\theta}$, one uses $\hat{\theta}$ to define an ordering $p_1 < \ldots < p_r$ on the critical points $\{p_1, \ldots, p_r\}$ of $F$. This is used to define first the \emph{directed Fukaya--Seidel category} $\mathrm{FS}^{\to}_{\hat{\theta}}(X,F)$, which only has morphisms from $L_{p_j}$ to $L_{p_k}$ if $j \leq k$, with the only endomorphism of $L_p$ being the identity map. 
The full category $\mathrm{FS}_{\hat{\theta}}(X,F)$ is then the category of \emph{twisted complexes} on $\mathrm{FS}^{\to}_{\hat{\theta}}(X,F)$. By algebraic properties of Hochschild homology, this implies that $\mathrm{FS}_{\hat{\theta}}(X,F)$  has the same Hochschild homology as $\mathrm{FS}_{\hat{\theta}}^\to(X,F)$, which, because of the directedness, is easily seen to be a vector space freely generated by $\{L_{p_1}, \ldots, L_{p_r}\}$.  Geometrically, these cycles give a basis for the relative homology group $H_*(X, \{\Re(F) \gg0\})$ (with corresponding modification when $\hat{\theta} \neq \pi$). When $F$ is a Morsification of a non-Morse holomorphic function $F_0$ with a single singular fiber, this would be the group of \emph{vanishing cycles} of the singular fiber.



\paragraph{Background angle and monodromy.}
As we see in the description above, the Fukaya--Seidel category $\mathrm{FS}_{\hat{\theta}}(X,F)$ a priori depends on the choice of an angle $\hat{\theta}$, and the directed subcategory $\mathrm{FS}^{\to}_{\hat{\theta}}(X,F)$ manifestly depends on $\hat{\theta}$ because $\hat{\theta}$ determines the ordering on the critical points. 

However, using the topological interpretation of the Hochschild homology group $HH_*(\mathrm{FS}_{\hat{\theta}}(X,F))$ as $H_*(X, \{ \Re(e^{-i\hat\theta} F) \gg 0 \})$, one sees that these groups form a local system of groups over the space of $\hat{\theta} \in S^1$. On the topological level, the dependence of the categories on $\hat{\theta}$ corresponds to the fact that the thimbles $L_{p_i}$, which implicitly depend on $\hat{\theta}$, do not define a trivialization of the local system of relative homology groups. Now, as $\hat{\theta}$ varies through angles avoiding the exceptional values 
\begin{equation*}
    \hat{\theta}^* = \arg(F(p)) - \arg(F(q))
\end{equation*}
for pairs $p, q \in \mathrm{Crit}(F)$, the ordering on the critical points does not change, the thimbles $L_p$ stays the same up to exact Lagrangian isotopy, and the bases of $HH_*(\mathrm{FS}_{\hat{\theta}}(X,F))$ given by the open-closed map form a trivialization of the local system.

At the exceptional angles, a theorem of Seidel \cite{SeidelVanishingCycles} tell us that the difference between the directed subcategories $\mathrm{FS}^{\to}_{\hat{\theta}^*-\epsilon}(X,F)$ and $\mathrm{FS}^{\to}_{\hat{\theta}^*+\epsilon}(X,F)$ for a small $\epsilon > 0$ is given by an algebraic operation called a \emph{mutation}, inducing an equivalence 
\begin{equation}
    \label{eq:mutation-equivalence}
    F_{\hat{\theta}^*}: \mathrm{FS}_{\hat{\theta}^*-\epsilon}(X,F) \to \mathrm{FS}_{\hat{\theta}^*+\epsilon}(X,F). 
\end{equation}

Thus, in the end, the full Fukaya--Seidel category $\mathrm{FS}_{\hat{\theta}}(X,F)$ up to equivalence does not depend on the choice of background angle $\hat{\theta}$. However, it is more precise to think of the categories $\mathrm{FS}_{\hat{\theta}}(X,F)$ as forming a \emph{local system of categories} over $S^1$. Indeed, the content of Picard--Lefschetz theory is that the bases of $H_*(X, \{ \Re(e^{-i\hat\theta} F) \gg 0 \})$ given by the thimbles $L_{p_i}$ are related at the exceptional angles $\hat{\theta}^*$ by an elementary integral change-of-basis matrix $\tau_{\hat{\theta}^*}$ referred to as a Picard--Lefschetz transformation. 
It turns out that the equivalences $F_{\hat{\theta}^*}$ categorify the Picard--Lefschetz transformations on relative homology in the sense that
\begin{equation}
\label{eq:picard-Lefschetz-identity}
    \mathcal{OC}(1_{L_p^{\hat{\theta}^* + \epsilon}}) = \tau_{\hat{\theta}^*}(\mathcal{OC}(1_{L_p^{\hat{\theta}^*-\epsilon}})). 
\end{equation}
where $1_O$ denotes the unit of the object $O$ and $L_p^{\hat{\theta}} \in \mathrm{FS}_{\hat{\theta}}$ denotes the thimble defined using the background angle $\hat{\theta}$. In other words, as we change the background angles, when the homology classes of thimbles change by a Picard-Lefshetz transformations, the corresponding objects change by a mutation -- the object associated to the new thimble is algebraically a \emph{cone} on a morphism from the previous thimble to a thimble associated to the end of an exceptional gradient trajectory at angle $\hat{\theta}^*$.  For many natural choices of function $F$, the total monodromy of the local system around $S^1$ is not trivial. On the level of Fukaya--Seidel categories, this corresponds to the unpublished theorem  \cite{Abouzaid-Ganatra} conjectured by Seidel \cite{SeidelSymplecticHomology} that the composition of the functors $F_{\hat{\theta}}$ is the Serre functor of $\mathrm{FS}_{\hat{\theta}}(X,F)$, which tends to be nontrivial as well.  

\subsection{The Fueter $2$-category from complex Morse theory}
\label{Subsec_FromMorseToFueter}

As explained in \autoref{sec:fueter-tqft}, the hom-category $\mathrm{Fuet}_M(L_0, L_1)$ in the Fueter $2$-category is meant to be based off of the complex Morse theory model of the holomorphic symplectic action functional $\cA_\C$. We now explain the correspondence, focusing on the exact setting as in \eqref{eq:exact-holomorphic-symplectic-action-functional}. By \eqref{eq:holomorphic-action-1-form} and the nondegeneracy of $\Omega$, the critical points of $\cA_\C$ are the constant paths with values in $L_0 \cap L_1$. We can view $M$ as a real symplectic manifold with respect to any of the symplectic forms
\begin{equation}
    \omega_\theta = \Re(e^{-i \theta} \Omega) = (\cos\theta)\omega_J + (\sin\theta)\omega_K. 
\end{equation}
For each $\theta \in S^1$, we have that $\omega_\theta$ is compatible with $J_\theta = \cos \theta J + \sin \theta K$. In the exact setting these forms have primitives given by $\lambda_\theta = \Re(e^{-i\theta} \Lambda)$, and $\lambda_{\theta}|_{L_i}$ has a primitive given by $H_i^\theta = \Re(e^{-i\theta}H_i)$. The holomorphic symplectic action functional $\cA_\C$ thus gives rise to a family of real symplectic action functionals 
\begin{equation}
    \mathcal{A}_\theta(\gamma) := \Re(e^{- i \theta} \cA_\C(\gamma)) = \int_\gamma \lambda_\theta - H_1^\theta(\gamma(1)) + H_0^\theta(\gamma(0)), 
\end{equation}
which have differentials
\begin{equation}
    \rd\mathcal{A}_\theta(\gamma) v  = Re(e^{- i \theta} \rd \cA_\C(\gamma) v) = \int_0^1 \omega_\theta \left(v(\tau), \frac{\rd\gamma}{\rd\tau} \right) \rd\tau. 
\end{equation}
In particular, $\mathcal{A}_J :=\mathcal{A}_0$ is the symplectic action functional when $M$ is thought of as symplectic with respect to  $\omega_J$, and $\mathcal{A}_K := \mathcal{A}_{\pi/2}$ is the symplectic action functional when $M$ is thought of as symplectic with respect to $\omega_K$. In these conventions, we can write $\cA_\C = \mathcal{A}_J + i \mathcal{A}_K$. 

By the standard computation of the Hessian of the real symplectic action functional \cite{floer-infinite-dimensional-morse-theory}, we see that the critical points of $\mathcal{A}^\C$ are nondegenerate exactly when $L_0$ intersects $L_1$ transversely. Moreover, in our conventions, summarized in \autoref{sec:conventions_lagrangian_floer_theory}, the gradient flow equation for  $\mathcal{A}_\theta$ is the $J_\theta$-holomorphic curve equation
\begin{equation}
\label{eq:symplectic-gradient-flow-section-2}
\begin{gathered}
    u: \R \to \mathcal{P}, \\
    \partial_t u = \Grad \mathcal{A}_\theta(u(t)) = - J_\theta (u)\partial_\tau u. 
\end{gathered} 
\end{equation}
In our conventions for the complex Morse theory model, the morphisms from $p$ to $q$ should be generated by gradient flows of $-\Re(e^{- i \theta}\cA_\C)$ where $\theta = \arg(p-q)$, i.e. gradient flows of $-\mathcal{A}_\theta = \mathcal{A}_{\theta + \pi}$, that is: $J_{\theta+\pi}$-holomorphic strips. 

Finally, the differential, and more generally the $A_\infty$-operations, in the complex Morse theory model should be given by solutions to the complex gradient flow equation \eqref{eq:complex-gradient-sec-2}. Plugging in $F= \cA_\C$ into \eqref{eq:complex-gradient-sec-2} and using \eqref{eq:symplectic-gradient-flow-section-2} to compute the gradient of the action functional, we get that complex gradient trajectories $U: \C \to \mathcal{P}$ are equivalent to solutions of 
\begin{gather*}
    U: \R^2 \times [0,1] \to M, \\
    \partial_s U + I(U)\partial_t U  = \Grad \Re(e^{-i (\theta+\pi/2)}\cA_\C) = - J_{\theta + \pi/2}(U) \partial_\tau U = - K_\theta(U) \partial_\tau U,
\end{gather*}
where 
\begin{equation}
    K_\theta = I J_\theta = \cos \theta K - \sin \theta J
\end{equation}
is the complex structure such that $(I, J_\theta, K_\theta)$ is a quaternionic triple. Moving all terms to the left hand side, and multiplying the equation by $J_\theta$ we get that the above solutions solve the Fueter-type equation 
\begin{equation}
    I(U) \partial_\tau U  + J_\theta (U) \partial_s U - K_\theta(U) \partial_t U = 0.
\end{equation}
In order to recover the more standard Fueter equation \eqref{Eq_FueterIntroduction}, we can instead define $\mathrm{FS}(\cP, \cA_\C)$ via solutions to the \emph{anti-complex gradient flow equation}
\begin{gather*}
    U: \C \to \mathcal{P}, \\
    \partial_s U - I (U) \partial_t U = \Grad \Im(-e^{- i \theta}\cA_\C).
\end{gather*}
Solutions to the above equation are in bijection with solutions to \eqref{eq:sec2-fueter-eq} as desired. The $s$-invariant asymptotics are the $J_\theta$-\emph{anti}-holomophic strips, i.e. the $J_{\theta+\pi}$-holomorphic strips, that give generate the morphism complexes $\mathrm{Fuet}_M(L_0, L_1)(p, q)$. The coordinate change $(s,t) \leftrightarrow (-s,t)$ induces a bijection between complex gradient trajectories and complex anti-gradient trajectories, and this explains why the opposing conventions in the differentials \eqref{eq:fueter-differential} and \eqref{eq:complex-morse-diferential}. We have thus reproduced the description of the Fueter hom-category of \autoref{sec:fueter-tqft}. We discuss further aspects of the formalism of the Fueter $2$-category in Appendix \ref{app:more-on-fueter}.

\subsection{Decategorifications of the Fueter $2$-category}
\label{Sec_RecoveringFukaya}
The Fukaya category of a symplectic manifold $M$ categorifies the topological invariants of $M$, such as the intersection form. In the same way, we expect that if $M$ is hyperkähler, the Fueter $2$-category of $M$ categorifies the symplectic invariants of $M$. Because $\mathrm{Fuet}(M)$ is a $2$-category it has two natural decategorifications: these turn out to be the Fukaya category $\mathrm{Fuk}(M)$, and the complex Morse category of the loop space $LM$, respectively.

As should be clear from  \autoref{sec:fueter-tqft}, the Fueter $2$-category takes into account the many nontrivial $J_\theta$-holomorphic strips stretching between pairs of complex Lagrangians. By analogy with the finite-dimensional case, discussed in \autoref{sec:zeta-instanton-model}, the Fueter category defined using background angle $\hat\theta$ should satisfy
\begin{equation}
\label{eq:decategorification-1}
    HH_*(\mathrm{Fuet}^{\hat\theta}_M(L_0, L_1))) \simeq   HM_*(\mathcal{A}_{\hat \theta + \pi}) = HF^*(L_0, L_1; \omega_{\hat{\theta}}, J_{\hat{\theta}}). 
\end{equation}
On other words, \emph{the Fueter $2$-category should categorify the subcategory of the Fukaya category containing complex Lagrangians}. In particular, \emph{assuming that all complex Lagrangians under consideration intersect transversely}, applying Hochschild homology to the $A_\infty$-bifunctors (see \cite{Seidel2008} for algebraic background),  
\begin{equation}
    \mathrm{Fuet}_M(L_0, L_1) \tensor \mathrm{Fuet}_M(L_1, L_2) \to \mathrm{Fuet}_M(L_0, L_2)
\end{equation}
should recover the composition in the Fukaya category. This identify follows from the definition of composition of $1$-morphisms in  \autoref{sec:fueter-tqft}. We symbolically write this naive decategorification of the Fukaya category as
\begin{equation}
\label{eq:naive-decategorification}
    HH_*(\mathrm{Fuet}_M^{\hat{\theta}}) \simeq \mathrm{Fuk}^h(M, \omega_\theta),
\end{equation}
where $\mathrm{Fuk}^h(M, \omega_\theta) \subset \mathrm{Fuk}(M, \omega_\theta)$ is the full subcategory with objects the complex Lagrangian submanifolds.

In the case when the holomorphic Lagrangians do not intersect transversely, we expect that the above identity is replaced with 
\begin{equation}
\label{eq:naive-decategorification-2}
    HP_*(\mathrm{Fuet}_M^{\hat{\theta}}) \simeq \mathrm{Fuk}^h(M, \omega_{\hat{\theta}})[u, u^{-1}]
\end{equation}
where $HP$ is the \emph{periodic cyclic homology} \cite{loday2013cyclic}, which we apply to every hom-category $Fuet_M^{\hat{\theta}}(L_0, L_1)$ individually, and on the right hand side we have tensored the Fukaya category with a ring $k[u, u^{-1}]$ where $u$ is a formal variable in degree $2$.

Let us motivate the expectation \eqref{eq:naive-decategorification}. The basic observation is that if two complex Lagrangian submanifolds intersect transversely, then the gradings of their intersection points are all the same.  As such, one has that
\begin{equation}
\label{eq:floer-homology-of-transverse-holomorphic-lagrangians}
    HF^*(L_0, L_1) \simeq \Z(L_0 \cap L_1) \text{ when } L_0 \text{ is transverse to }L_1.  
\end{equation}
In the case of non-transverse intersections, Solomon--Verbitsky have proven 
\begin{proposition}[{\cite{Solomon2019}}]
\label{prop:solomon-verbitsky}
For generic values of $\theta \in S^1$, all $J_\theta$-holomorphic polygons with boundary on a finite collection of complex Lagrangians are constant.
\end{proposition}
The basic reason for both of these phenomena is explained by the fact that the Floer cohomology of $L_0, L_1$ is the gradient flow of the real part of a holomorphic Morse function. Indeed, when we compute $HF^*(L_0, L_1, \omega_\theta)$, we are allowed to use any complex structure \emph{tamed} \cite[Chapter~1]{McDuff2012} by $\omega_\theta$, which is an open condition. Thus, instead of using $J_\theta$ to define the holomorphic curve equation, we can use $J_{\theta+\epsilon}$ for a small generic $\epsilon$ to compute $HF^*(L_0, L_1; \omega_\theta)$. The image of a $J_{\theta+\epsilon}$-holomorphic curve under $\cA_\C$ will project to a straight line with argument $\theta+\epsilon$ in the complex plane; as such, for generic $\epsilon$, there will be no such lines between a pair of critical values, and thus no such non-constant holomorphic curves between intersection points, and in the case of transverse intersection, and no differential on the resulting Floer complex. We discuss expectations around non-transversely intersecting holomorphic Lagrangians in \autoref{sec:more-on-decategorification}.

Thus, when $L_0$ and $L_1$ intersect \emph{cleanly}, the action functional $\mathcal{A}^\C_{L_0, L_1}$ is Morse--Bott, and so $FS(\mathcal{A}^\C_{L_0, L_1})$ should have a semi-orthogonal decomposition with pieces given by the Fukaya categories of the connected components of $L_0 \cap L_1$ (see \autoref{sec:morse-bott-action-functionals}, \eqref{eq:wrapped-complex-action-functional}). But \autoref{prop:solomon-verbitsky} implies in this case that $HF(L_0, L_1) = H^*(L_0 \cap L_1)$; meanwhile, known  properties of nondegenerate Fukaya categories \cite{AbouzaidGeneration,SeidelVanishingCycles, ganatra2019cyclic} should identify periodic cyclic homology of $FS(\mathcal{A}^\C_{L_0, L_1})$ with a $2$-periodic version of $H^*(L_0 \cap L_1)$. This motivates \eqref{eq:naive-decategorification-2} in the case of clean intersections; when the Lagrangians intersect transversely,  \eqref{eq:naive-decategorification-2} essentially reduces to \eqref{eq:naive-decategorification}.

Now, instead of applying Hochschild homology to morphisms, another way decategorifying a TQFT is to consider it on closed manifolds instead of on manifolds with boundary or corners. For example, as we saw in \autoref{sec:aspects-of-tqft}, an open-closed 2d TQFT assigns a vector space not just to an interval but also to a circle; in the symplectic setting, the circle corresponds to the Hamiltonian Floer cohomology of $M$, and is defined by computing the Morse homology of the action functional on the \emph{loop space} instead of on the space of paths. Algebraically, the Hamiltonian Floer cohomology of $M$ is meant to be recovered by taking the Hochschild homology of the Fukaya category.

In the setting of the Fueter equation, we can formally do the same, and formally define a $1$-category 
\begin{equation*}
        \mathrm{Fuet}_M(S^1) = \mathrm{FS}(\mathcal{A}^\C_{S^1})
\end{equation*}
where $\cA^\C_{S^1}$ is the holomorphic symplectic action functional on the loop space
\begin{equation}
\label{eq:loop-space-complex-functional}
\begin{gathered}
  \mathcal{A}^\C_{S^1}: LX \to \C, \\
  \rd\mathcal{A}^\C_{S^1}(\gamma) v = \int_{S^1} \Omega\left( v(\tau), \frac{\rd\gamma}{\rd \tau} \right)\rd \tau.
 \end{gathered}
\end{equation}
On a formal level, $\mathrm{Fuet}_M(S^1)$ has objects given by critical points of $\mathcal{A}^\C_{S^1}$, which are the constant loops into $X$. The $1$-morphisms would then be given by gradient flows of $\Re(e^{i \theta} \mathcal{A}^\C_{S^1})$ between these points, which are the \emph{$J_\theta$-holomorphic spheres in $X$}, for all $\theta \in S^1$.  Correspondingly, the the differential on the morphism complexes in this category should be given by counts of Fueter maps
\begin{equation}
\begin{gathered}
    U: \R \times S^2 \to M, \\ 
    \partial_s U - I(U)(\bar{\partial}^{J_\theta}_z U) = 0, 
\end{gathered}
\end{equation}
with $s$ denoting the coordinate on $\R$, $z$ denoting the holomorphic coordinate on $S^2$, and with asymptotics given by $J_\theta$-holomorphic spheres in $M$. Similarly, the higher $A_\infty$-operations should count solutions Fueter maps from non-compact $3$-manifolds without boundary with ends modeled on $\R \times S^2$.  

Unfortunately, $\mathcal{A}^\C_{S^1}$ is only Morse--Bott rather than Morse, so the complex Morse theory model of the Fukaya-Seidel category is not available; see \autoref{sec:morse-bott-action-functionals} for more on this problem. However, Ansatz \eqref{eq:A-side-knorrer} suggests that $Fuet_M(S^1)$ is a category built out of the Fukaya categories of the critical loci of $\mathcal{A}^\C_{S^1}$, e.g. 
\begin{equation}
    \label{eq:loop-space-category-simple}
    FS(e^{-i\hat{\theta}}\mathcal{A}^\C_{S^1}) \simeq \mathrm{Fuk}(M, \omega_{\hat{\theta}}).
\end{equation} 
Thus, the two decategorifications would actually be the same, although if we consider \emph{wrapped} Fueter categories (see \autoref{sec:morse-bott-action-functionals}, equation \eqref{eq:wrapped-complex-action-functional}), this equivalence of decategorifications may no longer hold. 

In any case, applying the finite dimensional relation between Hochschild homology and Morse homology again, we should expect that
\begin{equation}
\label{eq:hoschild-homology-of-closed-string-sector}
    HH_*(\mathrm{Fuet}_M(S^1)) \simeq HM_*(\Re(e^{- i \hat{\theta}}\mathcal{A}^\C_{S^1})) = HF^*(X, \omega_{\hat{\theta}}),
\end{equation}
with the right-hand side denoting the Hamiltonian Floer cohomology of $(X,\omega_{\hat\theta})$. We expect that there is a categorification of the isomorphism above to a map
\begin{equation}
\label{eq:categorical-open-closed-map}
    \overline{HH}_*(\mathrm{Fuet}_M) \to \mathrm{Fuet}_M(S^1). 
\end{equation}
which  (assuming that this map is an isomorphism) recovers \eqref{eq:hoschild-homology-of-closed-string-sector} by taking $HH_*$ of both sides \emph{again}. Here, $\overline{HH}_*$ is an invariant of a $2$-category described in  \autoref{sec:more-on-decategorification} (see \autoref{eq:categorical-hochscild-homology-def}) which is distinct from the $1$-category \eqref{eq:naive-decategorification} produced by taking Hochschild homology of all hom-categories. We note that an isomorphism like \eqref{eq:loop-space-category-simple} implies \eqref{eq:hoschild-homology-of-closed-string-sector} for nondegenerate symplectic manifolds $M$.

\subsection{Cotangent bundles of Kähler manifolds}
\label{sec:cotangent-bundles-conjecture}
As the previous discussions show, there are significant analytical and combinatorial challenges that arise in the definition of the Fueter $2$-category. We now come to the first examples where we have explicit expectations as to the behavior of the Fueter $2$-category. These examples are cotangent bundles of Kähler manifolds. As we explain in the introduction, the following conjecture is a quaternionic version of Floer's result. 

\begin{conjecture}
\label{conj:categorical-floer-theorem}
Let $X$ be a Kähler manifold with an exact symplectic form and pseudo-convex end (for example, a Stein manifold). Let $F: X \to \C$ be a holomorphic Morse function. Then, given certain geometric assumptions on $F$ near the ends of $X$, 
\begin{equation}
\mathrm{Fuet}_{T^*X}(L_0, L_1) \simeq \mathrm{FS}(X,F),
\end{equation}
where $L_0$ is the zero-section of $T^*X$ and $L_1 = \Gamma(\rd F)$ is the graph of $\rd F$.
\end{conjecture}

This conjecture is supported by our \autoref{Thm_FloerCorrespondence}. As a definition of the complex Morse theory model of $\mathrm{FS}(X,F)$ in the Haydys version is now available \cite{Wang2022}, this conjecture should be accessible by adapting that model to the infinite-dimensional setting and generalizing the analysis performed in the current paper. 

This conjecture, together with the existence of the Fueter $2$-category, would have remarkable implications. Indeed, just as Lagrangian Floer cohomology groups are invariant under Hamiltonian flows, the Fueter equation and the associated categorical invariants enjoy natural invariance properties under the flows of \emph{holomorphic Hamiltonians}. We explore this phenomenon further in \autoref{sec:further-analytic-aspects}; for now, it suffices to note that a holomorphic function 
\begin{equation*}
    H: M \to \C
\end{equation*}
determines a \emph{holomorphic Hamiltonian vector field} $X_H$ on $M$ by
\begin{equation}
\label{eq:holomorphic-hamiltonian-flow}
  \iota_{X_H}\Omega = -dH, 
\end{equation}
such that $X_H$ is the Hamiltonian vector field of $\Re(e^{-i \theta} H)$ with respect to $\omega_\theta$ for all $\theta \in S^1$ simultaneously. As such, given a pair of holomorphic functions $F_1, F_2: X \to \C$, we get a triple of Lagrangians $L_1 = \Gamma(\rd F_1)$, $L_2 = \Gamma(\rd F_2)$, and $L_{12} = \Gamma(\rd(F_1 + F_2))$. Writing $\pi: T^*X \to X$ for the projection, the time one holomorphic Hamiltonian flow of $H = -\pi^*F_1$ takes $L_1$ to the zero section $L_0$, and $L_{12}$ to $L_2$.  Thus, the existence of of the composition operation 
\begin{equation}
    \mathrm{Fuet}_{T^*X}(L_0, L_1) \tensor \mathrm{Fuet}_{T^*X}(L_1, L_{12}) \to \mathrm{Fuet}_{T^*X}(L_0, L_{12})
\end{equation}
implies, after applying the time one flow of $X_H$ to the second hom-category above and using  \autoref{conj:categorical-floer-theorem}, the surprising 
\begin{conjecture}
\label{conj:composition-in-Lefschetz-fibrations}
Let $X$ be as above and let $F_1, F_2: X \to \C$ be holomorphic Morse functions. Then, given certain geometric assumptions on $F_1$ and $F_2$ near the ends of $X$, there exists a monoidal functor 
\begin{equation}
    \mathrm{FS}(X, F_1) \tensor \mathrm{FS}(X, W_2) \to \mathrm{FS}(X, F_1 + F_2).
\end{equation}
Given a third holomorphic Morse function $F_3$, under appropriate assumptions on the geometry at infinity, the monoidal products for the pairs $(F_1, F_2)$, $(F_2, F_3)$, $(F_1+F_2, F_3)$, and $(F_1, F_2+F_3)$ should commute. 
\end{conjecture}

This is surprising, as there is no straightforward way to compose Lefschetz thimbles associated to multiple symplectic Lefschetz fibrations on the same symplectic manifold using purely symplectic methods. Searching for the definitions of such functors using symplectic geometry is a significant problem; we caution that the functors of \autoref{conj:composition-in-Lefschetz-fibrations} should only exist as stated under yet-uncertain tameness assumptions on the geometry of the $F_i$ at infinity, together with similarly uncertain geometric conditions ruling out the contribution of essential singularities of Fueter maps to the categorical structures. The most likely setting where such functors may exist are when $X$ is flat, as in that case all bubbling of Fueter maps can be ruled out a priori. It is possible that in the general case, the bubbling of Fueter maps contributes in an interesting way to the above algebraic structure. We explore a possible interpretation of  \autoref{conj:composition-in-Lefschetz-fibrations} in \autoref{sec:mirror-symmetry-for-toric-varieties} below.

\subsection{$3d$ mirror symmetry}
\label{sec:3d-mirror-symmetry}

This section is strongly influenced by our discussions with Justin Hilburn. Thus far, we have discussed the Fueter TQFT from a purely mathematical perspective, by complexifying and categorifying the Fukaya category of symplectic topology. A different way to arrive at the same story is through the perspective of quantum field theory. 

\paragraph{The $A$-type twist.} The Fueter equation can be seen to arise by by studying the $3d\;N=4\;\sigma$-model, which is a certain supersymmetric field theory where the classical fields are maps from $3$-manifolds to a hyperk\"ahler $4n$-manifold, augmented with supersymmetric data. This model is the dimensional reduction of the $4d\;N=2$ $\sigma$-model, which is a corresponding theory for maps from a $4$-manifold with a hyperk\"ahler target. The identification of the $4$-dimensional Fueter equation as the partial differential equation governing the $A$-type twist of the $4d\;N=2$ model dates back to \cite{anselmi-fre}; one can derive from that work that the $A$-type twist of the $3d\;N=4$ $\sigma$-model is governed by the Fueter equation \eqref{Eq_FueterIntroduction} , using the fact that the dimensional reduction of the $4d\;N=2$ model is the $3d\;N=4$ model. Dimensionally reducing further, we have that the reduction of the $3d\;N=4$ model to $2d$ is the $2d\;N=(2,2)$-model, the $\sigma$-model responsible for mirror symmetry of Calabi-Yau manifolds; this corresponds to the fact that the dimensional reduction of the Fueter equations is the pseudo-holomorphic curve equation.

The $3d\;N=4$ and $4d\;N=2$ theories are objects of active study. Most commonly in mathematics, they arise in their \emph{gauged variants}. Namely, there are natural partial differential equations on maps from a $3$ or $4$-manifold to a hyperk\"ahler manifold $M$ equipped with a hyperhamiltonian action \cite{hitchin1987self} of a compact Lie group $G$. We do not write down these equations in full generality, but we note that in the case of $M = T^*V$ for $V$ a complex representation of a compact Lie group, these equations reproduce many standard equations of gauge theory; for example, taking $V = \C$ and $M = U(1)$ recovers, after a minor modification, the usual Seiberg-Witten equations \cite[Remark~6.2]{nakajima-coloumb-1}. Thus, gauge theorists should think of this paper as studying the \emph{nonlinear, un-gauged variant of the Seiberg-Witten equations}. 

\paragraph{The $B$-type twist.}

All of the field theories mentioned above, besides admitting an $A$-type twist, also admit a $B$-type twist. Just as the $A$-type twist of the $3d$ $N=4$ $\sigma$-model has a $2$-category of boundary conditions $\mathrm{Fuet}_M$, which we explore in this paper, the $B$-type twist also has a $2$-category of boundary conditions, for which the initial attempts at a definition were given by Kapustin--Rozansky--Saulinas \cite{kapustin-rozansky-saulinas, kapustin-rozansky}. We thus call this $2$-category of boundary conditions for the $B$-type twist the \emph{KRS} $2$-category $KRS_M$. This category, just like the category $Coh_X$ of boundary conditions of the $B$-type twist of the $2D$ $N=(2,2)$ $\sigma$-model into $X$, is of purely algebraic nature. The original papers propose that for a cotangent bundle $T^*X$, the objects $2$-category $KRS_{T^*X}$ should contain the complex Lagrangians in $X$, just like the Fueter $2$-category. The model case for morphisms in the KRS $2$-category is the analog of  \autoref{conj:categorical-floer-theorem} for the $B$ model: when $L_0 \subset T^*X$ is the zero section and $L_1 \subset T^*X$ is the graph of the differential of $F:X \to \C$, one has 
\begin{equation}
    KRS_{T^*X}(L_0, L_1) \simeq D_{\Z/2}(X, F)
\end{equation}
where $D_{\Z/2}(X, F)$ is a small modification of the category $\mathrm{MF}(F)$ of \emph{matrix factorizations} \cite{Orlov04}, the $B$-twist analog of the Fukaya-Seidel category. . Similarly, in the KRS $2$-category, the analog of  \autoref{conj:composition-in-Lefschetz-fibrations} holds with Fukaya-Seidel categories replaced by the categories $D_{\Z/2}(X, F)$. 

More generally, $KRS_M$ should make sense for any complex symplectic $M$, and should contain complex Lagrangians as objects. The morphism categories between complex Lagrangians $KRS_{M}(L_0, L_1)$ should the categories of global sections of a sheaf of categories supported on $L_0 \cap L_1$ which near a point $p \in L_0 \cap L_1$ is meant to be a shift of the category $D_{\Z/2}(U, F)$ for some holomorphic $F: L_0 \supset U \to \C$, where $L_1$ is viewed as the graph of the differential of a function $F$ defined on a neighborhood $U$ of $p$.  Beyond this, Kapustin--Rozansky--Saulinas propose that in some sense, there should be more objects of the KRS $2$-category beyond the complex Lagrangians; in particular, a holomorphic generating function should specify an object in $KRS(T^*X)$, and more generally, certain sheaves of categories should define objects in the $KRS$ $2$-category. Partial foundational work for this idea can be found in \cite{gaitsgory-sheaves-of-categories, arinkin-singular-support}.

\paragraph{3d mirror symmetry.}

Most commonly, the $3d\;N=4$ theory is studied in its gauged linear variant via the geometry and representation theory associated to its \emph{Higgs and Coloumb branches}, which are certain non-compact complete hyperk\"ahler manifolds \cite{Huybrechts2011}. The $3d \;N=4$ theories are supposed to come in pairs, and the correspondence between paired theories goes by the name of \emph{3d mirror symmetry} \cite{Intriligator1996}. Given a pair of dual theories $\mathcal{X}, \mathcal{X}^\vee$, the correspondence is expected to pair
\begin{itemize}
    \item The Coloumb branch of $\mathcal{X}$ with the Higgs branch of $\mathcal{X}^\vee$
    \item The $A$-type twist of $\mathcal{X}$ with the $B$-type twist of $\mathcal{X}^\vee$. In particular, there should be an equivalence of $2$-categories of boundary conditions. 
\end{itemize}

In certain cases, dual pairs of $3d$ $N=4$ theories are associated to pairs of hyperk\"ahler manifolds, or stacks, $X, X^{\vee}$, and one expects that the Fueter $2$-category of $X$ is exchanged with the KRS $2$-category of $X^{\vee}$. This is the statement of 2-categorical 3d mirror symmetry. Some initial progress towards this idea has been achieved in \cite{gammage-hilburn-mazel-gee-2022}, who develop a conjectural sheaf theoretic model, by analogy to \cite{GPS1},  of a \emph{wrapped} variant of the Fueter $2$-category (which we discuss in Appendix \ref{sec:further-analytic-aspects}) and relate this to a KRS $2$-category of a 3d mirror geometry.

%% file: FueterEquation.tex
\section{The Fueter equation}
\label{Sec_FueterEquation}

In this section we introduce the Fueter equation, discuss the asymptotic boundary condtions for Fueter strips, and prove \autoref{Thm_EnergyIdentity} about the energy of Fueter strips.

\subsection{Quaternion-antilinear maps}
\label{Subsec_LinearAlgebra}

A map $f \colon \C \to \C$ is holomorphic if and only if its differential is complex-linear. The obvious quaternionic analog of this notion is not interesting as any map $f \colon \H \to \H$ with quaternionic-linear derivative is linear \cite{Fueter1935}. There is, however, an interesting theory of maps whose derivative is, in an appropriate sense, quaternion-antilinear.

We begin with some basic linear algebra over quaternions. Let $V$ be a quaternionic vector space, i.e. a real vector space equipped with a triple of complex structures $(I,J,K)$ satisfying the quaternionic relation $IJ = K$. This quaternionic triple makes $V$ into a module over $\H$. Given another quaternionic vector space $(U,i,j,k)$, the space of real homomorphisms from $U$ to $V$ decomposes into two $\H$-invariant subspaces:
\begin{equation}
  \label{Eq_DecompositionQuaternionLinearAntiLinear}
  \Hom(U,V) = \Hom_\H(U,V) \oplus \bHom_\H(U,V),
\end{equation}
where $\Hom_\H(U,V)$ is the space of $\H$-linear endomorphisms and $\bHom_\H(U,V)$ is the direct sum of the subspaces
\begin{gather*}
   \{ A \in \Hom(U,V) : A i = I A, \ A j = - JA, \ A k = - K A \}, \\
   \{ A \in \Hom(U,V) : A j = J A, \ A k = - KA, \ A i = - I A \}, \\
   \{ A \in \Hom(U,V) : A k = K A, \ A i = - IA, \ A j = - J A \}.
\end{gather*}
While none of the three subspaces is $\H$-invariant,  their direct sum is.  By analogy with complex vector spaces, we will call endomorphisms in $\bHom_\H(U,V)$ \emph{quaternion-antilinear} or \emph{$IJK$-antilinear}. The projection $\Hom(U,V) \to \Hom_\H(U,V)$ is given by
\begin{equation*}
  A \mapsto \frac14(A-IAi-JAj-KAk),
\end{equation*}
so $A$ is quaternion-antilinear if and only if
\begin{equation*}
 A-IAi-JAj-KAk = 0.
\end{equation*}

Our chief interest is in the case $U = \R\oplus \mathbb{E}$, where $\mathbb{E}$ is a three-dimensional vector space. A basis $(e_1,e_2,e_3)$ of $\mathbb{E}$ induces an identification $U = \H$ by
\begin{equation*}
    1\mapsto 1, \quad e_1 \mapsto i, \quad e_2 \mapsto j, \quad e_3 \mapsto k.
\end{equation*} 
Given such a basis, we say that a linear map in $\mathbb{E}\to V$ is quaternion-antilinear if the induced map $U \to V$ is quaternion-antilinear with respect to this identification. Explicitly, we have an isomorphism
\begin{gather*}
    \Hom_\H(U,V) \cong V \\
    A \mapsto A(1)
\end{gather*}
and $A \colon \mathbb{E} \to V$ is quaternion-antilinear if and only if it belongs to the kernel of the composition 
\begin{equation*}
    \Hom(\mathbb{E},V) \hookrightarrow \Hom(U,V) \to \Hom_\H(U,V) \cong V, \\
\end{equation*}
which we denote by $A \mapsto A_\H$ and which is given up to a constant by
\begin{equation*}
      A_\H = I A e_1 + J A e_2 + K A e_3.
\end{equation*}
Denote by $\overline{\Hom}_\H(\mathbb{E},V)$ the space of quaternion-antilinear maps. We have  
\begin{equation}
    \label{Eq_AntilinearDecomposition}
    \begin{split}
        \overline{\Hom}_\H(\mathbb{E},V) &=    \overline{\Hom}_I(\mathbb{E},V)\oplus\overline{\Hom}_J(\mathbb{E},V) \\
        &= \overline{\Hom}_J(\mathbb{E},V) \oplus \overline{\Hom}_K(\mathbb{E},V) \\
        &= \overline{\Hom}_K(\mathbb{E},V) \oplus   \overline{\Hom}_I(\mathbb{E},V),
    \end{split}
\end{equation}
where  $\overline{\Hom}_I(\mathbb{E},V)$ is the space of \emph{$I$-antilinear maps}, or, more precisely, $I$-antilinear maps from the orthogonal complement of $e_1$ in $\mathbb{E}$ to $V$ (and similarly for $J$ and $K$):
\begin{gather*}
    \overline{\Hom}_I(\mathbb{E},V) = \{ A \in \overline{\Hom}_\H(\mathbb{E},V) \ | \ Ae_1 = 0 \}, \\
      \overline{\Hom}_J(\mathbb{E},V) = \{ A \in \overline{\Hom}_\H(\mathbb{E},V) \ | \ Ae_2 = 0 \}, \\
        \overline{\Hom}_K(\mathbb{E},V) = \{ A \in \overline{\Hom}_\H(\mathbb{E},V) \ | \ Ae_3 = 0 \}.
\end{gather*}

\subsection{Fueter maps}
\label{Subsec_FueterMaps}

An \emph{almost quaternionic manifold} is a manifold $M$ of dimension $4n$ equipped with a triple of almost complex structures $(I,J,K)$ satisfying the quaternionic relation $IJ = K$; equivalently, with a reduction of the structure group of $TM$ from $\GL(4n,\R)$ to $\GL(n,\H)$. Since $\GL(n,\H)$ can be retracted to $\Sp(n) \subset \GL(n,\H)$, the structure group can be further reduced to $\Sp(n)$, yielding a Riemannian metric with respect to which $I,J,K$ are orthogonal. Such a metric is hyperkähler if and only if the two-forms induced by the metric and $I,J,K$ are closed, in which case it follows that $I,J,K$ are integrable, see \autoref{Lem_IntegrabilityCondition}.

A \emph{framed manifold} consists of a manifold together with a framing, i.e. trivialization of its tangent bundle. A framing induces an orientation and a Riemannian metric making the framing orthonormal. A classical theorem of topology says that every orientable three-manifold admits a framing. 

Let $(E,e_1,e_2,e_3)$ be a framed three-manifold and let $(M,I,J,K)$ be an almost quaternionic manifold. A map $U \colon E \to M$ is called a \emph{Fueter map} if for every $x\in E$, the derivative
\begin{equation*}
    \rd U(x) \colon T_x E \to T_{U(x)} M
\end{equation*}
is quaternion-antilinear. It is convenient to introduce the \emph{Fueter operator}
\begin{equation*}
    \sF(U) = I(U)\rd U(e_1) + J(U)\rd U(e_2) + K(U)\rd U(e_3),
\end{equation*}
which maps a function $U \colon E \to M$ to a section of $U^*TM$, so that $U$ is a Fueter map if and only if it satisfies the nonlinear first order partial differential equation called the \emph{Fueter equation}:
\begin{equation*}
    \sF(U) = 0.
\end{equation*}

\begin{example}
    If $E$ is a subset of $\R^3$, with the standard coordinates $(\tau,s,t)$, then the Fueter operator is
    \begin{equation*}
        \sF(U) = I(U) \partial_\tau U + J(U)\partial_s U + K(U) \partial_t U.
    \end{equation*}
    Observe that
    \begin{itemize}
        \item $\tau$-independent solutions are $I$-antiholomorphic maps from the $st$-plane,
        \item $s$-independent solutions are $J$-antiholomorphic maps from the $t\tau$-plane,
        \item $t$-independent solutions are $K$-antiholomorphic maps from the $\tau s$-plane.
    \end{itemize}
    In this paper, we are interested in the case
    \begin{equation*}
        E = [0,1] \times \R^2.
    \end{equation*}
\end{example}

\begin{remark}
  As in pseudo-holomorphic map theory, it is natural to consider a domain-dependent version of the Fueter equation. Let $E$ be a Riemannian three-manifold and suppose that we have a bundle map $\sI$ from the bundle $TE$ over $E\times M$ to $\End(TM)$  such that for every orthonormal frame $(e_1,e_2,e_3)$ of $TE$, the endomorphisms $\sI(e_1), \sI(e_2), \sI(e_3)$ form a quaternionic triple on $M$. The domain-dependent Fueter equation for maps $U \colon E \to M$ is
  \begin{equation*}
      \sum_{i=1}^3 \sI(e_i)\rd U(e_i) = 0
  \end{equation*}
  for any local orthonormal frame $e_1,e_2,e_3$ of $TE$. The equation does not depend on the choice of the local frame. For more details, see \cite{Hohloch2009, Haydys2012, Walpuski2017a}.
\end{remark}

Recall that the Cauchy--Riemann equation for antiholomorphic maps from a Riemann surface $u \colon (\Sigma,i) \to (M,I)$ can be succinctly written as 
\begin{equation*}
    \rd u \circ i + I \circ \rd u =0.
\end{equation*}
There is a similar way to write the Fueter equation, with the complex structure $i$ on $\Sigma$ replaced by the Hodge star (equivalently, the cross product) on the three-manifold $E$. Given a map $U \colon E \to M$, its differential $\rd U$ is a $TM$-valued one-form whose Hodge dual $\ast\rd U$ is a $TM$-valued two-form. Define an $\End(TM)$-valued one-form 
\begin{equation*}
    \cI = I \otimes e^1 + J \otimes e^2 + K \otimes e^3,
\end{equation*}
where $(e^1, e^2, e^3)$ is the dual frame of $T^*E$. Strictly speaking, $\cI$ is a one-form on $E\times M$, which, using $\id \times U \colon E \to E \times M$, we pull-back to get an element of $\Omega^1(E,U^*\End(TM))$. 

\begin{lemma}
  \label{Lem_FueterEqForms}
  The Fueter equation for $U \colon E \to M$ is equivalent to
\begin{equation}
    \ast \rd U + \cI \wedge \rd U =0,
\end{equation}
where the right-hand side combines the wedge product of differential forms with the action of $\End(TM)$ on $TM$. 
\end{lemma}
\begin{proof}
  Let $U_i = \rd U(e_i)$ and $(I_1,I_2,I_3) = (I,J,K)$, so that
  \begin{equation*}
        \label{Eq_FueterEquationForms} 
      \rd U = \sum_{i=1}^3 U_i \otimes e_i \quad\text{and}\quad \ast\rd U = \sum_{\mathrm{cyclic}} U_i \otimes e_j \wedge e_k.
  \end{equation*}
  We compute
  \begin{equation*}
      \cI \wedge \rd U = \sum_{j,k=1}^3 I_j U_k \otimes e_j \wedge e_k = \sum_{j < k}(I_j U_k - I_k U_j)\otimes e_j\wedge e_k.
  \end{equation*}
  Therefore, \eqref{Eq_FueterEquationForms} is equivalent to
  \begin{equation*}
      U_i + I_j U_k - I_k U_j = 0
  \end{equation*}
  for every cyclic permutation $(i,j,k)$ of $(1,2,3)$. After multiplying both sides by $I_i$ each of these equations is the Fueter equation $\sF(U)=0$.
\end{proof}

\subsection{Boundary conditions}
\label{Sec_BoundaryConditions}

Let $(M,I,J,K)$ be an almost quaternionic manifold and let $L_0,L_1 \subset M$ be two submanifolds which are totally real with respect to $J$ and $K$.  A \emph{Fueter strip} between $L_0$ and $L_1$ is a map
\begin{equation*}
    U \colon [0,1]\times \R^2 \to M
\end{equation*}
which satisfies the Fueter equation 
\begin{equation*}
    I(U)\partial_\tau U + J(U)\partial_s U + K(U)\partial_t U = 0,
\end{equation*}
together with the boundary conditions
\begin{equation}
    \label{Eq_LagrangianBoundaryConditions}
    U(0, s,t) \in L_0, \quad U(1, s,t) \in L_1 \quad\quad\text{for all } (s,t) \in \R^2,
\end{equation}
and the asymptotic boundary conditions
\begin{gather*}
    \lim_{s\to\pm\infty} U(\tau, s,t) = u_\pm(t,\tau) \quad\quad\text{for all } (t,\tau) \in \R\times[0,1], \\
    \lim_{t\to\pm\infty} U(\tau, s,t) = p_\pm \quad\quad\text{for all } (\tau,s) \in [0,1]\times\R.
\end{gather*}
Here $p_-$ and $p_+$ are points in $L_0\cap L_1$, and $u_-, u_+$ are $J$-antiholomorphic strips with boundary on $L_0$, $L_1$, that is: maps $u \colon \R\times[0,1] \to M$ satisfying
\begin{equation*}
    \partial_t u - J(u)\partial_\tau u =0,
\end{equation*} 
with boundary on $L_0$ and $L_1$ and converging to $p_\pm \in L_0\cap L_1$ as $t\to\pm\infty$. 

More precisely, throughout the paper we consider Fueter strips which satisfy the following conditions:
\begin{equation*}
    U \in C^2([0,1]\times\R^2,M), \quad \partial_s U(\tau,\cdot,\cdot) \in L^1(\R^2,U^*TM) \cap L^2(\R^2,U^*TM)
\end{equation*}
and
\begin{gather*}
   \lim_{s\to\pm\infty} U(\tau,s,t) = u_\pm(t,\tau) \quad\text{in } C^1(\R\times[0,1]), \\
   \lim_{t\to\pm\infty} U(\tau,s,t) = p_\pm \quad\text{in } C^1_\loc([0,1]\times\R).
\end{gather*}

\begin{remark}
    While the Fueter equation for maps from a closed, framed three-manifold is elliptic \cite{Hohloch2009}, a Fredholm theory for Fueter strips has not yet been developed and will be treated in subsequent papers. By analogy with complex gradient trajectories \cite{Haydys2015,Wang2022}, one expects that the technical assumptions listed above, as way as the exponential decay of Fueter strips along infinite ends, are automatically satisfied for Fueter strips of appropriate Sobolev class and with finite energy. Somewhat similar boundary value problems appear in the study of the Kapustin--Witten equations \cite{Mazzeo2020} and the Atiyah--Floer conjecture \cite{Daemi2021}.
\end{remark}

\subsection{Taming triples}
\label{sec:taming-triples}

The area of a pseudo-holomorphic disc in a symplectic manifold is bounded by a topological constant depending on the relative homology class of the disc and the cohomology class of the symplectic form. We wish to establish an analogous result for Fueter strips. For that purpose, it is convenient to introduce the notion of a \emph{taming triple}, which is a quaternionic analog of a two-form taming an almost complex structure. This notion is inspired by Donaldson and Segal's discussion of \emph{taming forms} in special holonomy geometry \cite{Donaldson2011}. 

We return to the linear algebra setup from  \autoref{Subsec_LinearAlgebra}. 
Let $(\mathbb{E},e_1,e_2,e_3)$ be a three-dimensional vector space with a basis and let $(V,I,J,K)$ be a quaternionic vector space. Let $(e^1,e^2,e^3)$ be the dual basis of $\mathbb{E}^*$. An element of $\mathbb{E}^*\otimes \Lambda^2 V^*$ can be seen as a triple of two-forms $\eta = (\eta_1,\eta_2,\eta_3)$ on $V$, or a three-form $\eta = \sum e^i \otimes \eta_i$ on $\mathbb{E}\oplus V$ using 
\begin{equation*}
    \mathbb{E}^* \otimes \Lambda^2 V^* \subset \Lambda^3(\mathbb{E}\oplus V)^*.
\end{equation*}

\begin{definition}
  \label{Def_TamingTriple}
  We say that $\eta \in \mathbb{E}^*\otimes\Lambda^2 V^*$ is an \emph{$IJK$-taming triple} if for every non-zero $IJK$-antilinear map $A \colon \mathbb{E} \to V$, the three-form  
\begin{equation*}
     A^*\eta := \sum_{i=1}^3 e^i \wedge A^*\eta_i   \in \Lambda^3 \mathbb{E}^*
\end{equation*} 
is a negative multiple of the volume form
\begin{equation*}
    \vol = e^1\wedge e^2\wedge e^3.
\end{equation*}
\end{definition}

Equip $\mathbb{E}$ and $V$ with Euclidean inner products making $(e_1,e_2,e_3)$ into an orthonormal basis and $I,J,K$ into orthogonal endomorphisms. We will often implicitly identify $\Lambda^3V^*$ with $\R$ using the volume form and write $\alpha \leq \beta$, and so on, for  $\alpha, \beta \in \Lambda^3 V^*.$

\begin{lemma} 
    \label{Lem_TamingInequality}
    A triple of two-forms $\eta\in \mathbb{E}^*\otimes\Lambda^2V^*$ is $IJK$-taming if and only if there is $c>0$ such that for every $IJK$-antilinear map $A \colon \mathbb{E}\to V$
    \begin{equation}
        \label{Eq_TamingInequality}
        |A|^2 \vol  \leq - c A^*\eta.
    \end{equation}
\end{lemma}

\begin{proof}   
    If the above inequality holds, then $\eta$ is clearly $IJK$-taming. Conversely, let $\eta$ be an $IJK$-taming triple. Denote by $F \subset \Hom(\mathbb{E},V)$  the set of $IJK$-antilinear maps with unit norm. Since $F$ is compact and the function 
    \begin{equation*}
        \eta \mapsto -A^*\eta  
    \end{equation*} 
    is positive on $F$, there is $c>0$ such that $-A^*\eta \geq 1/c$ for every $A \in F$. Since $A^*\eta$ and $|A|^2$ scale quadratically with $A$, the inequality follows.
\end{proof} 

The taming condition is open in the following sense. 

\begin{lemma}
If $\eta$ is an $IJK$-taming triple, then there are open neighborhoods of $\eta$ in $\mathbb{E}^*\otimes\Lambda^2V^*$ and of $(I,J,K)$ in the set of quaternionic triples on $V$ such that for every $\eta'$ and $(I',J',K')$ from these neighborhoods, $\eta'$ is an $I'J'K'$-taming triple.
\end{lemma}
\begin{proof}
  Let $\sQ$ be the space of quaternionic triples on $V$. Given  $(I,J,K) \in \sQ$, let 
  \begin{equation*}
      F(I,J,K) \subset \Hom(\mathbb{E},V)
  \end{equation*}
  be the space of $IJK$-antilinear maps $\mathbb{E} \to V$ with unit norm. The spaces $F(I,J,K)$ form a fiber bundle $\cF \to \sQ$ with compact fibers. A triple $\eta \in \mathbb{E}^*\otimes \Lambda^2 V^*$ is $IJK$-taming if and only if 
  \begin{equation*}
      - A^*\eta > 0 \quad\text{for all } A\in F(I,J,K).
  \end{equation*}
  Therefore, the lemma follows from the continuity of the map
  \begin{gather*}
      \cF \times (\mathbb{E}^*\otimes\Lambda^2 V^*) \to \R \\
      (A,\eta) \mapsto - A^*\eta.  \qedhere
  \end{gather*} 
\end{proof}

The following computation, which is an analogue of Wirtinger's inequality in K\"ahler geometry, shows that the set of $IJK$-taming triples is nonempty.

\begin{proposition}[\cite{Hohloch2009}]
  \label{Prop_TamingTriple}
  If $\omega_I,\omega_J,\omega_K \in \Lambda^2 V^*$ are the two-forms defined by $\omega_I(v,w) = \langle Iv, w\rangle$, and so on,  then $\eta = (\omega_I,\omega_J,\omega_K)$ is an $IJK$-taming triple. In fact, for every $A \in \Hom(\mathbb{E},V)$, 
  \begin{equation*}
      |A|^2 \vol = |A_\H|^2 \vol - 2 A^*\eta.
  \end{equation*} 
\end{proposition}

\begin{proof}
  Let $(I,J,K) = (I_1,I_2,I_3)$ and $a_i = A e_i$. We compute
  \begin{align*}
      |A_\H|^2 &= |I_1 a_1 + I_2 a_2 + I_3 a_3|^2 \\
      &= |a_1|^2 + |a_2|^2 + |a_3|^2 + \sum_{ij} \langle I_i a_i, I_j a_j \rangle \\
      &= |A|^2 + 2\sum_{\text{cyclic}} \langle I_i a_j , a_k \rangle  \\
      &= |A|^2 +  2\sum_{\text{cyclic}}\omega_i (a_j, a_k).
  \end{align*}
  Therefore,
  \begin{equation*}
      |A|^2 e^1 \wedge e^2 \wedge e^3 = |A_\H|^2 e^1 \wedge e^2 \wedge e^3 - 2 \sum_{i=1}^3 e^i \wedge A^*\omega_i. \qedhere 
  \end{equation*}
\end{proof}

Recall that $\omega \in \Lambda^2 V^*$ is said to tame  $I$ if $\omega(v,Iv) > 0$ for every non-zero $v\in V$. The space of quaternion-antilinear maps $\mathbb{E} \to V$ contains the space of $I$-antilinear maps from the orthogonal complement of $e_1$ in $\mathbb{E}$, and similarly for $J$ and $K$. This implies the following. 
 
\begin{proposition}
  \label{Prop_ComplexTaming}
  If $(\eta_1,\eta_2,\eta_3)$ is an $IJK$-taming triple, then 
  \begin{itemize}
      \item $\eta_1$ tames $I$,
      \item $\eta_2$ tames $J$,
      \item $\eta_3$ tames $K$.
  \end{itemize} 
\end{proposition}

\begin{remark}
  \label{Rem_ComplexTaming}
  The converse is not true. For example, for every $\lambda >0$ the triple $\eta_\lambda = ((2+\lambda) \omega_I,\omega_J,\omega_K)$ consists of two-forms taming $I,J,K$ respectively. In fact, this triple has the property that every complex structure from the twistor sphere
  \begin{equation*}
    aI + bJ + cK \quad\text{with } a^2+b^2+c^2 =1
  \end{equation*}
  is tamed by the corresponding linear combination of the two-forms from the triple $\eta_\lambda$. However, for $\lambda \gg 1$, this triple is not $IJK$-taming.
  To see this, observe that, by \autoref{Prop_TamingTriple},  for every quaternion-antilinear map $A \colon \mathbb{E} \to V$,
  \begin{equation*}
      - A^*\eta_\lambda = \left( |A|^2 + \lambda \omega_I(Ae_2, Ae_3) \right) \vol.
  \end{equation*}
  Choose $A$ so that $\omega_I(A e_2, A e_3) < 0$. This is possible since we can prescribe $A e_2$ and $A e_3$ and choose $A e_1$ so that $A$ is quaternion-antilinear. Given such $A$ and sufficiently large $\lambda$, the right-hand side is negative, which shows that $\eta_\lambda$ is not a taming triple.
\end{remark}
 
There is another notion of taming, adapted to the study of Fueter strips with cylindrical ends, introduced in \autoref{Subsec_FueterMaps}. In this approach, we break the symmetry between $I,J,K$, and distinguish one of the complex structures, say $J$.

\begin{definition}
  A \emph{$KI$-taming form} is a complex two-form $\eta \in \Lambda^2V^*\otimes\C$ such that
  \begin{itemize}
      \item $\eta$ is $J$-linear, or type $(2,0)$ with respect to $J$, in the sense that 
      \begin{equation*}
          \eta(Jv,w) = \eta(v,Jw) = i\eta(v,w),
      \end{equation*}
      \item $\Re(\eta)$ tames $K$,
      \item $\Im(\eta)$ tames $I$.
  \end{itemize}
  We define similarly $KI$-taming and $IJ$-taming forms. 
\end{definition}

\begin{remark}
  If $\eta$ is $J$-linear, then $\Re(\eta)$ tames $K$ if and only if $\Im(\eta)$ tames $I$.
\end{remark}

\begin{remark}
  Unlike the $IJK$-taming condition, the $KI$-taming condition is not open in the space of quaternionic triples because the requirement that $\eta$ is $J$-linear is a closed condition. However, if $J$ is fixed and $\eta$ is a $KI$-taming form, then it is also a $K'I'$-taming form for all nearby quaternionic triples $(I',J,K')$,
\end{remark}

\begin{definition}
For $A \in \Hom(\mathbb{E},V)$, define the \emph{$J$-norm} of $A$ by
\begin{equation*}
    |A|_J^2 = |Ae_2|^2 + |Ae_1 - JAe_3|^2
\end{equation*}
If $A$ is $IJK$-antilinear, this is the norm of the projection of $A$ on the orthogonal complement of the space of $J$-antilinear maps. We define similarly $|A|_K$ and $|A|_I$. 
\end{definition}

The following is an analog of \autoref{Lem_TamingInequality} for $KI$-taming forms.

\begin{lemma}
    A $J$-linear complex two-form $\eta$ is $KI$-taming if and only if there is $c>0$ such that for every $IJK$-antilinear map $A \colon \mathbb{E}\to V$ 
  \begin{equation*}
      |A|_J^2 \vol \leq - c A^*\eta_{KI},
  \end{equation*}
  where
  \begin{equation*}
      \eta_{KI} = \Re(\eta)\otimes e^3 + \Im(\eta)\otimes e^1 \in \Lambda^2V^*\otimes\mathbb{E}^*;
  \end{equation*}
  and similarly for $KI$- and $IJ$-taming pairs. 
\end{lemma}

\begin{proof}
  Set
  \begin{equation*}
      v = Ae_2 \quad\text{and}\quad w = Ae_3 - JAe_1,
  \end{equation*}
  so that the $IJK$-antilinearity condition is $v=Iw$. Therefore,
  \begin{equation*}
    |A|_J^2 = 2|v|^2
  \end{equation*}
  and, using that $\eta$ is $J$-linear, 
  \begin{equation*}
      -A^*\eta_{KI} = -\Re(\eta)(Ae_1,Ae_2) - \Im(\eta)(Ae_2,Ae_3) = - \Im(\eta)(v,w) = \Im(\eta)(v,Iv).
  \end{equation*}
  Therefore, $\Im(\eta)$ tames $I$ (equivalently, $\Re(\eta)$ tames $K$) if and only if for every $IJK$-antilinear map we have $|A|_J \vol \leq -c A^*\eta_{KI}$ for some $c>0$.
\end{proof}

The following is proved in the same way as \autoref{Prop_TamingTriple}. 

\begin{proposition}
  In the situation of \autoref{Prop_TamingTriple}, let
  \begin{align*}
     \eta_{JK} &= e^2\otimes\omega_J + e^3\otimes\omega_K, \\
     \eta_{KI} &= e^3\otimes\omega_K + e^1\otimes\omega_I, \\
     \eta_{IJ} &= e^1\otimes\omega_I + e^2\otimes\omega_J.
  \end{align*}
  For every $A \in \Hom(\mathbb{E},V)$, 
  \begin{align*}
      |A|_I^2 \vol = |A_\H|^2 \vol - A^*\eta_{JK}, \\
      |A|_J^2 \vol = |A_\H|^2 \vol - A^*\eta_{KI}, \\
      |A|_K^2 \vol = |A_\H|^2 \vol - A^*\eta_{IJ}.
  \end{align*}
  In particular,
  \begin{itemize}
      \item $\omega_J + i \omega_K$ is $JK$-taming,
     \item $\omega_K + i \omega_I$ is $KI$-taming,
     \item $\omega_I + i \omega_J$ is $IJ$-taming. 
  \end{itemize} 
\end{proposition}

\subsection{Energy bounds}
\label{sec:energy-bounds}

Let $(M,I,J,K)$ be an almost quaternionic manifold. Suppose that $\eta = (\eta_1,\eta_2,\eta_3)$ is an $IJK$-taming triple of two-forms on $M$, in the sense that inequality \eqref{Eq_TamingInequality} holds at every point of $M$ with the constant independent on the point. By definition, this means that for every subset $E \subset \R^3$ and a Fueter map $U \colon E \to M$ we have
\begin{equation}
\label{eq:fueter-energy}
    \int_E |\rd U|^2 \leq -c \int_E A^*\eta_1 \wedge \rd\tau + A^*\eta_ 2 \wedge \rd s + A^*\eta_3 \wedge \rd t.
\end{equation}
If $\eta_1,\eta_2,\eta_3$ are closed, then this inequality is an analog of the energy bound for $J$-holomorphic maps, for an almost complex structure $J$ tamed by a symplectic form. In particular, if $(M,I,J,K)$ is a hyperkähler manifold and $(\omega_I,\omega_J,\omega_K)$ is the triple of Kähler forms on $M$, so that we are in the situation of \autoref{Prop_TamingTriple}, this gives us the energy identity for Fueter maps proved in \cite{Hohloch2009}.

However, these integrals are typically infinite for Fueter strips asymptotic to $J$-antiholomorphic strips. This issue arises already in complex Morse theory: the $L^2$ norm of the derivative of a complex gradient trajectory \eqref{Eq_ComplexGradient} asymptotic to real gradient trajectories is infinite, and one has to consider instead the notion of energy known from Hamiltonian Floer theory. By analogy, we define the \emph{cylindrical energy} of a Fueter strip $U \colon [0,1]\times \R^2  \to M$ by
\begin{equation*}
    E(U) = \frac12 \int_{[0,1]\times\R^2} |\partial_s U|^2 + |\partial_t U - J(U) \partial_\tau U|^2 
\end{equation*}

The next result, which is a more precise statement of \autoref{Thm_EnergyIdentity} from the introduction, gives sufficient conditions for the energy of Fueter strips to be bounded by a topological constant. 

We begin with a general discussion of homology classes associated with maps from $\R^2$. 
Let $S \subset M$ be a submanifold or, more generally, the image of a smooth map from a manifold. Suppose that $f \in C^1(\R^2, M)$ is such that for every $\eta \in \Omega^2(M)$ the integral
\begin{equation*}
    \int_{\R^2} f^*\eta
\end{equation*}
is convergent (for example, $\partial_s f$ is in $L^1$), and $f$ is asymptotic to maps into $S$ as $t\to\pm\infty$ and $s\to\pm\infty$, in the following sense:
\begin{gather*}
    \lim_{s\to\pm\infty} f(s,\cdot) = a_\pm \quad\text{in } C^1(\R,M), \\
    \lim_{t\to\pm\infty} f(\cdot,t) = b_\pm \quad\text{in } C^1_\loc(\R,M),
\end{gather*}
where $a_\pm$ and $b_\pm$ are $C^1$ functions with values in $S$. These asymptotic conditions guarantee that if $\eta$ is closed and vanishes along $S$, this integral depends only on the relative cohomology class $[\eta] \in H^2(M,S)$ and the relative homotopy class of $f$ among maps satisfying these conditions. More precisely, two such functions $f_0$ and $f_1$ are homotopic relative $S$, if there is a map $F \in C^1([0,1]\times\R^2, M)$ such that the integral of $F_\tau^*\eta$ converges for every $\tau\in[0,1]$ and $\eta \in \Omega^2(M)$, on the boundary we have
\begin{equation*}
    F|_{\tau=0} = f_0 \quad\text{and}\quad F|_{\tau=1} = f_1,
\end{equation*}
and  
\begin{gather*}
    \lim_{s\to\pm\infty} F(\cdot, s,\cdot) = A_\pm \quad\text{in } C^1(\R\times[0,1],M), \\
    \lim_{t\to\pm\infty} F(\cdot, \cdot, t) = B_\pm \quad\text{in } C^1_\loc([0,1]\times \R,M),
\end{gather*}
where $A_\pm$ and $B_\pm$ are $C^1$ functions with values in $S$. The existence of such a homotopy and Stokes' theorem imply that
\begin{equation*}
    \int_{\R^2} f_0^* \eta = \int_{\R^2} f_1^*\eta 
\end{equation*}
for every $\eta$ representing a class in $H^2(M,S)$. Therefore, with a relative homotopy class of a function $f$ as above we associate a relative homology class
\begin{equation*}
    [f] \in H_2(M,S)
\end{equation*}
such that
\begin{equation*}
    \langle [\eta], [f] \rangle = \int_{\R^2} f^*\eta.
\end{equation*}

\begin{theorem}
  \label{Prop_EnergyBound}
  Let $(M,I,J,K)$ be an almost quaternionic manifold. Let $\eta$ be a closed, $KI$-taming complex two-form on $M$. If $L_0, L_1 \subset M$ are submanifolds such that $\Re(\eta)|_{L_i} = 0$, then there is a constant $C>0$ with the following property. For every Fueter strip $U \colon [0,1]\times \R^2 \to M$ with boundary on $L_0, L_1$ and asymptotic to $J$-antiholomorphic strips $u_\pm \colon \R\times[0,1]\to M$ and intersection points $p_\pm \in L_0 \cap L_1$, we have
  \begin{equation*}
      E(U) \leq -C \int_{\{\tau\}\times\R^2} U^*\Im(\eta),
  \end{equation*}
  where the integral on the right-hand side depends only on the relative classes
   \begin{equation*}
       [\Im(\eta)] \in H^2(M, u_+ \cup u_-) \quad\text{and}\quad [U |_{\{\tau\}\times\R^2}] \in H_2(M, u_+\cup u_-).
   \end{equation*}
   Moreover, if there is a one-form $\lambda$ on $M$ such that $\Im(\eta) = \rd\lambda$, then
   \begin{equation*}
       E(U) \leq C   \int_{\R\times\{\tau\}} u_-^*\lambda - u_+^*\lambda  . 
   \end{equation*}
\end{theorem}

\begin{proof} 
  Since $u = u_\pm$ is a $J$-antiholomorphic map and $\eta$ is $J$-linear,
  \begin{equation*}
    \Im(\eta)(\partial_t u, \partial_\tau u) = -\Re(\eta)(\partial_t u, J\partial_\tau u) = -\Re(\eta)(\partial_t u, \partial_t u) = 0.
  \end{equation*}
  Therefore, $\Im(\eta)$ vanishes on $u_+$, $u_-$ and defines a relative cohomology class
  \begin{equation*}
    [\Im(\eta)] \in H^2(M,u_+\cup u_-).
  \end{equation*}
  For $\tau\in[0,1]$ set $U_\tau(s,t) = U(\tau,s,t)$. The asymptotic boundary conditions imply that $U_\tau$ represents a relative homology class
  \begin{equation*}
      [U_\tau] \in H_2(M,u_+\cup u_-),
  \end{equation*}
  which is independent of $\tau$. It follows from the general discussion above that the integral 
  \begin{equation*}
      \langle [\Im(\eta)], [U_\tau] \rangle = \int_{\{\tau\}\times\R^2} U^*\Im(\eta) 
  \end{equation*}
  depends only on the relative classes $[\Im(\tau)]$ and $[U_\tau]$. 
  To prove the energy bound, we use the $KI$-taming condition: 
  \begin{align*}
      |\partial_s U|^2 + |\partial_t U - J\partial_\tau U|^2 \leq - C (U^*\Re(\eta) \wedge \rd t + U^*\Im(\eta) \wedge \rd  \tau).
  \end{align*}
  Therefore,
  \begin{equation}
    \label{Eq_EnergyTamingInequality}
      E(U) \leq -C \int_{[0,1]\times \R^2} U^*\Re(\eta)\wedge\rd t + U^*\Im(\eta) \wedge \rd\tau.
  \end{equation}
  Since $u_\pm$ are $J$-antiholomorphic and $\Re(\eta)$ vanishes on $L_0$ and $L_1$,
  \begin{equation*}
      u_\pm^*\Re(\eta) = 0  \quad\text{and}\quad \Re(\eta) |_{L_i} = 0. 
  \end{equation*}
   Therefore, $\Re(\eta)$ represents a relative cohomology class 
  \begin{equation*}
      [\Re(\eta)] \in H^2(M,L_0\cup L_1 \cup u_+ \cup u_-).
  \end{equation*}
  For $t\in\R$ the restriction $V_t(\tau,s) = U(\tau,s,t)$ represents a relative homology class
  \begin{equation*}
      [V_t]  \in H_2(M,L_0\cup L_1 \cup u_+ \cup u_-),
  \end{equation*}
  which is independent of $t \in \R$. Since $V_t$ converges to a constant map as $t\to\infty$, this homology class is trivial. Therefore,
  \begin{equation*}
      \int_{[0,1]\times\R^2} U^*\Re(\eta) \wedge \rd t = \int_\R \langle [\Re(\eta)], [V_t] \rangle \rd t = 0.
  \end{equation*}
  The second integral in \eqref{Eq_EnergyTamingInequality} is
  \begin{equation*}
      \int_{[0,1]\times\R^2} U^*\Im(\eta) \wedge \rd \tau = \int_0^1 \langle [\Im(\eta) ], [U_\tau] \rangle \rd\tau = \int_0^1 \langle [\Im(\eta)], [U_0] \rangle \rd\tau = \langle [\Im(\eta)], [U_0] \rangle,
    \end{equation*}
    which proves the energy bound. The last statement follows from Stokes' theorem. 
\end{proof}

The proof of \autoref{Prop_EnergyBound} relies crucially on the assumption that $M$ admits a closed $KI$-taming form $\eta$, which is rather restrictive. Since $\eta$ is a non-degenerate closed form of type $(2,0)$ with respect to $J$, it follows that the almost complex manifold $(M,J)$ carries a holomorphic volume form $\eta^n$ where $n = \dim_\C M$. This, in turn, implies that $J$ must be integrable, as shown by from the following well-known lemma.

\begin{lemma}
  \label{Lem_IntegrabilityCondition}
  Let $(M,J)$ be an almost complex manifold and let $n = \dim_\C M$. If $M$ admits a closed,  nowhere vanishing complex  $n$-form of type $(n,0)$ with respect to $J$, then $J$ is integrable.
\end{lemma}
\begin{proof}
  Let $T_{0,1}M \subset TM\otimes\C$ be the subbundle of vectors of type $(0,1)$. If $\Omega$ is a nowhere-vanishing form of type $(n,0)$, then \begin{equation*}
      T_{0,1}M = \{ v \in TM\otimes\C \ | \ v \lrcorner \Omega = 0 \}.
  \end{equation*}
  If $v\lrcorner \Omega = 0$ and $w \lrcorner\Omega = 0$, then the formula for $\rd\Omega=0$ in terms of contractions and Lie brackets gives us that $[v,w]\lrcorner \Omega = 0$, and so $[v,w] \in T_{0,1}M$. It follows from the Frobenius theorem that the distribution $T_{0,1}M$ is integrable. 
\end{proof}

%% file: Convexity.tex
\section{Convexity theory}
\label{Sec_Convexity}

In this section, we introduce a quaternionic analog of the notion of pseudo-convexity known from almost complex geometry. We then use it to prove \autoref{Thm_MaximumPrinciple}: a maximum principle for Fueter maps to convex almost quaternionic manifolds. 

\begin{remark}
    Other notions of convexity in quaternionic geometry were introduced by Alesker--Verbitsky \cite{Alesker2006} and Harvey--Lawson \cite{Harvey2009}. Our definition is different and adapted specifically to the study of Fueter maps. 
\end{remark}

\subsection{Complex convexity and pseudo-holomorphic maps}

 We begin with a review of convexity theory for almost complex manifolds. For more details, see, for example, \cite[Chapter 2]{Cieliebak2012}.  Let $(M,J)$ be an almost complex manifold. The \emph{Levi form} of a function $\rho \colon M \to \R$ with respect to $J$ is the two-form 
\begin{equation}
    \label{Eq_LeviForm}
    \sigma_\rho^J =  -\rd(\rd\rho\circ J).
\end{equation}
The function is \emph{$J$-convex} if $\sigma_\rho^J$ tames $J$, or, equivalently, if for $x \in M$ and every non-zero $J$-antilinear map $a \colon \C \to T_xM$, we have 
\begin{equation*}
    a^*\sigma_\rho^J < 0
\end{equation*}
in the sense that the pull-back is a negative multiple of the volume form on $\C$. This characterization will be useful when we discuss quaternionic convexity. From the viewpoint of pseudo-holomorphic curve theory, the importance of $J$-convexity lies in the following maximum principle.

\begin{proposition}
    \label{Prop_MaximumPrinciple}
   Let $\rho \colon M \to \R$ be an $J$-convex function. If $u \colon \Sigma \to M$ is an $J$-holomorphic or $J$-antiholomorphic map from a Riemann surface, then 
   \begin{equation*}
   \rho\circ u \colon \Sigma \to \R
   \end{equation*}
   is a subharmonic function. More precisely, $\Delta(\rho\circ u) \geq 0$, with equality at those points where $\rd u =0$. In particular, if $\Sigma$ is a compact Riemann surface with boundary, then
   \begin{equation*}
       \sup_\Sigma (\rho\circ u) = \sup_{\partial\Sigma} (\rho\circ u).
   \end{equation*}
\end{proposition}
\begin{proof}
  Let $i$ be the complex structure on $\Sigma$ and set $\sigma = \sigma_\rho^J$. If $u$ is $J$-holomorphic, then
  \begin{align*}
      \Delta(\rho\circ u)\vol_\Sigma &= -\rd(\rd(\rho\circ u)\circ i) = -\rd u^*(\rd\rho\circ J) \\
      &= -u^*\rd(\rd\rho\circ J) = u^*\sigma. 
  \end{align*}
  Let $s+it$ be a local holomorphic coordinate on $\Sigma$. Since $u$ is $J$-holomorphic, 
  \begin{equation*}
      u^*\sigma  =  \sigma(\partial_s u,\partial_t u) \rd s \wedge \rd t=  \sigma(\partial_s u, J\partial_s u) \rd s \wedge \rd t \geq c|\rd u|^2\vol_\Sigma ,
  \end{equation*}
  so that
  \begin{equation*}
      \Delta(\rho\circ u) \geq c |\rd u|^2,
  \end{equation*}
  which completes the proof. The antiholomorphic case is proved in the same way. 
\end{proof}

There is a variant of the maximum principle for pseudo-holomorphic maps with boundary on a conical submanifolds. 

\begin{definition}
    \label{Def_Conical}
    ~
    \begin{enumerate}
        \item
        Let $(M,J)$ be a non-compact almost complex manifold. We say that $M$ has an \emph{$J$-conical end} if there is an exhausting function $\rho \colon M \to \R$ which, outside a compact subset, is $J$-convex and has no critical points. 
        \item
        If $M$ has an $J$-conical end as above, we say that a submanifold $L \subset M$ is $J$-\emph{conical}  if, outside a compact subset, the restriction of $\rho$ to $L$ has no critical points and $\rd\rho$ vanishes on $J(TL)$.
    \end{enumerate}
\end{definition}

\begin{proposition}
  \label{Prop_MaximumPrincipleConical}
  Let $(M,J)$ be an almost complex manifold with an $J$-conical end defined by a function $\rho \colon M \to \R$ and let $L \subset M$ be an $J$-conical submanifold. For every sufficiently large $c$, the subset $K_c = \{ \rho \leq c \}$ has the following property. If $u \colon \Sigma \to M$ is an $J$-holomorphic from a Riemann surface with boundary such that
  \begin{itemize}
      \item $u^{-1}(M\setminus K_c)$ is compact,
      \item $u(\Sigma) \cap K_c \neq \emptyset$,
      \item $u(\partial\Sigma) \subset L$,
  \end{itemize}
  then $u(\Sigma) \subset K_c$.
\end{proposition}

\begin{proof}
  Let $c$ be large enough so that $\rho$ and $L$ satisfy the conditions from \autoref{Def_Conical} outside $K_c$. Without loss of generality, we may assume that $c$ is a regular value of $r = \rho\circ u$. Set $U = r^{-1}([c,\infty))$ and assume that $U$ is non-empty and compact. Since $r|_U$ is subharmonic, it has to attain its maximum on the boundary
  \begin{equation*}
      \partial U = \partial_r U \cup \partial_\Sigma U,
  \end{equation*}
  where
  \begin{equation*}
      \partial_r U = U \cap r^{-1}(c) \quad\text{and}\quad \partial_\Sigma U = U \cap \partial\Sigma. 
  \end{equation*}
  Suppose that $r$ is not constant and attains its maximum at $z \in \partial_\Sigma U$. By the Hopf lemma, if $n \in T_z \Sigma$ is the vector normal to the boundary of $\Sigma$ at $z$,  we have
  \begin{equation*}
      \rd r(n) > 0.
  \end{equation*}
  On the other hand, since $u$ is $J$-holomorphic,
  \begin{equation*}
      \rd r (n) = \rd \rho (\rd u(n)) = - \rd\rho(J\rd u(in)).
  \end{equation*}
  The right-hand side is zero because $in$ is tangent to $\partial\Sigma$, so $\rd u(in)$ is tangent to $L$, and $\rd\rho$ vanishes on $J(TL)$. We conclude that $z$ must lie in $\partial_r U$ and $r \leq c$ on $U$. By definition, $r \leq c$ on $\Sigma\setminus U$, so $r \leq c$ on all of $\Sigma$.
\end{proof}

\subsection{Quaternionic convexity and Fueter maps}

Let $(M,I,J,K)$ be an almost quaternionic manifold. Denote by $(\tau,s,t)$ the standard coordinates on $\R^3$.

\begin{definition}
\label{def:ijk-convex}
   The \emph{$IJK$-Levi form} of a function $\rho \colon M \to \R$ is the three-form on $\R^3\times M$ defined by
   \begin{equation*}
       \sigma_\rho = \sigma_\rho^I \wedge \rd\tau + \sigma_\rho^J \wedge \rd s + \sigma_\rho^K \wedge \rd t.
   \end{equation*}
   We say that $\rho$ is \emph{$IJK$-convex} if
   \begin{equation*}
       A^*\sigma_\rho < 0
   \end{equation*}
   for every $x \in M$ and every non-zero quaternion-antilinear map $A \colon \R^3 \to T_x M$. Here we pull back $\sigma_\rho$ using the graph of $A$ which is a map $\R^3 \to \R^3\oplus T_x M$, and by the above inequality we mean that the pull-back is a negative multiple of the volume form on $\R^3$. 
   In other words, $\rho$ is $IJK$-convex if at every point  $\sigma_\rho$ is an $IJK$-taming triple in the sense of \autoref{Def_TamingTriple}.
\end{definition}
By \eqref{Eq_LeviForm}, the $IJK$-Levi form can be written as
\begin{equation*}
    \sigma_\rho = -\rd( \rd\rho\circ \cI),
\end{equation*}
where
\begin{equation*}
    \cI = I \otimes \rd \tau + J \otimes \rd s + K \otimes \rd t
\end{equation*}
is an $\End(TM)$-valued one-form on $\R^3\times M$, and $\circ$ denotes the composition of the wedge product with the action of $\End(TM)$ on $TM$. By \autoref{Prop_ComplexTaming}, every $IJK$-convex function is $I$-convex, $J$-convex, and $K$-convex. However, the converse is not true, see \autoref{Rem_ComplexTaming}.
The following is a quaternionic analog of \autoref{Prop_MaximumPrinciple}.

\begin{proposition}
    \label{Prop_FueterMaximumPrinciple}
    If $\rho \colon M \to \R$ is $IJK$-convex and $U \colon E \to M$ is a Fueter map from a domain $E \subset \R^3$, then
    \begin{equation*}
        \rho\circ U \colon E \to \R
    \end{equation*}
    is a subharmonic function. More precisely,  $\Delta(\rho\circ U) \geq 0$,with equality at those points where $\rd U = 0$. In particular, if $E$ is compact and has smooth boundary, then
    \begin{equation*}
        \sup_E (\rho\circ U) = \sup_{\partial E} (\rho\circ U).
    \end{equation*}
\end{proposition}
\begin{proof}
    Let $r = \rho\circ U$ and denote by $\ast$ the Hodge start on $E$. By \autoref{Lem_FueterEqForms}, 
    \begin{align*}
        \ast \rd r &= \ast( \rd \rho\circ \rd U) = \rd\rho \circ (\ast \rd U ) \\
        &= -\rd\rho \circ ( \cI \wedge \rd U) =  U^*(\rd\rho\circ\cI).
    \end{align*}
    The minus sign in the second line comes from anticommuting $\wedge$. Using \autoref{Lem_TamingInequality} and $IJK$-convexity, we arrive at 
    \begin{equation*}
        \Delta r \, \vol_E = \rd(\ast\rd r) =  U^*\rd(\rd\rho\circ\cI) = - U^*\sigma_\rho \geq c |\rd U|^2 \vol_E. \qedhere
    \end{equation*}
\end{proof}

In order to obtain a maximum principle for Fueter maps with Lagrangian boundary conditions, we introduce the notion of an $IJK$-conical end.

\begin{definition}
    \label{Def_QuaternionConical}
   ~
   \begin{enumerate}
       \item A non-compact almost quaternionic manifold $(M,I,J,K)$ has an \emph{$IJK$-conical end} if there is an exhausting function $\rho \colon M \to \R$ which, outside a compact subset, is $IJK$-convex and has no critical points.
       \item If $M$ has an $IJK$-conical end as above, a submanifold $L \subset M$ is \emph{$JK$-conical}  if, outside a compat subset, $\rd\rho$ is non-zero on $TL$ and vanishes on $J(TL)$ and $K(TL)$.
   \end{enumerate}
\end{definition}

\begin{theorem}
    Let $(M,I,J,K)$ be an almost quaternionic manifold with an $IJK$-conical end defined by a function $\rho \colon M \to \R$. Let $L \subset M$ be a $JK$-conical submanifold. For every sufficiently large $c$, the subset $K_c = \{ \rho \leq c \}$ has the following property. If $E \subset \R^3$ is a domain with smooth boundary and $U \colon E \to M$ is a Fueter map such that
    \begin{itemize}
        \item $U^{-1}(M\setminus K_c)$ is compact,
        \item $U(E) \cap K_c \neq \emptyset$,
        \item $U(\partial E) \subset L$,
    \end{itemize}
    then $U(E) \subset K_c$.
\end{theorem}
\begin{proof}
    This is an obvious generalization of \autoref{Prop_MaximumPrincipleConical} and the proof is similar. Let $c$ be sufficiently large so that $\rho$ is   $L$ satisfies the conditions from \autoref{Def_QuaternionConical} outside $K_c$. Set $r = \rho\circ U$ and  $E_c = r^{-1}([c,\infty))$. By \autoref{Prop_FueterMaximumPrinciple}, $r$ is subharmonic on $E_c$. Suppose that it attains a maximum at a point $x \in \partial E_c$ such that $r(x) > c$. If $r$ is not constant, then the Hopf lemma implies that
    \begin{equation*}
        \rd r(n) > 0,
    \end{equation*}
    where $n = \pm\partial_\tau$ is a unit vector normal to the boundary of $[0,1]\times \R^2$. However,
    \begin{equation*}
        \pm \rd r(n) = \rd\rho( \partial_\tau U(x)) =   \rd\rho( K\partial_s U(x) - J\partial_t U(x)) = 0,
    \end{equation*}
    since $\partial_s U, \partial_t U$ are tangent to $L$ and $L$ is $JK$-conical.
\end{proof}

Applied to Fueter strips, i.e. solutions to the Fueter equation with boundary conditions from \autoref{Sec_BoundaryConditions}, the above theorem implies \autoref{Thm_MaximumPrinciple} stated in the introduction.

\subsection{Conical completions}

Conical manifolds and conical submanifolds can be constructed by completing compact manifolds with convex boundary. We begin with the complex case.  

\begin{definition}
~
\begin{enumerate}
    \item 
    Let $M$ be a compact manifold with boundary. A \emph{boundary-defining function} is a smooth function $\rho \colon M \to \R$ such that $\rho < 0$ in the interior of $M$ and $\partial M = \rho^{-1}(0)$ is a regular level set. If $J$ is an almost complex structure on $M$, we say that $\partial M$ is $J$-convex if there exists a boundary-defining function which is $J$-convex in a neighborhood of $\partial M$.
    \item In the situation above, we say that a submanifold $L \subset M$ has an \emph{$J$-conical boundary} if $\partial L = L \cap \partial M$ is a transverse intersection and $L$ is $J$-conical in a neighborhood of $\partial L$.
\end{enumerate}
\end{definition}
 
\begin{proposition}
  \label{Prop_ComplexCompletion}
  Let $(M,J)$ be a compact almost complex manifold and let $\rho \colon M \to \R$ be a boundary-defining function which is $J$-convex near $\partial M$. The almost complex structure $J$ can be extended to an almost complex structure on
  \begin{equation*}
      \widehat M = M \cup_{\partial M} \partial M\times [0,\infty)
  \end{equation*}
  in such a way that the function
  \begin{gather*}
      \partial M\times [0,\infty) \to \R \\
      \rho(x,r) = e^r-1
  \end{gather*} 
  is an $J$-convex extension of $\rho$. Moreover, any submanifold $L \subset M$ with an $J$-conical boundary can be extended to a submanifold
  \begin{equation*}
      \widehat L = L\cup_{\partial L} \partial L \times [0,\infty) \subset \widehat M
  \end{equation*}
  which is $J$-conical with respect to $\rho$ in $\partial M\times [0,\infty)$.
\end{proposition}

\begin{proof}
  Let $\rho \colon M \to \R$ be an $J$-convex boundary-defining function. Let $g$ be a Riemannian metric compatible with $J$ and such that $|\nabla \rho|=1$ in a neighborhood of $\partial M$ and $\nabla\rho$ is tangent to $L$. The existence of such a metric follows from the following observation.  Around every point of $L$ there is a vector field $v$ tangent to $L$ and such that $\rd \rho(v)=1$. Since $\rd\rho$ vanishes on $J(TL)$, $Jv$ is nowhere tangent to $L$. Therefore, we can locally construct $g$ compatible with $J$ and such that $|v|=1$, $\nabla\rho = v$ and $J\nabla\rho = Jv$ is orthogonal to $TL$. Using partitions of unity, we patch together these local constructions to define $g$ in a neighborhood of $L$ and extend to all of $M$. Given such $g$, the flow of $\nabla \rho$ gives us a diffeomorphism of a neighborhood of of $\partial M$ with
  \begin{equation*}
      \partial M \times (-\epsilon,0]
  \end{equation*}
  such that $\rho$ agrees with $(x,r) \mapsto e^r-1$,  $\nabla\rho = e^{-r}\partial_r$, and $L$ agrees with $\partial L \cap (-\epsilon,0]$.    Extend $J$ to $\partial M\times [0,\infty)$ so that it is $r$-invariant.  Define $\lambda \in \Omega^1(\partial M)$ to be the restriction of $\rd\rho\circ J$ to $\partial M$. By construction, $\rd r \circ J = \lambda$, so that on $\partial M \times[0,\infty)$ we have
  \begin{equation*}
      \rd\rho\circ J = e^r\rd r\circ J = e^r\lambda = e^r (\rd\rho\circ J ) |_{r=0} .
  \end{equation*}
  and the Levi form of $\rho$ is
  \begin{equation*}
    \rd(\rd\rho\circ J) = \rd(e^r\lambda) = e^r( \rd r \wedge \lambda + \rd \lambda) = e^r (\rd(\rd\rho\circ J))|_{r=0}.
  \end{equation*} 
  At $r=0$, $L$ satisfies the cylindrical condition and the Levi form of $\rho$ tames $J$. It follows from the above formulae that the same is true for any $r>0$.
\end{proof}

Next, we discuss a quaternionic analog of the above construction.

\begin{definition}
~
\begin{enumerate}
    \item 
    Let $(M,I,J,K)$ be a compact almost quaternionic manifold with boundary. We say that $\partial M$ is $IJK$-convex if there exists a boundary-defining function which is $IJK$-convex in a neighborhood of $\partial M$.
    \item In the above situation,  we say that a submanifold $L \subset M$ has an \emph{$JK$-conical boundary} if $\partial L = L \cap \partial M$ is a transverse intersection and $L$ is $JK$-conical in a neighborhood of $\partial M$.
\end{enumerate}
\end{definition}
 
\begin{proposition}
  Let $(M,I,J,K)$ be a compact almost complex manifold and let $\rho \colon M \to \R$ be a boundary-defining function which is $IJK$-convex near $\partial M$. The almost quaternionic structure $(I,J,K)$ can be extended to 
  \begin{equation*}
      \widehat M = M \cup_{\partial M} \partial M\times [0,\infty)
  \end{equation*}
  in such a way that the function
  \begin{gather*}
      \partial M\times [0,\infty) \to \R \\
      \rho(x,r) = e^r-1
  \end{gather*} 
  is an $IJK$-convex extension of $\rho$. Moreover, any submanifold $L \subset M$ with a $JK$-conical boundary can be extended to a submanifold
  \begin{equation*}
      \widehat L = L\cup_{\partial L} \partial L \times [0,\infty) \subset \widehat M
  \end{equation*}
  which is $JK$-conical with respect to $\rho$ in $\partial M\times [0,\infty)$ and, if $L$ is $I$-complex, so is $\widehat L$. 
\end{proposition}

\begin{proof}
    As in the proof of \autoref{Prop_ComplexCompletion}, we can choose a Riemannian metric on $g$ which is compatible with $I$, $J$, $K$, and such that in a neighborhood of $\partial M$,
    \begin{itemize}
        \item $|\nabla\rho|=1$,
        \item $\nabla\rho$ is tangent to $L$,
        \item $J\nabla\rho$ and $K\nabla\rho$ are orthogonal to $L$.
    \end{itemize}
    If $L$ is $I$-complex, we can also guarantee that $I\nabla\rho$ is tangent to $L$. Given such $g$, we extend $(I,J,K)$ and $L$ to the conical end in a $\rho$-invariant way as in the proof of \autoref{Prop_ComplexCompletion}. 
\end{proof}

%% file: CotangentBundles.tex
\section{Cotangent bundles of almost complex manifolds}
\label{Sec_CotangentBundles}

The cotangent bundle $M=T^*X$ of an almost complex manifold $X$ is in a natural way an almost quaternionic manifold. The goal of this section to describe Fueter maps to $M$ in terms of maps to $X$ and  to prove a quaternionic analog of Floer's theorem discussed in the introduction. 

\subsection{Almost complex and quaternionic structures}

Let $(X,g, I)$ be an almost Hermitian manifold, that is a manifold equipped with a Riemannian metric and a compatible almost complex structure. Denote by $\omega(\cdot,\cdot) = g(I \cdot, \cdot)$ the induced two-form. Let $\nabla$ be the connection on $TX$ which preserves $I$ and $g$ and whose torsion $N \in \Omega^2(X,TX)$ consists of the Nijenhuis tensor of $I$ and $\rd\omega$. For example, if $\rd\omega=0$ and $\nabla^0$ is the Levi--Civita connection, then
\begin{equation*}
    \nabla = \nabla^0 - \frac12 I \nabla^0 I,
\end{equation*}
see \cite[Appendix C.7]{McDuff2012}. In general, $\nabla$ agrees with the Levi--Civita connection if and only if $\rd\omega=0$ and $I$ is integrable, so that $(X,g,I)$ is a Kähler manifold. Using $\nabla$, we decompose the tangent space to $M = T^*X$ into the horizontal and vertical subspaces
\begin{equation}
  \label{Eq_HorizontalVertical}
    TM = H\oplus V \cong \pi^* TX \oplus \pi^* T^*X \cong \pi^*TX \oplus \pi^*TX 
\end{equation}
where $T^*X$ with $TX$ are identified using the metric. With respect to this decomposition, the canonical quaternionic triple $(I,J,K)$ on $M$ is defined by
\begin{equation*}
    I = \begin{pmatrix} I & 0 \\ 0 & -I \end{pmatrix},
    \quad
    J = \begin{pmatrix} 0 & -\mathrm{id} \\ \mathrm{id} & 0 \end{pmatrix}, 
    \quad 
    K = \begin{pmatrix} 0 & -I \\ -I & 0 \end{pmatrix}.
\end{equation*}
(Since the first almost complex structure is the canonical almost complex structure on $T^*X$ induced by one on $X$, when it is not likely to cause confusion we will denote both of them by $I$.)

\subsection{Convexity}
\label{sec:cotangent-convexity}

A basic fact about the almost complex geometry of cotangent bundles is that they are pseudo-convex. More precisely, the fiberwise norm function
\begin{gather*}
    r \colon T^*X \to \R \\
    r(\xi) = \frac12 |\xi|^2
\end{gather*}
is $J$-convex in a neighborhood of the zero section with respect to the almost complex structure $J$ introduced above. 

\begin{proposition}
    \label{Prop_CotangentConvexity}
    If $\varphi \colon X \to \R$ is an $I$-convex function, then there are positive functions $\epsilon$ and $\delta$ on $X$ such that the function
    \begin{equation*}
        \rho = r + \epsilon\varphi
    \end{equation*}
    is $IJK$-convex in the neighborhood of the zero section given by
    \begin{equation*}
        T_\delta^*X = \{ \xi \in T^*X \ | \ |\xi| \leq \delta(\pi(\xi)) \},
    \end{equation*}
    and the zero section is $JK$-conical with respect to $\rho$. 
    
    If $\varphi$ is uniformly $I$-convex and the first and second derivatives of $\varphi$, the Riemannian metric, and $I$ are uniformly bounded, then the functions $\epsilon$ and $\delta$ can be taken to be constant and $\rho$ is uniformly $IJK$-convex.
\end{proposition}

The proof is preceded by a general fact about taming two-forms. 

\begin{lemma}
    \label{Lem_TamingFormEstimate}
    Let $V$ be a vector space with a complex structure $I$ and a Hermitian metric.  
    If $\sigma \in \Lambda^2 V^*$ tames $I$, then there are constants $C>0$ and $c>0$ such that  
    \begin{equation*}
      \sigma(v,w) \leq -c(|v|^2+|w|^2) + C|v-I w|(|v|+|w|)\quad\text{for all } v,w \in V.
    \end{equation*}
\end{lemma}
\begin{proof}
    Set $u = v-I w$, or equivalently $I u = I v+w$. Applying $\sigma(\cdot, w)$ to the first equation and $\sigma(v,\cdot)$ to the second yields
    \begin{align*}
        \sigma(v,w) &= -\sigma(w,Iw) + \sigma(u,w), \text{and} \\
        \sigma(v,w) &= -\sigma(v,I v) + \sigma(Iu,v).
    \end{align*}
    The taming condition gives us $\sigma(v,Iv) \geq c|v|^2$ and similarly for $w$, so that
    \begin{equation*}
        \sigma(v,w) \leq - c(|v|^2+|w|^2) + C|u|(|v|+|w|). \qedhere
    \end{equation*}
\end{proof}

\begin{proof}[Proof of \autoref{Prop_CotangentConvexity}]
    To simplify the notation, we will consider the case when $\varphi$ is uniformly $I$-convex and the first and second derivatives of $\varphi$, $g$, and $I$ are uniformly bounded.
    The proof in the general case is the same except that the constants in the estimates that follow have to be replaced by functions on $X$.
    
    We want to show that the pull-back of the $IJK$-Levi form of $\rho$ by any quaternion-antilinear map from $\R^3$ to the tangent bundle to $T^*X$ is negative. The first step is to estimate the Levi form of $r$ with respect to $I,J,K$:
    \begin{align*}
          \sigma^I_r &= -\rd(\rd r\circ I), \\
          \sigma^J_r &= -\rd(\rd r\circ J), \\
          \sigma^K_r &= -\rd(\rd r\circ K).
    \end{align*}
    Let $\omega_I,\omega_J,\omega_K$ be the two-forms on $T^*X$ induced by the Riemannian metric and quaternionic triple: 
    \begin{itemize}
        \item $\omega_J$ is the canonical symplectic form on $T^*X$,
        \item $\omega_I = \omega \oplus -\omega$ with respect to decomposition \eqref{Eq_HorizontalVertical},
        \item $\omega_K(\cdot,\cdot) = - \omega_J(I\cdot,\cdot)$.
    \end{itemize} 
    Note that while $\omega_J$ is closed, this is not necessarily true for $\omega_I$ and $\omega_K$. By \eqref{Eq_HorizontalVertical}, the space of two-forms on $T^*X$ decomposes into
    \begin{equation*}
        \Lambda^2 T^*(T^*X) = \Lambda^2 H \oplus (\Lambda^1 H\otimes\Lambda^1 V) \oplus \Lambda^2 V,
    \end{equation*}
    which we will refer to as subspaces of forms of type $(2,0)$, $(1,1)$, and $(0,2)$ correspondingly. If $(q_i)$ are local coordinates on $X$ with the Riemannian metric given by $g_{ij}$,  and $(q_i,p_i)$ are the induced coordinates on $T^*X$, then
    \begin{equation*}
        \rd r = \sum_{ij} g_{ij}p_i\rd q_j + \sum_{ijk}\frac{\partial g_{ij}}{\partial q_k} p_ip_j \rd q_k. 
    \end{equation*}
    A straightforward computation in coordinates, using this formula and the definitions of $I,J,K$, shows that
    \begin{align*}
           \sigma^I_r &=  2\omega_I^{0,2} + \alpha_I\\
          \sigma^J_r &= \omega_J + \alpha_J, \\
          \sigma^K_r &= \omega_K + \alpha_K, 
    \end{align*}
    where $\alpha_I,\alpha_J,\alpha_K$ satisfy
    \begin{equation*}
        |\alpha_\bullet(\xi)| \leq C |\xi| \quad\text{for } \bullet = I,J,K, \text{ and }\xi \in T^*X,
    \end{equation*}
    where the constant depends on the derivatives of the metric and $I$. 
    
    Let $A = (A_1,A_2,A_3)$ be a triple of vectors tangent defining a quaternion-antilinear map $\R^3 \to T^*_\xi X$, that is:
    \begin{equation*}
        I A_1 + J A_2 + K A_3 = 0.
    \end{equation*}
    We want to estimate
    \begin{equation*}
        A^*\sigma_r = \sigma^I_r(A_2,A_3) + \sigma^J_r(A_3,A_1) + \sigma^K_r(A_1,A_2),
    \end{equation*}
    where, as usual, we identify three-forms on $\R^3$ with scalars using the volume form.
   Decomposing $A_i = a_i + b_i$ into the horizontal and vertical part yields
    \begin{gather}
        a_1 - I(b_2  + Ib_3) = 0, \nonumber \\
        b_1 - I(a_2-Ia_3) = 0. \label{Eq_FueterHorizontalVertical}
    \end{gather}
    Using these equations, we compute in the same way as in the proof of \autoref{Lem_TamingFormEstimate}:
    \begin{align*}
        2 \omega^{0,2}_I(A_2,A_3) +
        \omega_J(A_3,A_3) +
        \omega_K(A_1,A_2) = - |b|^2.
    \end{align*}
    Therefore
    \begin{equation*}
        A^*\sigma_r \leq -|b|^2 + C|\xi|(|a|^2+|b|^2)
    \end{equation*}
    
    Next, we estimate the $IJK$-Levi form of $\varphi \colon T^*X \to \R$. For convenience we use the same notation for $\varphi\colon X\to \R$ and its composition with the projection $T^*X\to X$. Let $\sigma^I_\varphi, \sigma^J_\varphi, \sigma^K_\varphi$ be the Levi forms of $\varphi \colon T^*X \to \R$ with respect to $I$, $J$, and $K$. By \autoref{Lem_TamingFormEstimate}, 
    \begin{equation*}
        \sigma_\varphi^I(A_2,A_3) \leq \sigma_\varphi^{I}(a_2,a_3) + C(|a|b|+|b|^2) \leq - c(|a_2|^2 + |a_3|^2) + C(|a||b|+|b|^2).
    \end{equation*}
    We also have
    \begin{gather*}
        \sigma_\varphi^J(A_3,A_1) \leq C(|a||b|+|b|^2),  \\
        \sigma_\varphi^K(A_1,A_2) \leq C(|a||b| + |b|^2),
    \end{gather*}
    since both terms contain $A_1 = a_1 + b_1$ and $|a_1| \leq |b|$ by \eqref{Eq_FueterHorizontalVertical}. We conclude that the term
    \begin{equation*}
        A^*\sigma_\varphi = \sigma^I_\varphi(A_2,A_3) + \sigma^J_\varphi(A_3,A_1) + \sigma^K_\varphi(A_1,A_2)
    \end{equation*}
    is bounded above by
    \begin{equation*}
        A^*\sigma_\varphi \leq - c|a|^2 + C(|a|b|+|b|^2). 
    \end{equation*}
    Finally, we have
    \begin{align*}
        A^*\sigma_\rho = A^*\sigma_r + \epsilon A^*\sigma_\varphi.
    \end{align*}
    Putting the estimates for $r$ and $\varphi$ together and assuming that $|\xi|\leq \delta$, we get
    \begin{equation*}
        A^*\sigma_\rho \leq -|b|^2 - c\epsilon|a|^2 + C\epsilon(|a|b|+|b|^2) + C\delta (|a|^2+|b|^2).
    \end{equation*}
    Therefore, for $\epsilon$ and $\delta$ sufficiently small,
    \begin{equation*}
        A^*\sigma \leq -c\epsilon |A|^2
    \end{equation*}
    for every quaternion-antilinear map $A \colon \R^3 \to T(T^*X)$.
\end{proof}

\subsection{Cauchy--Riemann operators}

Let $\partial_I$, $\delbar_I$ be the Cauchy--Riemann operators corresponding to the almost complex structure $I$ on $X$ and acting on the spaces of maps $v \colon \C \to X$:
\begin{gather*}
    \partial_I = \partial_s - I \partial_t, \\
    \delbar_I = \partial_s + I \partial_t. 
\end{gather*}
These operators take values in $(1,0)$ and $(0,1)$ forms on $\C$ with values in $TX$ with complex structure $I$. To simplify the notation we trivialize the spaces of forms using $\rd z = \rd s + i \rd t$ and $\rd \bar z = \rd s - i \rd t$, which allows us to identify
\begin{equation*}
    \Lambda^{1,0}\otimes TX = T_{1,0}X, \quad \Lambda^{0,1}\otimes TX = T_{0,1}X,
\end{equation*}
both of which are isomorphic as complex vector bundles to $TX$, the first with complex structure $I$ and the second with $-I$. 

These operators extend to the Cauchy--Riemann operators the spaces of maps $U \colon \C \to M$, defined by the induced almost complex structure $I$ on $M$. Every such map $U$ corresponds to a pair $(v,\xi)$ where $v \colon \C \to X$ and $\xi \in \Gamma(\C, v^*TX)$. We use here the anti-linear isomorphism $TX \cong T^*X$ given by the metric. We have 
\begin{equation}
    \label{Eq_DelOperator}
    \partial_I U = (\partial_I v, (\nabla\xi)^{0,1}).
\end{equation}
In the upcoming discussion, we will consider a perturbed version of the Cauchy--Riemann operator \eqref{Eq_DelOperator}. Denote by 
\begin{equation*}
    L_v \colon \Gamma(v^*TX) \to \Omega^{0,1}(v^*TX) \cong \Gamma(v^*TX)
\end{equation*}
the linearization of \eqref{Eq_DelOperator} at $v$. Let $L_v^*$ be its formal $L^2$ adjoint. Consider a perturbation of $\partial_I$ which, with respect to the splitting $U = (v,\xi)$, is given by
\begin{equation}
    D_I U = (\partial_I v, -L_v^*\xi). 
\end{equation}

\begin{lemma}
    \label{Lem_NijenhuisPerturbation}
    There is a section $\fn \in \Gamma(M, \End TM)$ such that
    \begin{equation*}
        D_I U = \partial_I U + \fn(\delbar_I U)
    \end{equation*}
    and 
    \begin{equation*}
        |\fn(\xi)| \leq |N(\pi(\xi))|\xi| \quad\text{for } \xi \in M= T^*X,
    \end{equation*}
    where $N$ is the torsion of $\nabla$, proportional to the Nijenhuis tensor of $I$ and $\rd\omega$. In particular, $D_I = \partial_I$ if and only if $I$ is integrable and $\rd\omega=0$.
\end{lemma}

\begin{proof}
    The linearization of $\partial_I$ at $v$ is
    \begin{gather*} 
        L_v \xi = (\nabla \xi)^{1,0} - N(\xi,\delbar_I v), 
    \end{gather*}
    where $N$ is the torsion of $\nabla$, proportional to the Nijenhuis tensor of $I$ and $\rd\omega$, see \cite[Lemma C.7.3]{McDuff2012}. With the identifications above, $L_v$ is an operator acting on the space of sections of $v^*TX$. The formal $L^2$ adjoint of $L_v$ is
    \begin{equation*}
        L_v^* \xi = -(\nabla\xi)^{0,1} - N^\dagger(\xi,\delbar_I v),
    \end{equation*}
    where $N^\dagger$ is the bilinear, $I$-antilinear operator defined by
    \begin{equation*}
        \langle N(v,w), z \rangle = \langle v,  N^\dagger(w,z)\rangle. 
    \end{equation*} 
    Therefore,
    \begin{equation*}
        D_I U = \partial_I U + \fn(\delbar_I U),
    \end{equation*}
    where, at $\xi \in T^*X \cong TX$ and with respect to the splitting of $TM$ into horizontal and vertical vectors \eqref{Eq_HorizontalVertical},
    \begin{equation*}
        \fn =
        \begin{pmatrix}
            0 & 0 \\
            N^\dagger(\xi,\cdot) & 0 
        \end{pmatrix}
    \end{equation*}
\end{proof}

\subsection{Complexified Floer correspondence}

Let $(X,g,I)$ be a compact almost Hermitian manifold with boundary and $M=T^*X$, and let $F \colon X \to \C$ be an $I$-holomorphic function. Endow $M$ with the quaternionic triple $(I,J,K)$ described previously. We have two $I$-complex Lagrangians $L_0 = X$ and $L_1 = \mathrm{graph}(\rd f)$ where $f=\Re(F)$. We are interested in the solutions of the Fueter equation with boundary on $L_0$ and $L_1$. Instead, consider the Fueter equation for maps $U \colon [0,1]\times\R^2 \to M$ perturbed by $F$:
\begin{equation}
    \label{Eq_FueterGradientPert}
    I(U)\partial_\tau U + J(U)\partial_s U + K(U)(\partial_t U +  \nabla f(U)) = 0,
\end{equation}
where $f = \Re(F)$ and $(\tau,s,t)$ are coordinates on $[0,1]\times\R^2$. As is the case for non-homogeneous perturbations of the pseudo-holomorphic map equation, solutions to  \eqref{Eq_FueterGradientPert} can be identified with solutions to a time-dependent Fueter equation.

\begin{lemma}
    Denote by $\varphi_\tau \colon T^*X \to T^*X$ the Hamiltonian flow of $f \colon T^*M \to \R$ with respect to the canonical symplectic form on the cotangent bundle. Define
    \begin{equation*}
        I_\tau = (\varphi_\tau)_* I(\varphi_\tau)_*^{-1}, \quad  J_\tau = (\varphi_\tau)_* J(\varphi_\tau)_*^{-1}, \quad  K_\tau = (\varphi_\tau)_* I(\varphi_\tau)_*^{-1}.
    \end{equation*}
    and
    \begin{equation*}
        \tilde U(\tau,s,t) = \varphi_\tau(U(\tau,s,t)).
    \end{equation*} 
    Then $\tilde U$ satisfies the homogeneous $\tau$-dependent Fueter equation
    \begin{equation*}
        I_\tau(\tilde U)\partial_\tau \tilde U + J_\tau(\tilde U)\partial_s U + K_\tau(\tilde U)\partial_t \tilde U = 0
    \end{equation*}
    if and only if $U$ satisfies the non-homogeneous equation  \eqref{Eq_FueterGradientPert}.
    Moreover, $\tilde U$ has boundary values on $L_0$ and $L_1$ if and only if $U$ has boundary values on $L_0$.
\end{lemma}
\begin{proof}
    We have
    \begin{gather*}
        \partial_\tau \tilde U = (\varphi_\tau)_* \partial_\tau U + \frac{\partial\varphi_\tau}{\partial\tau}\circ U, \\
        \partial_s \tilde U = (\varphi_\tau)_* \partial_s U,\\
        \partial_t \tilde U = (\varphi_\tau)_* \partial_t U.
    \end{gather*}
    Since $\varphi_\tau$ is the Hamiltonian flow of $f$, 
    \begin{equation*}
        \frac{\partial\varphi_\tau}{\partial \tau} = (\varphi_\tau)_* J \nabla f.
    \end{equation*}
    Multiplying each of the three equations by the corresponding almost complex structure and adding them together we obtain the result.
\end{proof}

We can rewrite equation \eqref{Eq_FueterGradientPert} using $I$ and $J$ only:
\begin{equation}
  \label{Eq_FueterUnperturbed}
  I(U)\partial_I U - J(U)\partial_\tau U + \nabla f(U) = 0.
\end{equation}
Moreover, we can consider the following modification of the equation, in which $\del_I$ is replaced by the operator $D_I$ introduced earlier: 
\begin{equation}
  \label{Eq_FueterPerturbed}
  I(U) D_I U - J(U)\partial_\tau U + \nabla f(U) = 0.
\end{equation}
Equations \eqref{Eq_FueterUnperturbed} and \eqref{Eq_FueterPerturbed} agree if and only if the almost complex structure $I$ on $X$ is integrable and $\rd\omega=0$, that is if $(X,g,I)$ is a Kähler manifold.

\begin{theorem}
  \label{Thm_FloerCorrespondencePerturbed}
  Let $X$ be a compact almost Hermitian manifold with boundary. If $F$ is $C^2$ small, then every solution $U \colon [0,1]\times\R^2 \to T^*X$ to equation \eqref{Eq_FueterPerturbed} with boundary on the zero section and asymptotic boundary conditions described in \autoref{Sec_BoundaryConditions}  satisfies $\partial_\tau U =0$ and corresponds to a complex gradient trajectory of $F$ in $X$, that is a map $v \colon \R^2 \to X$ satisfying
  \begin{equation*}
      \del_I v + i \nabla \bar F(v) = 0. 
   \end{equation*}
\end{theorem}

\begin{remark}
\label{remark:exclude-fueter-bubbling}
    It is straightforward that in the situation described by the theorem the cylindrical Fueter energy of $U$ and the Floer energy of $v$ agree, which gives us an a priori bound on the Fueter energy of $U$. 
    (Note that since the almost complex structures $J$ and $K$ on $T^*X$ are not integrable, we are not in the tamed situation of \autoref{Prop_EnergyBound}.)
    
    Furthermore, assuming that $(X,\omega,I)$ is such that $\rd\omega=0$ and we have a $C^0$ bound on the complex gradient trajectories of $F$, standard analysis of the Floer equation gives us compactness for the moduli space of solutions complex gradient trajectories and, therefore, for the moduli space of Fueter strips \cite{Haydys2015}, \cite{Wang2022}.
    While general conditions for the existence of a-priori $C^0$ estimates for complex gradient trajectories remain to be understood, one setting where such estimates are available can be found in \cite[Definition~2.1]{Wang2022}, following \cite{fan2018fukaya}.

    In particular, in such a situation we can exclude the possibility of the bubbling and singularity formation phenomena as described in \cite{Walpuski2017c} for a sequence of Fueter strips. These phenomena consist of the formation of interior pseudo-holomorphic sphere bubbles, and of essential singularities of the Fueter maps. This is interesting because the compactness problem for Fueter maps with values in a general hyperkähler is at the moment not very well understood. 
    Since $T^*X$ is not even hyperk\"ahler, and the the metric associated to an $IJK$-convex function on $T^*X$ satisfies no natural curvature inequalities, it seems difficult to use the general methods of \cite{Hohloch2009, Walpuski2017c} to exclude Fueter bubbling.
    Here we exclude it a posteriori using the correspondence with complex gradient trajectories.
    
    It is an interesting problem in geometric analysis to investigate whether such bubbling is excluded for sequences of Fueter maps mapping into all \emph{quaternionic Weinstein domains}, i.e. almost quaternionic manifolds with an exhausting $IJK$-convex function. Towards this question, we note that the existence of an $IJK$-convex function on a quaternionic Weinstein domain $M$ excludes the existence of pseudo-holomorphic spheres in $M$ for all complex structures in the twistor sphere of $M$ due to the existence of an exact symplectic form taming each $J_\tau$. Thus, only essential singularities of such Fueter maps remain to be excluded. 
    
    
\end{remark}

\begin{proof}
Using identification $T^*X = TX$, write $U = (v,\xi)$ as a pair consisting of a map $v \colon [0,1]\times\R^2 \to X$ and $\xi \in \Gamma(v^*TX)$. Decomposing \eqref{Eq_FueterPerturbed} into horizontal and vertical components, we obtain
\begin{equation}
    \label{Eq_FloerPlanesCotangent}
    \begin{split}
        \nabla_\tau \xi + I \partial_I v + \nabla f(v) &= 0, \\
         \partial_\tau v - I L_v^*\xi  &= 0.
    \end{split}
\end{equation}
In particular, solutions with $\xi=0$ satisfy $\partial_\tau v=0$, and in that case $v \colon \R^2 \to X$ is a complex gradient trajectory of $F$.

Since $\nabla$ preserves the metric, we have  $\nabla_\tau^* = -\nabla_\tau$. Apply this operator to the first equation. By definition, the linearization operator $L_v$ satisfies 
\begin{equation*}
    \nabla_\tau \partial_I v = L_v\partial_\tau v.
\end{equation*}
Using this, the fact that $\nabla$ preserves $I$, and the second equation, we arrive at a Weitzenb\"ock type formula
\begin{equation}
  \label{Eq_Weitzenbock}
  \nabla_\tau^*\nabla_\tau \xi + L_v L_v^*\xi - I \nabla^2 f(v)L_v^*\xi =0.
\end{equation}
Since $\xi = 0$ for $\tau=0,1$ and $\xi$ decays as $t\to\pm\infty$ and $s\to\pm\infty$, we can take the $L^2$ inner product with $\xi$ and integrate by parts to get
\begin{equation}
    \label{Eq_WeitzenbockIntegral}
    \int_{[0,1]\times\R^2} |\nabla_\tau\xi|^2 + |L_v^*\xi|^2 -   \langle I\nabla^2 f(v)L_v^*\xi, \xi \rangle = 0. 
\end{equation}
Since $\nabla^2 f$ is bounded and $\xi = 0$ for $\tau=0,1$, the last term satisfies
\begin{equation*}
    \int_0^1 |I \langle\nabla^2 f(v)L_v^*\xi, \xi \rangle| \rd \tau \leq C \int_0^1 |L_v^*\xi||\xi| \rd \tau \leq C \int_0^1 |L_v^*\xi||\nabla_\tau \xi| \rd \tau
\end{equation*}
Therefore, 
\begin{equation*}
     0 \geq \int_{[0,1]\times\R^2} (1-C\| f \|_{C^2})(|\nabla_\tau\xi|^2 + |L_v^*\xi|^2)
\end{equation*}
which for  $\| f \|_{C^2}$ sufficiently small implies $\nabla_\tau\xi = 0$ and so $\xi=0$.
\end{proof}

\subsection{Remarks on the modified equation}

If $(X,g,I)$ is not a Kähler manifold, then, unlike \eqref{Eq_FueterUnperturbed}, the modified equation \eqref{Eq_FueterPerturbed} considered in \autoref{Thm_FloerCorrespondencePerturbed} cannot be interpreted, at least not in an obvious way, as a genuine Fueter equation for a quaternionic triple on $M=T^*X$. However, since the two equations differ by a term proportional to the distance from the zero section, as in \autoref{Lem_NijenhuisPerturbation}, it is reasonable to expect that there is a bijection between solutions to \eqref{Eq_FueterPerturbed} and those solutions to \eqref{Eq_FueterUnperturbed} which stay sufficiently close to the zero section.  To be more precise, in the situation of \autoref{Thm_FloerCorrespondencePerturbed} we ask if there is a neighborhood of the zero section such that if $F$ is $C^2$ small, then all Fueter strips $U \colon [0,1]\times\R^2 \to X$ solving \eqref{Eq_FueterUnperturbed} and contained in that neighborhood satisfy $\partial_\tau U =0$. 

The proof of \autoref{Thm_FloerCorrespondencePerturbed} can be adapted to obtain the following estimate on the $L^2$ norm of $\partial_\tau U$.

\begin{lemma}
    \label{Lem_L2Estimate}
    In the situation of \autoref{Thm_FloerCorrespondencePerturbed}, there are constants $C>0$ and $\delta>0$ such that if $f$ is $C^2$ small, then every Fueter strip $U \colon [0,1]\times\R^2 \to X$ solving \eqref{Eq_FueterUnperturbed} with $\| \xi \|_{L^\infty} \leq \delta$ satisfies
   \begin{equation*}
       \|\partial_\tau U\|_{L^2} \leq C \delta^{1/2}.
   \end{equation*}
\end{lemma}

\begin{proof}
    We proceed in the same way as in \eqref{Eq_FloerPlanesCotangent} in the proof of \autoref{Thm_FloerCorrespondencePerturbed}, by decomposing the unperturbed equation  \eqref{Eq_FueterUnperturbed} into horizontal and vertical components:
    \begin{equation}
        \label{Eq_FloerPlanesCotangentUnperturbed}
        \begin{split}
            \nabla_\tau \xi + I \partial_I v + \nabla f(v) &= 0, \\
             \partial_\tau v - I (L_v^*\xi + N^\dagger(\xi,\delbar_I v)  &= 0.
        \end{split}
    \end{equation}
    We apply $\nabla_\tau^*$ to the first equation and use the second equation to obtain a Weitzenböck type formula for $\xi$. The torsion term contributes error terms to \eqref{Eq_Weitzenbock} and consequently to \eqref{Eq_WeitzenbockIntegral}. The contribution to the latter is
    \begin{equation*}
        \langle N^\dagger(\xi, \del_I v), L_v^*\xi \rangle_{L^2} +  \langle \nabla^2 f N(\xi,\delbar_I v) , \xi \rangle_{L^2}.
    \end{equation*}
    To control these terms, we estimate the norm of the torsion term:
    \begin{equation*}
        | N^\dagger(\xi, \delbar_I v) | \leq C |\xi| |\rd v| \leq C|\xi||\rd v - \nabla f \rd t| + C\| f \|_{C^1}  |\xi|
    \end{equation*}
    Using asymptotic boundary conditions and the energy identity for the Cauchy--Riemann operator $\del_I$ for maps into $X$, we get
    \begin{equation*}
        \int_\C |\partial_s v|^2 + |\partial_t v - \nabla f|^2  \leq  \int_\C |\partial_s v - I(\partial_t v - \nabla f)|^2 + C = \int_\C |\nabla_\tau\xi|^2 + C,
    \end{equation*}
    where $C$ is the maximum of the symplectic action of the limiting Hamiltonian trajectories of $g$, and therefore depends only on $(X,\omega)$ and $F$. Combining the above with
    \begin{equation*}
        \|\xi\|_{L^2} \leq C \|\nabla_\tau\xi\|_{L^2},
    \end{equation*}
    we obtain
    \begin{align*}
        |\langle N^\dagger(\xi, \delbar_I v) , L^*_v\xi \rangle_{L^2}| &\leq C( \|\xi\|_{L^\infty}+ \| f \|_{C^2})\| \nabla_\tau \xi\|_{L^2}\|L_v^*\xi\|_{L^2} + C \|\xi\|_{L^\infty}\|L_v^*\xi\|_{L^2}    \\
        &\leq C(\delta+\| f \|_{C^2})\|\partial_\tau U\|_{L^2}^2 + C\delta.
    \end{align*}
    Therefore, in the same way as in the proof of  \autoref{Thm_FloerCorrespondencePerturbed} we get
    \begin{equation*}
         C\delta \geq  (1-C(\epsilon+\delta))\|\partial_\tau U\|_{L^2}^2,
    \end{equation*} 
    which proves the statement for $\delta$ and $\epsilon$ sufficiently small.
\end{proof}

Suppose by contradiction that there is a sequence of solutions $U_i = (v_i, \xi_i)$ contained in the $\delta_i$-neighborhood of the zero section with $\| \xi_i \|_{L^2} > 0$ and $\delta_i \to 0$. One can envision using \eqref{Eq_FloerPlanesCotangentUnperturbed} to bootstrap the $L^2$ estimate from \autoref{Lem_L2Estimate} to extract a weak limit $U_i \to U$ contained in the zero section with $\partial_\tau U =0$. Such a limit is a complex gradient trajectory $v \colon \R^2 \to X$. The rescaled sequence $\xi_i / \| \xi_i \|_{L^2}$ would then converge to a non-zero $\xi \in \Gamma(v^*TX)$  satisfying
\begin{equation*}
    \nabla_\tau \xi = 0 \quad\text{and}\quad \nabla^{0,1}\xi=0. 
\end{equation*}
However, any such section is forced to be zero by the boundary conditions at $\tau=0,1$. Therefore, we would conclude by contradiction that for $\delta$ sufficiently small, all solutions contained in the $\delta$-neighborhood of the zero section must satisfy $\partial_\tau U = 0$ and $\xi =0$. We will not attempt to make this argument rigorous in this paper.

We end this section by remarking that the modified Fueter equation \eqref{Eq_FueterPerturbed} appearing in \autoref{Thm_FloerCorrespondencePerturbed} has a natural interpretation from an infinite-dimensional point of view. The idea is to apply Floer's adiabatic limit theorem, which identifies solutions of Floer's equation in the cotangent bundle of a manifold $\cX$, with gradient trajectories of a Morse function on $\cX$, to $\cX$ being the space of paths in a symplectic manifold $X$. The upcoming discussion uses \autoref{Sec_FloerRevisited}, so the reader might find it helpful to read the appendix first.

Let $\cX$ be the space of paths $\gamma \colon \R \to X$ such that $\gamma(s)$ converges exponentially as $s\to\pm\infty$. As this is only a heuristic discussion, we ignore various analytical subtleties such as precise asymptotic conditions at infinity.  Here we think of the domain $\R$ as the second $s$-factor in $[0,1]\times\R^2$ with coordinates $(\tau,s,t)$. Let $F = f+ig$, so that Hamiltonian trajectories of $g$ are anti-gradient trajectories of $f$. Consider the symplectic action functional with respect to the symplectic structure on $X$, perturbed by $g$:
\begin{equation*}
    \cA \colon \cX \to \R.
\end{equation*}
Gradient trajectories of $\cA$ are complex gradient trajectories of $F$:
\begin{gather*}
    v \colon \R^2 \to X, \\
    \partial_s v - I(\partial_t v + \nabla f(v)) = 0.
\end{gather*}
We wish to apply Floer's adiabatic limit theorem to the infinite-dimensional manifold $\cX$ equipped with the function $\cA$. In Floer's original approach, we consider a compatible almost complex structure $J$ on $T^*\cX$ and perturb the equation for $J$-holomorphic strips $U\colon \R\times[0,1]\to T^*\cX$ by the gradient of $\cA$. The resulting equation is \eqref{Eq_StandardFloer}. We would expect that solutions of this equation are invariant in the $\tau$-direction and correspond to gradient trajectories of $\cA$ in $\cX$, i.e. complex gradient trajectories of $F$ in $X$. 

Formally, the cotangent bundle $T^*\cX$ is the space of paths $\gamma \colon \R \to T^*X$, so that a strip $U \colon \R\times[0,1] \to T^*\cX$ can be identified with a map
\begin{equation*}
    U \colon [0,1]\times\R^2 \to T^*X,
\end{equation*}
and Floer's equation \eqref{Eq_StandardFloer} in infinite dimensions translates into a partial differential equation for $U$. However, it is easy to see that this equation is not elliptic. The reason is that the gradient of $\cA$ contains the $s$-derivative $U$ only in the horizontal direction in the splitting of the tangent space to $T^*X$ into horizontal and vertical spaces. 

This issue can be fixed this by considering instead a different Floer equation \eqref{Eq_NonStandardFloer}, in which the $J$-holomorphic strip equation is perturbed by the Hamiltonian vector field of a function $\cH \colon T^*\cX \to \R$ rather than the gradient of $\cA$. Here $\cH$ is chosen so that its Hamiltonian vector field restricts to the gradient vector field of $\cA$ along the zero section in $T^*\cH$. For more details, see  \autoref{Sec_FloerRevisited}.  In the case at hand, the resulting equation \eqref{Eq_NonStandardFloer} translates into a partial differential equation for $U \colon [0,1]\times\R^2 \to T^*X$ which contains the full derivative $\partial_s U$ and, in fact, agrees with the perturbed Fueter equation \eqref{Eq_FueterPerturbed} discussed earlier. This is shown by a straightforward computation, by decomposing the derivative of $U$ into horizontal and vertical components as in the proof of \autoref{Thm_FloerCorrespondencePerturbed}. 

\begin{lemma}
    The Floer equation \eqref{Eq_NonStandardFloer} for strips $U \colon \R\times[0,1] \to T^*\cX$ with a Hamiltonian perturbation corresponding to $\cA$ agrees with the perturbed Fueter equation \eqref{Eq_FueterPerturbed}, if we interpret $U$ as a map $U \colon [0,1]\times\R^2 \to T^*X$.
\end{lemma}

In the finite-dimensional setup, \autoref{Prop_NonStandardFloerCorrespondence} establishes a correspondence between solutions of the Floer equation \eqref{Eq_NonStandardFloer} with values in $T^*\cX$ and gradient trajectories of $\cA$ in $\cX$. In the infinite-dimensional situation at hand, the former are solutions of the perturbed Fueter equation with values in $T^*X$ and the latter are complex gradient trajectories of $F$ in $X$. We conclude that \autoref{Thm_FloerCorrespondencePerturbed} is an infinite-dimensional version of \autoref{Prop_NonStandardFloerCorrespondence} applied to the space of paths in $X$.

%% file: TQFT.tex
\section{Topological field theories}
\label{sec:aspects-of-tqft}

This section briefly summarizes the basic philosophy of topological quantum field theories (TQFT) and its connection to Floer homology. This material is completely standard, but may be of interest to those who are not completely familiar with the connection between partial differential equations and algebraic structures. 


\newcommand{\kk}{\mathbf{k}}

\paragraph{Algebraic structure.}
According to Atiyah \cite{atiyah-tqft}, a TQFT is a functor from a cobordism category of manifolds to a monoidal category, e.g. the category of graded vector spaces over a field $\kk$. Concretely, the simplest notion of a $2$-dimensional TQFT $F: \mathrm{Cob}_2 \to \mathrm{Vect}_{\kk}$ assigns in a functorial way
\begin{itemize}
    \item a vector space $V = F(S^1)$ to the circle $S^1$;
    \item the vector space $V^{\tensor n} = F(\bigsqcup_{i=1}^n S^1)$ to $n$ copies of the circle $\bigsqcup_{i=1}^n S^1$; 
    \item a linear map $V^{\tensor n} \to V^{\tensor m}$ to any cobordism $\Sigma$ from $\bigsqcup_{i=1}^n S^1$ to $\bigsqcup_{i=1}^m S^1$.
\end{itemize}
This gives us the following structure:
\begin{itemize}
    \item A commutative associative product $V \tensor V \to V$ and cocommutative coassociative coproduct $V \to V \tensor V$ correspoding to the pair of pants cobordism and its inverse, together with a unit  $\kk \to V$ for the  product and a counit  $V \to \kk$ for the coproduct.
    \item By composing the cocounit with the product one gets a pairing $\langle \cdot, \cdot \rangle: V \tensor V \to \kk$, and dually, by composing the unit and the product, a copairing $\Delta: \kk \to V \tensor V$. Decomposing $S^1 \times I$ (corresponding to the identity $V \to V$) into the composition of a pairing and a copairing implies that 
    \[ V \simeq V \tensor \kk \xrightarrow{1 \tensor \Delta} V \tensor V \tensor V \xrightarrow{\langle \cdot , \cdot \rangle \tensor 1} \kk \tensor V \simeq V \]
    is the identity map, and thus that the pairing is perfect. Writing the product of $a, b \in V$ as $ab$, one verifies that $\langle ab, c \rangle = \langle a, bc \rangle$.
    \item The data of a perfect pairing together with a commutative associative unital product satisfying the identity above is a \emph{Frobenius algebra}, and every $2d$ TQFT determines and is determined by a Frobenius algebra \cite{Kock2003-lj}.
\end{itemize}
In summary, a TQFT is a functor from a cobordism category, which can be specified by collection of algebraic operations satisfying certain identities.  

\paragraph{Floer cohomology.}
Certain examples of TQFTs are naturally associated with physical theories, and indeed that was Atiyah's original motivation for introducing the axioms pf TQFT. Given $M$ in the cobordism category, $F(M)$ is meant to be the space of quantum states of the field theory on $M$, while the map $F(M_1) \to F(M_2)$ associated to a cobordism $N$ from $M_1$ to $M_2$ is meant to be the Hamiltonian evolution of the quantum states through time. The cobordism map is then typically defined by evaluating a path-integral over some infinite-dimensional space of fields on $N$ \cite{freed-cobordism}. While the integration over the infinte-dimensional space of fields is hard to make rigorous, in certain cases, due to supersymmetry \cite{Schwarz1997} the path integral can be computed, or rigorously defined, by counting solutions to some partial differential equations. Many invariants of Floer-theoretic origin  are of this form. 

We now explain the general pattern by the example of Hamiltonian Floer cohomology of a symplectic manifold. A Hamiltonian function $H: M \times S^1 \to \R$ on a symplectic manifold $M$ defines an action functional $\mathcal{A}_H: \mathcal{L}M \to \R$  on the free loop space. Hamiltonian Floer cohomology $HF^*(H)$, which is formally the Morse homology of $\cA_H$, computes the space of states of $2$-dimensional TQFT on $M$ \cite{witten-susy-morse-theory, floer-infinite-dimensional-morse-theory,witten-tqft}. Defining $HF^*(H)$ involves the study of the gradient flow equation of $\mathcal{A}_H$, the Floer equation:
\begin{gather*}
    u \colon \R \times S^1 \to M, \\
    \partial_t u + J(u)\partial_\tau u = \Grad H(u),
\end{gather*} 
with $(t,\tau)$ denoting the coordinates on $\R\times S^1$ and $J$ is an almost complex structure on $M$ compatible with the symplectic form.  This equation admits a natural generalization to surfaces, the inhomogeneous pseudo-holomorpic map equation
\begin{equation}
\label{eq:inhomogeneous-map-equation}
\begin{gathered}
    u: \Sigma  \to M, \nonumber \\
    (\rd u - \alpha(u) )_J^{0,1} = 0, 
    \end{gathered}
\end{equation}
where $\alpha \in \Omega^1(\Sigma\times M, TM)$ is an inhomogenous term generalizing $\nabla H$ above.  Here $\Sigma = \hat{\Sigma} \setminus S$ is the complement of a finite collection of points $S$ in a closed Riemann surface $\hat\Sigma$. Counting solutions to \eqref{eq:inhomogeneous-map-equation} with asymptotic boundary conditions at each point of $S$ defines TQFT operations corresponding to a cobordism $\hat{\Sigma} \setminus U$, where $U$ is a small open neighborhood of $S$ \cite{Seidel2008}. 

Note that in the formulation of the inhomogeneous pseudo-holomorphic map equation, one must fix boundary conditions \emph{at infinity} for $\Sigma$, even though the corresponding cobordism is a compact manifold with boundary. In general, there can be different kinds of boundary conditions for partial differential equations, and these correspond to variations on the definition of a TQFT. It may be that the surface $\Sigma$ needs to be decorated with certain structures in order to make sense of the equations. Moreover, in order to ensure compactness of the set of solutions to the equations, one may need to restrict or modify the associated cobordism category further. As such, the  formal algebraic properties of the TQFT and the geometric and analytical structure of the associated partial differential equation are closely intertwined in a way that is challenging to fully axiomatize.

\paragraph{Category theory.}

It is an old idea, formalized by the Baez--Dolan cobordism hypothesis \cite{Baez1995}, that higher-dimensional TQFT are associated to structures in higher category theory. A cobordism is simply a manifold with boundary, interpreted as a morphism by dividing the boundary into two pieces. As such, one wishes to interpret a cobordism-between-cobordisms, i.e. a manifold with corners modeled on $\R_{>0}^2 \times \R^{n-2}$ with a certain decomposition of its boundary, as a morphism-between-morphisms in the cobordism \emph{2-category}; and more generally to define notions of cobordisms-between-cobordisms-between-cobordisms in terms of general manifolds with corners, viewed as higher morphisms in a cobordism $n$-category. An extended TQFT is then suppposed to be a functor from a higher category of cobordisms to another higher category, prototypically the $n$-category of $\kk$-linear $(n-1)$-categories. 

\begin{figure}[H]
\centering
\includegraphics[width=6cm]{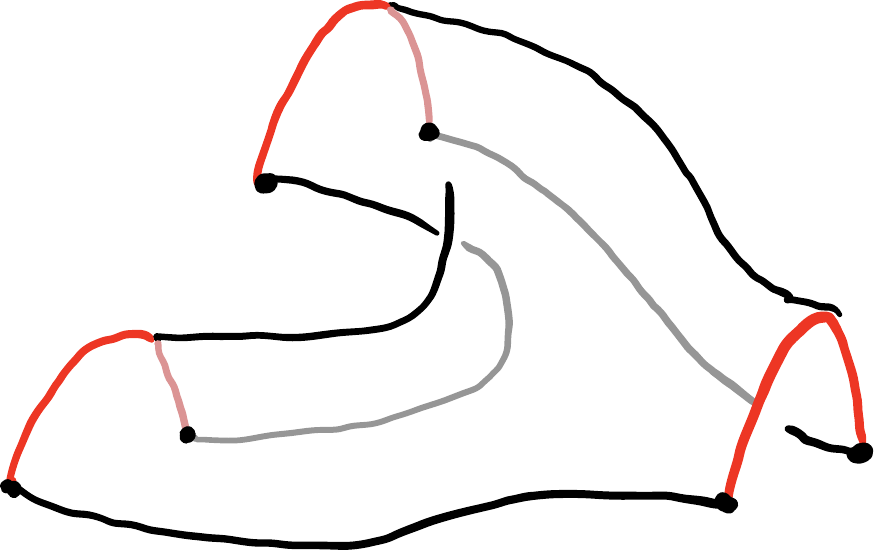}
\caption{\small A two-dimensional cobordism between a pair of $1$-dimensional cobordisms (in red). From the perspective of category theory this can be interpreted as a morphism between morphisms. }
\label{fig:cobordism-with-corners}
\end{figure}

While this general notion can be made precise \cite{lurie-tqft}, when building a TQFT out of solutions to partial differential equations it pays to be explicit. Moreover, the analysis of the given equations  may force us to modify the TQFT formalism in minor ways. For example, while many aspects of the Fukaya category can be formulated in a general extended TQFT framework, typically symplectic topologists work in a formalism that is better described by the language of an \emph{open-closed $2d$ TQFT} \cite{Moore-Segal, BlumbergCohenTeleman}. In this formalism:
\begin{itemize}
    \item We have a category $\mathcal{F}$.
    \item One assigns, to an interval with endpoints labeled by a pair of objects $(L_0, L_1)$ of $\mathcal{F}$, the vector space given by the morphisms $\mathcal{F}(L_0, L_1)$. There is also an assignment of a vector space $\mathcal{F}(S^1)$ to the circle.
    \item To every surface with corners, with certain codimension one boundary components labeled by objects, and the remaining codimension one boundary components partitioned into incoming and outgoing pieces, one defines a cobordism map. These maps compose functorially as we compose the corresponding bordisms.
\end{itemize}

\begin{figure}[H]
    \centering
    \includegraphics[width=9cm]{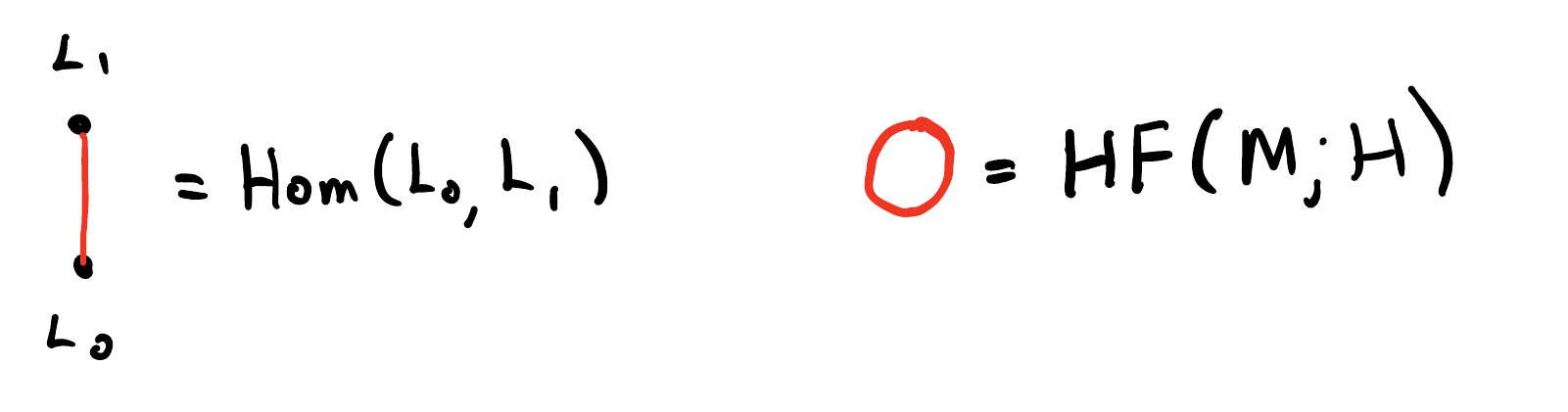}
    \includegraphics[width=8cm]{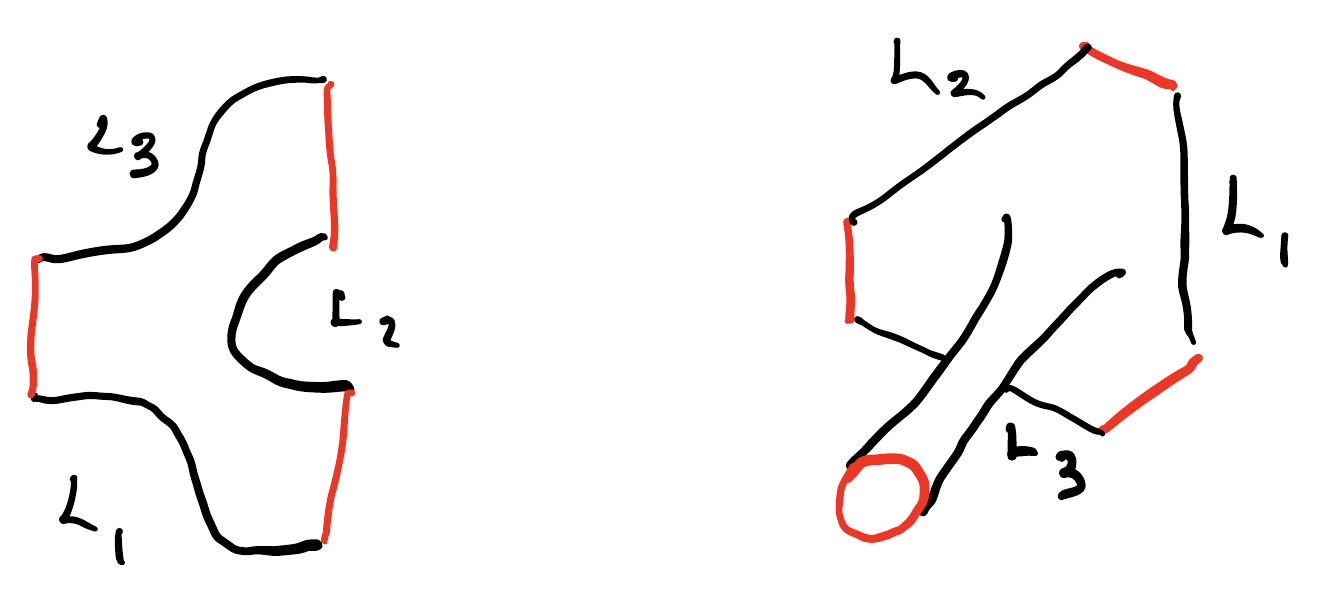}
    \caption{\small The open-closed field theory corresponding to the Fukaya category. Composition in the Fukaya category arises from the cobordism in the bottom left, and the open-closed map, discussed later in \autoref{sec:zeta-instanton-model}, from the cobordism in the bottom right.}
    \label{fig:my_label}
\end{figure}

To a symplectic topologist, $\cF$ is the Fukaya category of a symplectic manifold $M$, for Lagrangian submanifolds $L_0,L_1$ the vector space $\cF(L_0,L_1)$ is the Lagrangian Floer cohomology of $L_0,L_1$, and $\cF(S^1)$ is the Hamiltonian Floer cohomology of $M$.  One defines the cobordism map by counting the number of solutions to \eqref{eq:inhomogeneous-map-equation} over all possible conformal structures on the domain of the cobordism.

In the next section we will give an outline of a similar formalism which should encompass the basic algebraic structure associated to the Fueter TQFT. The most complex object is the object assigned to a square with certain labels, which will be the \emph{Fueter $2$-category}. There are various notions of $2$-categories with various levels of strictness in their composition; we recall the easiest and strictest notion now.  A strict $2$-category $\mathcal{C}$ is a category enriched in the category of categories. Concretely, this is
\begin{itemize}
    \item A set of objects $\mathrm{Ob}(\mathcal{C})$; we will write $L \in \mathcal{C}$ for $L \in \mathrm{Ob}(\mathcal{C})$. 
    \item For every pair of objects $(L_0, L_1)$, a \emph{hom-category}  $\mathcal{C}(L_0, L_1)$. The objects the categories $\mathcal{C}(L_0, L_1)$ are the $1$-morphisms in $\mathcal{C}$.
    \item The morphisms between a pair of $1$-morphisms $p, q \in \mathcal{C}(L_0, L_1)$ are to be thought of as $2$-morphisms of $\mathcal{C}$. Thus, every $2$-morphism has $1$-morphism as its source and its target, and given a triple of morphisms $p, q, r \in \mathcal{C}(L_0, L_1)$ and a pair of $2$-morphisms $u_1: p \to q$, $u_2: q \to r$, we can \emph{vertically compose} $(u_1, u_2)$ to produce $u_2u_1: p \to r$. 
    \item We have composition bifunctors $\mathcal{C}(L_0,L_1) \times \mathcal{C}(L_1, L_2) \to \mathcal{C}(L_0, L_2)$. Thus, $1$-morphisms can be composed along objects, and a pair of $2$-morphisms $u_1: p \to q$, $u_2: s \to t$ for $p, q \in \mathcal{C}(L_0,L_1)$ and $s, t \in \mathcal{C}(L_1, L_2)$ \emph{horizontally compose} to a $2$-morphism $u_2u_1: sp \to tq$ in $\mathcal{C}(L_0, L_2)$.
\end{itemize}

This data is graphically summarized in  \autoref{fig:2-category}. Note that for any $L \in \mathcal{C}$, the category $\mathcal{C}(L, L)$ is necessarily \emph{monoidal}, with the monoidal structure given by composition of $1$-morphisms. 

\begin{figure}[H]
\centering
\includegraphics[width=8cm]{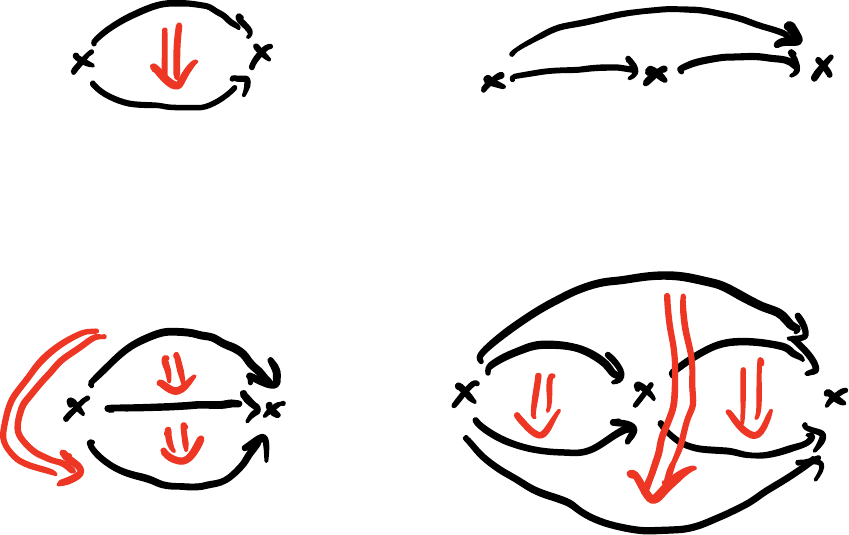}
\caption{\small The data of a $2$-category. Top left: crosses represent objects, black represent $1$-morphisms, red arrow represents a $2$-morphism. Top right: Composition of $1$-morphisms. Bottom left: vertical composition of $2$-morphisms. Bottom right: Horizontal composition of $2$-morphisms. }
\label{fig:2-category}
\end{figure}

%% file: CategoricalAppendix.tex
\section{Problems and conjectures}
\label{Sec_CategoricalAppendix}

\subsection{More on the Fueter $2$-category}
\label{app:more-on-fueter}

In this section, we provide several additional remarks on the expected behavior of the Fueter $2$-category. 

\paragraph{Complex linearity and anti-linearity.}

Upon first glance, one may find the appearance of the \emph{anti-complex} gradient flow in the derivation of \autoref{Subsec_FromMorseToFueter} strange. However, it is actually rather natural, as the Fueter equation is the condition that the differential of the map is \emph{anti}-quaternion-linear. We discuss the underlying linear algebra further in  \autoref{Subsec_LinearAlgebra}. Given this observation, is no surprise that the dimensional reductions of the Fueter equations \eqref{eq:sec2-fueter-eq} are $J_\theta$-\emph{anti}-holomorphic strips. Moreover, unlike in dimension two, where holomorphic and anti-holomorphic maps are related by a conjugation of the underlying domain, the analog of the Fueter equation where we ask for the differential to be quaternion-linear has essentially no solutions \cite{Fueter1935}, so if we wish to preserve the symmetric form of the Fueter equations \eqref{Eq_FueterIntroduction}, we must incorporate the corresponding preference for anti-gradient trajectories and anti-holomorphic strips into our conventions. Luckily, via the substitution $J_\theta \to J_{\theta + \pi}$ and the reordering of inputs in $A_\infty$-operations, one can set the conventions up such that anti-holomorphic strips never explicitly arise in the Fueter $2$-category, as we saw in \autoref{sec:fueter-tqft}.

\paragraph{Background angles.}
In the complex Morse theory model for the Fukaya--Seidel category, a choice of background angle $\hat{\theta}$ is used all existing proposals to put an ordering on the critical points of the complex Morse function, as in the more standard construction of the Fukaya-Seidel category. In particular, the complex Morse theory model of the Fukaya--Seidel category $\mathrm{FS}_{\hat{\theta}}^\zeta(X,F)$ also contains a directed subcategory $\mathrm{FS}^{\zeta, \to}_{\hat{\theta}}(X,F)$, see \cite{Wang2022}, and the categories $\mathrm{FS}_{\hat{\theta}}^\zeta(X,F)$ also form a local system of categories over $S^1$ as the paragraph on background angles in \autoref{sec:zeta-instanton-model}. 

To define $\mathrm{Fuet}_M$, we should also fix a single background angle $\hat{\theta}$ and then formally define $\mathrm{Fuet}_M(L_0, L_1) = \mathrm{FS}^{\zeta}_{\hat{\theta}}(\cP, \cA_\C)$, using the same $\hat{\theta}$ for all pairs $L_0, L_1$. The horizontal composition of morphisms $\mathrm{Fuet}_M(L_0, L_1) \tensor \mathrm{Fuet}_M(L_1, L_2) \to \mathrm{Fuet}_M(L_0, L_2)$ should be given in terms of a count of regular $J_{\hat{\theta}}$-holomorphic triangles. In between any two intersection points of exact Lagrangians $L_0$ and $L_1$, there are $J_\theta$-holomorphic strips only for isolated values of $\theta$. If we choose a different background angle $\hat{\theta}'$ to define $\mathrm{Fuet}_M$, there should be an equivalence between the $2$-categories built using these choices via analogs of the functors \eqref{eq:mutation-equivalence} above. 

    To recover the Fukaya category upon taking Hochschild homology \eqref{eq:naive-decategorification}, composition of $1$-morphisms in $\mathrm{Fuet}_M$ should be given by the usual counts of $J_{\hat{\theta}}$-holomorphic triangles, as opposed to any kind of parameterized counts over $S^1$. As one varies $J_{\hat{\theta}}$, by bifurcation theory, the counts of $J_{\hat{\theta}}$-holomorphic triangles should jump exactly when a $J_{\hat{\theta}}$-holomorphic strip appears, \emph{which should couple precisely with the nontrivial mutation isomorphisms} \eqref{eq:mutation-equivalence} from $\mathrm{Fuet}_{M, \hat{\theta}^*-\epsilon}(L_i, L_j) \to \mathrm{Fuet}_{M, \hat{\theta}^*+\epsilon}(L_i, L_j)$ as $\hat{\theta}$ varies through an exceptional angle $\hat{\theta}^*$, in order given an equivalence of two categories $\mathrm{Fuet}_{M, \hat{\theta}^*-\epsilon} \simeq \mathrm{Fuet}_{M, \hat{\theta}^*+\epsilon}$ making the choice of background angle \emph{irrelevant}.

\begin{remark}
\label{remark:model-comparison}
In a formal definition of the $A_\infty$-category $\mathrm{Fuet}_M(L_0, L_1)$, one must take either the Gaiotto--Moore--Witten approach  or the Haydys approach to the complex Morse theory model. It is probable that the Haydys approach is more straightforward in the setting of $J_\theta$-holomorphic curves and Fueter maps, as it does not involve solving  \autoref{problem:gmw-compactification} as a first step. However, this has the significant downside that the morphism complexes in the Haydys model for $\mathrm{Fuet}_M(L_0, L_1)$ will be generated by solutions to a perturbation of the $J_{\theta(t)}$-holomorphic strip equation, where $\theta(t)$ is a very particular family of angles $\theta: \R \to S^1$. Finding such $J_\theta(t)$-holomorphic strips in a situation with a concrete geometry may be significantly harder than finding $J_\theta$-holomorphic strips for fixed $\theta$, and thus may make it challenging to compute Fueter hom-categories in any situation where a formal method for computation may not be available. In certain cases, such as the setting of  \autoref{conj:categorical-floer-theorem} below, the Haydys model should suffice to provide a computable definition; but we see both approaches as valid and necessary for the development of $3d$ mirror symmetry.
\end{remark}

\paragraph{Novikov rings.}
\label{sec:novikov-rings}
In many cases, the holomorphic action functional $\mathcal{A}^\C$ defined by its differential \eqref{eq:holomorphic-action-1-form} is multivalued on the path space, and one must pass to the Novikov cover to make it single-valued. As such, defining $\mathrm{Fuet}_M(L_0, L_1) = \mathrm{FS}(\cP, \cA_\C)$ would then give rise to an infinite number of objects, somehow related by deck-transformations. Even in the finite-dimensional case, when $\rd F \in \Omega^{1,0}(X)$ is replaced by  a closed holomorphic one-form on $X$, there is yet no standard formalism for studying the Fukaya--Seidel categories of the pullback of the one-form to the appropriate Novikov cover of $X$.  While it is not clear to the authors if there is any essential challenge in adapting the existing definitions of the Fukaya--Seidel category to this setting, it is important for the foundations of the Fueter $2$-category to give a precise exposition of such a construction.
\begin{problem}
Define the Fukaya--Seidel category of a closed holomorphic one-form on a Kähler manifold, both with a more classical symplectic approach and in the complex Morse theory model. 
\end{problem}
It is advantageous and nontrivial to give a definition where the one-form is not required to be holomorphic. Indeed, Fukaya--Seidel categories make sense for symplectic Lefschetz fibrations, which have a significant amount of flexibility. However, it is not immediate what is the best definition of a correspondingly flexible symplectic of a holomorphic one-form; the essential challenges in defining this notion are about establishing energy bounds for the relevant pseudo-holomorphic curves or corresponding complex Morse trajectories. The flexibility of such a definition would be helpful in the infinite-dimensional setting, where one wishes to perturb the holomorphic symplectic action functional or to work in some \emph{weak} or \emph{almost} hyperk\"ahler setup, see \autoref{sec:further-analytic-aspects}.

\subsection{Fueter Category and Holomorphic Floer Theory}
\label{sec:more-on-decategorification}

There are several existing works which study the enhancements of the Floer theory of complex Lagrangians that arise from the ambient holomorphic symplectic geometry. We argue that all of them can viewed as simplifications of the data associated to the hom-categories of the Fueter $2$-category. 

\paragraph{Connections to perverse sheaves.}
As a consequence of  \autoref{prop:solomon-verbitsky}, one might expect that $HF^*(L_0, L_1)$ is computable in terms of classical topological invariants, even when $L_0$ is not transverse to $L_1$. A proposed definition of the correct topological invariant was given by Brav--Bussi--Dupont--Joyce--Szendroi \cite{Brav2015}, who defined a \emph{perverse sheaf of vanishing cycles} $P_{L_0, L_1}$ supported on $L_0 \cap L_1$. Heuristically, one defines this sheaf locally by writing $L_1$ as the graph of the differential of a holomorphic function $F: L_0 \to \C$, and then defines $P_{L_0, L_1}$ as the \emph{sheaf of vanishing cycles} \cite{MasseyPerverseSheafNotes} of $F$. An isolated intersection of $(L_0, L_1)$ corresponds to an isolated critical point of $F$, and hypercohomology of the summand of $P_{L_0, L_1}$ supported at that intersection point is simply the number of critical points of a Morsification $\tilde{F}$ of $F$.  It is a well-known conjecture that $HF^*(L_0, L_1) \simeq \mathbf{H}(P_{L_0, L_1})$, which generalizes \eqref{eq:floer-homology-of-transverse-holomorphic-lagrangians}. 

The sheaf of vanishing cycles $P_{L_0, L_1}$ admits a \emph{monodromy automorphism} $h:P_{L_0, L_1} \to P_{L_0, L_1}$.  In the case where $L_1$ is the graph of a differential of a function $F: L_0 \to \C$ in $T^*L_0$ and $F$ has a single isolated critical point, this corresponds to considering the monodromy automorphism of $\tilde{F}$. We call the induced automorphism $h: \mathbb{H}(P_{L_0, L_1}) \to \mathbb{H}(P_{L_0, L_1})$ the \emph{local monodromy}. 

On the side of the Lagrangian Floer cohomology group, one can consider the Floer-theoretic continuation map 
\begin{equation}
    HF^*(L_0, L_1, \omega_{\hat{\theta}}) \to HF^*(L_0, L_1, \omega_{\hat{\theta}})
\end{equation}
given by counting solutions to 
\begin{equation}
\label{eq:continuation-equation}
\begin{gathered} 
    u: \R \times [0,1] \to M, \\
    \partial_t u + J_{\theta(t)}(u) \partial_\tau u = 0, \\
    \lim_{s \to \pm \infty} u(s,t) \in L_0 \cap L_1
\end{gathered}
\end{equation}
where $\theta: \R \to S^1$ is equal to $\hat{\theta}$ for $|t| \gg 0$, and such that $\theta$ winds once around the $S^1$ as $t$ increases. We argue that this monodromy map is the map induced by the Serre functor of $\mathrm{Fuet}_M(L_0, L_1)$ under \eqref{eq:decategorification-1}. Indeed, in the finite-dimensional case, the Picard-Lefschetz transformations of \eqref{eq:picard-Lefschetz-identity} are simply the changes of basis induced by the bifurcation analysis of the family of Morse functions $Re(e^{-i \theta}W)$. It is standard \cite{Floer1988a} that instead of constructing isomorphisms of Morse-Floer homology groups using bifurcation analysis, one can construct the same maps using a continuation equation; the continuation equation for the family $\Re(e^{-i\theta}\cA_\C)$ is the equation \eqref{eq:continuation-equation} above. Thus, the global monodromy map, which is an object of real symplectic topology, is the simplest shadow of nontrivial structures associated to the Fueter $2$-category.

\paragraph{Connections to resurgence.}
It is expected, following Kontsevich--Soibelman \cite{Kontsevich2020}, that the difference between the local and the global monodromy can be viewed as a \emph{Stokes datum}. Indeed, the formal path integral
\begin{equation}
    \label{eq:path-integral}
    \int_{\mathcal{P}_{L_0, L_1}} e^{i \hbar \cA_\C(\gamma)} \mathcal{D} \gamma
\end{equation}
should satisfy an irregular differential equation in $\hbar$ by analogy to the finite-dimensional case of exponential integrals. The irregular Riemann--Hilbert correspondence, in the corresponding finite dimensional case 
\begin{equation}
\label{eq:finite-oscillatory-integral}
\int_{\Gamma} e^{i\hbar F(x)} \mathrm{vol}
\end{equation}
where $F: X \to \C$ is a finite dimensional holomorphic Morse function and $\mathrm{vol} \in \Omega^{top}(X)$ is a holomorphic volume form, then allows one to reconstruct the differential equation satisfied by functions of the form \eqref{eq:finite-oscillatory-integral}. We will not exposit the theory here; for details, see for example \cite{vanderPut2003}. One may summarize \cite{Kontsevich2020} by saying that the exceptional $J_\theta$-holomorphic strips between complex Lagrangians may be counted to enhance the $S^1$ family of Fukaya categories $\mathrm{Fuk}(M, \omega_\theta)$ into a holomorphic family of categories over  over $\mathbb{P}^1$ where morphisms between complex Lagrangians are $D$-modules with irregular singularities at $0$ and $\infty$, which formally correspond to the path integrals \eqref{eq:path-integral}. 

In the case of $(L_0, L_1)$ being a pair of complex Lagrangian submanifolds of a complex cotangent bundle $T^*X$, the Stokes data associated to the two integrals \eqref{eq:path-integral}, \eqref{eq:finite-oscillatory-integral} should be identified as a consequence of Floer's theorem \eqref{Eq_MorseEqualsFloer}. The categorification of this statement is our  \autoref{conj:categorical-floer-theorem} below. Thus, we see the structure of the Fueter $2$-category as a categorification of the approach to holomorphic Floer theory suggested in \cite{Kontsevich2020}.

The following problems, while not strictly related to the Fueter $2$-category, would go a long way to giving foundations to the Fueter $2$-category in the setting where the complex Lagrangians involved have bad intersections.
\begin{problem}
Make sense of the global monodromy operator $ HF^*(L_0, L_1; \omega_{\hat{\theta}}) \to HF^*(L_0, L_1; \omega_{\hat{\theta}})$ in cases of interest when $L_0$ intersects $L_1$ in a non-compact set. This involve finding appropriate perturbations of $L_1$ to a family of non-holomorphic submanifolds $\tilde{L}^{\theta}_1$, depending on $\theta$, which are Lagrangian for $\omega_\theta$ and are all transverse to $L_0$. The core of the problem is to understand the appropriate constraints on the perturbations near the ends at infinity of $L_0 \cap L_1$, in order to make the cohomology groups $HF^*(L_0, \tilde{L}^\theta_1; \omega_{\theta})$ isomorphic to $\mathbb{H}(P_{L_0, L_1})$ for every $\theta$. 
\end{problem}

\begin{problem}
Define a $D$-module associated to a pair of complex Lagrangians $(L_0, L_1)$ which intersect non-transversely in a non-compact locus. 
\end{problem}

\begin{remark}
The authors do not know of any cases where the global monodromy operator is known to differ from the local monodromy operator except for the cases of complex Lagrangians in complex cotangent bundles relevant to the work of \cite{Kontsevich2020}, although some such examples should exist in K3 surfaces \cite{Bousseau2022}. Explicit computations of this operator for interesting complex Lagrangians would be significant progress towards novel computations of the Fueter $2$-category.
\end{remark}

\paragraph{Categorical Hochschild homology.} 
In this section, we explain the meaning of the invariant $\overline{HH}_*(\mathrm{Fuet}_M)$, which should be connected to the category $Fuet_M(S^1)$ (\autoref{eq:loop-space-complex-functional}, \autoref{eq:categorical-open-closed-map})  that the Fueter TQFT assigns to a circle. 
Recall that the Fukaya category $\mathrm{Fuk}(X)$ is really an $A_\infty$-category, which, from the perspective of homotopy theory, is a particular model for a stable $(\infty, 1)$-category \cite{LurieHA}. A stable $(\infty, 1)$-category has a monoidal $(\infty,1)$-category of endomorphisms $\mathrm{End}(\mathcal{F}(X))$, and a unit object represented by the identity functor $\mathbf{1}: \mathcal{F}(X) \to \mathcal{F}(X)$. The Hochschild homology $HH_*(\mathcal{F}(X))$ is then simply the homotopy groups of the two-sided self-tensor product of the identity $\mathbf{1}$, thought of as an object of $\mathrm{End}(\mathcal{F}(X))$.

Correspondingly, the Fueter $2$-category described in \autoref{sec:fueter-tqft} should be the cohomology category of a sort of $(A_\infty, 2)$-category defined by counting Fueter maps from various domains. It is not clear what is the precise algebraic model of an $(A_\infty,2)$-category is appropriate for this setting; we discuss this problem in the next section. Nonetheless, whatever the appropriate model is, a basic desideratum is that the data of the Fueter $2$-category should give rise, from the perspective of homotopy theory, to an $(\infty, 2)$-category, which we will also denote by $\mathrm{Fuet}_M$ for convenience. An $(\infty, 2)$-category is meant to be an object where where the morphisms between two objects form an $(\infty, 1)$-category.  Certain foundations for $(\infty, 2)$-categories have been laid by Gaitsgory--Rozenblym \cite{gaitsgory-rozenblym}. In particular, given an $(\infty, 2)$-category $\mathcal{C}$ such as the one associated to the Fueter $2$-category, there is an $(\infty,2)$-category of endomorphisms $\mathrm{End}(\mathcal{C})$, and a unit object $\mathbf{1}_{\mathcal{C}}: \mathcal{C} \to \mathcal{C}$ corresponding to the identity functor. Using this machinery, we get a rather inexplicit definition of categorical Hochschild homology as
\begin{equation}
\label{eq:categorical-hochscild-homology-def}
    \overline{HH}_*(\mathrm{Fuet}_M) := \mathbf{1}_{\mathrm{Fuet}_M} \tensor_{(Fuet_M \tensor Fuet_M^{op})} \mathbf{1}_{\mathrm{Fuet}_M}. 
\end{equation}
A more convenient model for $\overline{HH}_*(\mathrm{Fuet}_M)$, analogous to the cyclic bar complex of an $A_\infty$-category, is needed for the definition of the open-closed map \eqref{eq:categorical-open-closed-map}; we state this as a problem below.

We can take a different perspective to link up with a topic of interest in representation theory. Assume that we have a collection of generating objects $L_0, \ldots, L_k$ of $\mathrm{Fuet}_M$, meaning that $\mathrm{Fuet}_M$ embeds faithfully into the category of modules over the category $\mathrm{End}(L_0 \oplus \ldots \oplus L_k)$ via the Yoneda embedding. Now $End(L_0 \oplus \ldots \oplus L_k)$ is going to be a \emph{monoidal} $A_\infty$ category, where the monoidal structure arises from composition of $1$-morphisms as in \autoref{sec:fueter-tqft}. We will then have that $\overline{HH}_*(\mathrm{Fuet}_M)$ is the \emph{categorical trace} of the identity functor on this monoidal $1$-category \cite[Chapter~3]{gaitsgory-on-traces}. Thus, one can use ansatzes motivated by representation theory to make guesses as to the nature of $\mathrm{Fuet}_M(S^1)$ for $M$ of representation-theoretic interest \cite{gammage2022betti}.

\subsection{Combinatorics of Fueter domains}
\label{sec:concrete-2-category-theory}

While abstract foundations for $(\infty,2)$-category have been systematically explored, they are not sufficient for development of the Fueter $2$-category. Indeed, symplectic topology depends significantly on the identification of the associahedron with a connected component the real locus in the moduli space of stable spheres, which underlies the combinatorics of $A_\infty$-categories and makes it possible to write down many basic operations such as the open-closed map. 

A higher-dimensional generalization of the $A_\infty$ combinatorics is required for the full formulation of the Fueter $2$-category. One basic combinatorial model for an $(\infty, 2)$-category which generalizes  is given by the $2$-associahedra, which underly the theory of $(A_\infty, 2)$-categories \cite{bottman-carmel}.  These were developed with the aim of piecing together the Fukaya $A_\infty$ categories of all sympletic manifolds into an $(\infty, 2)$-categorical object. It is an interesting question to the authors whether the $2$-associahedra can be arranged to control the boundary conditions in such a way that one can give a formal construction of the $(\infty, 2)$ category $\mathrm{Fuet}_M$ given a plausible compactification theory for the Fueter equations. 

A concrete manifestation of this problem is a consequence of the fact that the counts of holomorphic triangles defining composition of $1$-morphisms in the Fueter $2$-category may give a \emph{sum} of $1$-morphisms of $\mathrm{Fuet}_M$ as an output. In particular, it is possible that the composition of $1$-morphisms may only be defined on the total Fukaya Seidel categories
\begin{equation}
    tw \mathrm{FS}^{\zeta, \to}_{\hat{\theta}}(\cA^\C_{L_0,L_1}) \tensor tw \mathrm{FS}^{\zeta, \to}_{\hat{\theta}}(\mathcal{A}^\C_{L_1, L_2}) \to tw \mathrm{FS}^{\zeta, \to}_{\hat{\theta}}(\mathcal{A}^\C_{L_0, L_2}),
\end{equation}
where we recall that $tw \mathrm{FS}^{\zeta, \to}_{\hat{\theta}}(\mathcal{A}^\C_{L_i, L_j})= \mathrm{FS}^\zeta_{\hat{\theta}}(\mathcal{A}^\C_{L_i, L_j}) = \mathrm{Fuet}_M(L_i, L_j)$ by definition, and
 may not be induced by a functor on the directed Fukaya-Seidel categories 
\begin{equation}
    \mathrm{FS}^{\zeta, \to}_{\hat{\theta}}(\cA^\C_{L_0,L_1}) \tensor \mathrm{FS}^{\zeta, \to}_{\hat{\theta}}(\mathcal{A}^\C_{L_1, L_2}) \to \mathrm{FS}^{\zeta, \to}_{\hat{\theta}}(\mathcal{A}^\C_{L_0, L_2}).
\end{equation}

Indeed, the correct definition of composition may take as input a pair $p \in L_0 \cap L_1$, $q \in L_1 \cap L_2$, and produce a \emph{twisted complex of intersection points} (see \cite{Seidel2008} for algebraic background)
\begin{equation*}
    m_2(p, q) \in tw \mathrm{FS}^{\zeta, \to}_{\hat{\theta}}(\mathcal{A}_{L_0, L_2}^\C) = \mathrm{Fuet}_M(L_0, L_2).
\end{equation*} 
This possibility was pointed out to us by Ahsan Khan \cite{ahsan-khan}, and is explored in that forthcoming work. See \autoref{fig:twisted-complexes} for a diagram of the domain of a Fueter map producing the Maurer--Cartan datum defining the twisted complex $m_2(p, q)$, as well as comments as to why such putative Maurer--Cartan data may in fact be forced to be zero. Regardless, clarifying the combinatorial and higher-categorical operadic structures associated to the Fueter equations is an exciting problem, which in large part belongs to topology, geometry, and algebra rather than to geometric analysis, and thus may be interesting and accessible to mathematicians working in other disciplines.

\begin{figure}
    \centering
    \includegraphics[width=8cm]{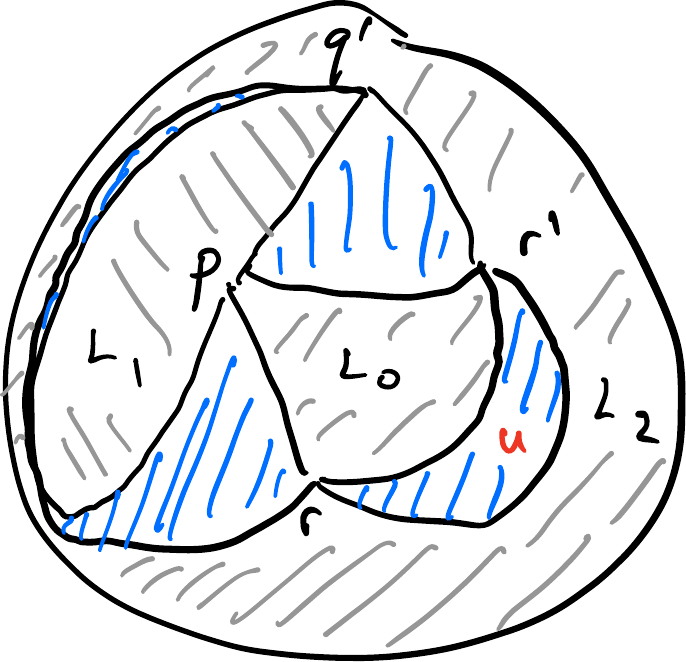}
    \caption{\small This diagram was explained to the authors by Ahsan Khan \cite{ahsan-khan}. It shows asymptotic conditions drawn on $S^2 = \partial B^3$ for a Fueter map which may give rise to a Maurer--Cartan datum given by a count of $2$-morphisms $u \in Fuet(L_0, L_2)(r, r')$ for the twisted complex of intersection points $r \in L_0 \cap L_2$ corresponding to the composition of the $1$-morphisms $p \in Fuet(L_0, L_1)$ and $q' \in Fuet(L_1, L_2)$. The counts of holomorphic triangles $C_{pq'}^r$, as per Section \ref{sec:fueter-tqft}, define the untwisted object $\mathcal{O} = \oplus_r C_{pq'}^r r$, and the above domains define morphisms $r \to r'$ which can be used to deform the differential on $\mathcal{O}$ appropriately. Note however that if each of $(L_0, L_1, L_2)$ intersect transversely then all holomorphic curves (blue) must be constant by Solomon--Verbitsky \cite{Solomon2019}, and then $r=r'$ and the Maurer--Cartan datum is zero. Whether the Maurer--Cartan data associated to such domains are ever nonzero  remains to be be clarified in general and in examples.  }
    \label{fig:twisted-complexes}
\end{figure}

\begin{problem}
Explain the precise cominatorics of surfaces and boundary conditions needed to make sense of the horizontal compositions of $2$-morphisms in the Fueter $2$-category. This likely involves the introduction of certain novel operads which are $2$-categorical counterparts to the $A_\infty$ operad. These combinatorics should be different in Gaiotto--Moore--Witten and Haydys models for the Fueter $2$-category. Can one provide a concrete, combinatorial definition of $(\infty,2)$-categories by following this perspective?
\end{problem}

\begin{problem}
Make sense of the cyclic bar complex of a $2$-category. Define the open-closed map \eqref{eq:categorical-open-closed-map} and use it to compute $\mathrm{Fuet}_M$ for interesting cases.
\end{problem}

\subsection{Examples from algebraic geometry}
\label{sec:mirror-symmetry-for-toric-varieties}

Using the conjectures of  \autoref{sec:cotangent-bundles-conjecture} together with the existing results on $2d$ mirror symmetry, one can derive predictions regarding existing invariants which may be surprising without the perspective of the Fueter TQFT. We thank Denis Auroux for initial inspiration regarding   \autoref{conj:mirror-to-addition-functor} and Justin Hilburn and Benjamin Gammage for further discussion regarding the relevant toric geometry.

For example, let us specialize \autoref{conj:categorical-floer-theorem} to $X = (\C^*)^n$. In this setting, $X$ and $T^*X$ admit flat metrics, and thus bubbling of Fueter maps is a priori excluded \cite{Hohloch2009}, and thus the conjecutres are more likely to hold without modification. The study of the Fukaya--Seidel categories of Laurent polynomials $F: (\C^*)^n \to \C$ is part of the subject of homological mirror symmetry for toric varieties, which identify these categories with categories of coherent sheaves on certain toric varieties. While there are at this point a variety of approaches to the subject \cite{abouzaid2009morse, fang2008t, hanlon2022aspects, shendetoric}, we state an early result below:

\begin{theorem}[Abouzaid \cite{abouzaid2009morse}]
\label{thm:abouzaids-toric-variety-theorem}
Let $Y$ be a smooth projective Fano toric variety, which is determined by a simplicial rational polyhedral fan $\Delta_Q$ on the faces on the convex polytope $Q \subset \R^n$, such that 
\begin{itemize}
    \item the interior of $Q$ contains $0$,
    \item the vertices of $Q$ are primitive vectors in $\Z^n$, and 
    \item the vertices on every codimension $1$-face of $Q$ are a $\Z$-basis of $\Z^n$. 
\end{itemize}   
Then $D^bCoh(Y)) \subset \mathrm{FS}((\C^*)^n,F_Y)$, where $F: (\C^*)^n \to \C$ is any generic Laurent polynomial which consists of monomials with exponents given by the the vertices of $Q$. 
\end{theorem}
\begin{remark}
Later papers \cite{hanlon2022aspects, shendetoric}, using a modified definition of $FS((\C^*)^n, F_Y)$ in terms of partially-wrapped Fukaya categories, show that the above containment is an isomorphism.
\end{remark}

One can ask what operation on toric varieties corresponds to taking the sum the sum of two such Laurent polynomials $(F_{Y_1}, F_{Y_2})$. We argue that this corresponds to superimposing the associated fans. Explicitly, let us take the product of the toric varieties $Y_1 \times Y_2$ and then taking a GIT quotient with respect to the diagonal torus action $(\C^*)^n \subset (\C^*)^n \times (\C*)^n \curvearrowright Y_1 \times Y_2$. We take this with respect to an equivariant ample line bundle $\xi = \xi_1 \boxtimes \xi_2$ where $\xi_i$ is an equivariant ample line bundle on $Y_i$. The GIT quotient $Y = Y_1 \times Y_2 //_\xi (\C^*)^n$ will will have a fan given by the cones through zero on the faces of the convex hull of the union of each of the Newton polytopes associated to $F_{Y_1}$ and $F_{Y_2}$. Indeed, this follows from the fact that the dual polytope $(C_1 \cap C_2)^*$ of an intersection of convex polytopes $C_1, C_2$ is the convex hull of $C_1^* \cup C_2^*$, together with the bijection between ample equivariant bundles on a toric variety $Y_\Delta$  and polyhedra with normal fan $\Delta$,   and the description of the GIT quotient of a toric variety by a subtorus \cite{thaddeus1996geometric}. 

Thus, given a pair of smooth proper toric varieties $(Y_1, Y_2)$ with associated Laurent polynomials $(F_{Y_1}, F_{Y_2})$, we have a diagram 
\begin{equation}
    \begin{tikzcd}
     & Y := Y_1 \times Y_2 //_\xi (\C^*)^n \arrow[ld, swap, "\pi_1"] \arrow[rd, "\pi_2"] & \\
     Y_1 & & Y_2
    \end{tikzcd}
\end{equation}
We conjecture that this interesting operation may be mirror to addition of Laurent polynomials in the following sense:

\begin{conjecture}
\label{conj:mirror-to-addition-functor}
Suppose that the toric variety $Y$ as above is smooth and proper. Then the functor of  \autoref{conj:composition-in-Lefschetz-fibrations}, under homological mirror symmetry for $(Y_1, Y_2, Y)$, becomes the functor 
\begin{equation}
\label{eq:mirror-to-addition-functor}
\begin{gathered}
    D^bCoh(Y_1) \tensor D^bCoh(Y_2) \to  D^bCoh(Y) \\
    (\mathcal{L}_1, \mathcal{L}_2) \mapsto \pi_1^* \mathcal{L}_1 \tensor \pi_2^* \mathcal{L}_2 .
\end{gathered}
\end{equation}
\end{conjecture}

\begin{figure}[H]
    \centering
    \includegraphics[width=11cm]{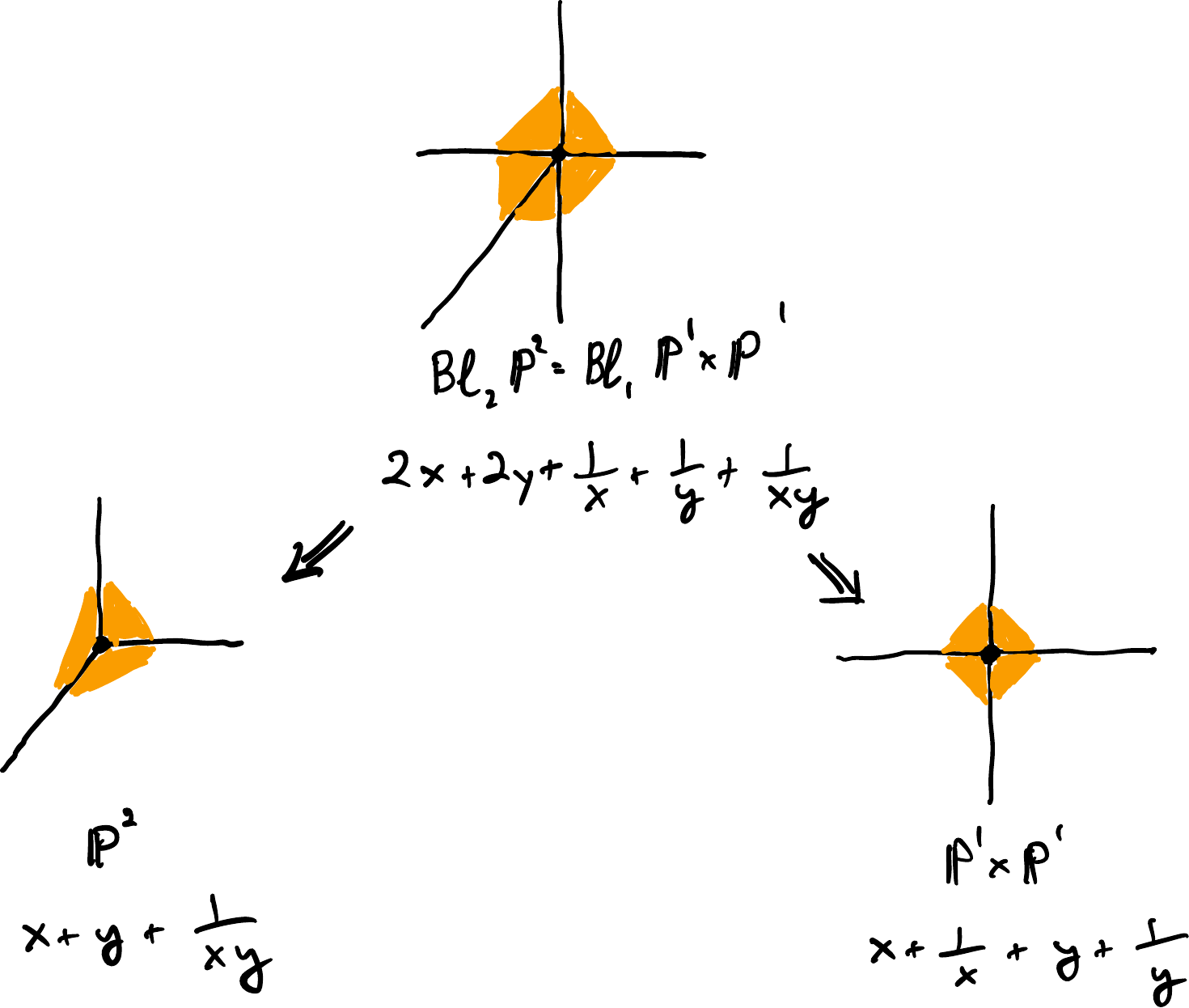}
    \caption{\small A composition functor on toric varieties inspired by Conjecture \ref{conj:composition-in-Lefschetz-fibrations}. Here, the sum of $(x+y+1/xy)$ and $x+y+1/x+1/y$ is mirror to the del Pezzo surface given by the $2$-point blow-up of $\mathbb{P}^2$, which is equal to the $1$-point blow-up of $\mathbb{P}^1 \times \mathbb{P}^1$. Addition of Laurent polynomials corresponds to the convex hull of the union of their Newton polytopes. }
    \label{fig:algebraic-geometry}
\end{figure}

It is a basic test of our understanding of the Fueter $2$-category to fully understand the composition operation on hom-categories in the simplest setting. One should be able to generalize this conjecture to cases where $Y$ is not smooth by taking a resolution of $Y$. It is interesting to consider the case where the $Y_i$ are not proper; unfortunately, while $\mathrm{FS}(Y_i, F_{Y_i})$ can be defined via sheaf-theoretic methods \cite{GPS}, there is currently no known complex Morse theory model of Fukaya--Seidel categories where the wrapping is not fully stopped. 
To get a handle on the functor of \eqref{eq:mirror-to-addition-functor}, recall that equivariant line bundles on $Y_i$ correspond to piecewise integral functions on the fan which are linear on each cone, and to forget the equivariant structure one considers such piecewise-linear functions up to globally linear functions \cite{klyachko1990equivariant}. Thus, the map \eqref{eq:mirror-to-addition-functor} simply corresponds to addition of piecewise-linear functions; the sum of two piecewise-linear functions is piecewise linear on the superimposed cones! Moreover, these line bundles are mirror to explicit Lagrangian submanifolds \cite{abouzaid2009morse}. This example can be worked out completely in the simplest nontrivial case, as in \autoref{fig:algebraic-geometry}.

\paragraph{Singularity categories.}
\label{sec:singularity-categories}
There are a number of other examples in 2-dimensional mirror symmetry beyond the case of toric mirror symmetry of  \autoref{sec:mirror-symmetry-for-toric-varieties} where multiple functions the same manifold are known the be mirror to certain objects of algebraic geometry. A natural case is the setting of $X = \C^n$, which is also flat; this includes the cases of mirror symmetry for $ADE$ singularities \cite{Khovanov2001, ebeling-singularities-survey} and Berglund--Hubsch mirror symmetry \cite{Berglund1993}. It is interesting to explore the consequences of  \autoref{conj:composition-in-Lefschetz-fibrations} on the $B$-side in these cases and to identify constructions for the putative monoidal operations via the methods of algebraic geometry. This may help clarify under what geometric conditions the composition operations of  \autoref{conj:composition-in-Lefschetz-fibrations} make sense, which should suggest the appropriate way to handle the differential geometry of non-compact complex Lagrangians in the Fueter $2$-category. 

\subsection{Further analytic aspects}
\label{sec:further-analytic-aspects}
To establish the existence categorical structures described above, a number of essential analytic aspects of the Fueter equation remain to be clarified. Excluding the well-known open problems regarding the compactness theory of the Fueter equation, we outline what we see as the basic conceptual challenges suggested by this categorical framework. These are the problems of transversality of the critical locus and Fueter trajectories, of wrapping Fueter maps, and the most general problem of the domain of definition of the Fueter $2$-category. 

\paragraph{Morse--Bott theory.}
\label{sec:morse-bott-action-functionals}
An essential challenge in setting up the Fueter $2$-category is that the holomorphic action functional $\mathcal{A}^\C_{L, L}$ is not Morse, but is instead \emph{Morse-Bott}, with critical locus equal to $L$ viewed as the set of constant paths from $L$ to $L$.  As such, one does not have access to the basic structures of the complex Morse theory model of $\mathrm{FS}(\mathcal{A}^\C_{L, L})$, which require that the holomorphic function is Morse, and a Morse--Bott analog requires significant new ideas

\begin{problem}
\label{problem:endomorphism-category}
Make sense of the monoidal endomorphism category $\mathrm{Fuet}_M(L, L)$ where $L$ is a complex Lagrangian. 
\end{problem}

In the finite dimensional case, one way to define the Fukaya--Seidel category of a Morse--Bott symplectic Lefschetz fibration $F:X \to \C$ by perturbing $F$ through maps to a nearby map $\tilde{F}$ which is Morse (if this is possible), and defining $\mathrm{FS}(X,F) := \mathrm{FS}(X,\tilde{F})$. It is crucial, to control the behavior of the perturbation $F \rightsquigarrow \tilde{F}$ at infinity, since otherwise this definition of $\mathrm{FS}(X,F)$ may be highly dependent on perturbation. For example, if the critical locus $\mathrm{Crit}(F)$ of $F$ is compact, one should use a perturbation of $F$ which is constant near infinity. If $\mathrm{Crit}(F)$ is non-compact, then there may be no preferred perturbation, and there may be multiple inequivalent ways of defining the complex Morse theory of $F$.  Indeed, in the real case of the Floer cohomology of a non-compact Lagrangian (which amounts to the real  Morse theory of a real Morse--Bott action functional with non-compact critical locus), there are essentially different variants of Lagrangian Floer cohomology, such as the wrapped \cite{Abbondandolo2005} or infinitesimally wrapped \cite{NadlerZaslow} variants.  We note that  \autoref{Thm_FloerCorrespondence} and the associated \autoref{conj:categorical-floer-theorem} suggest that in the case of non-compact complex Lagrangians $L$, perturbing $L$ through nearby exact complex Lagrangians can drastically change the behavior of the Fueter hom-category, essentially because different holomorphic functions can have very different asymptotic geometry. 

Unfortunately, even the idea of perturbing the Morse--Bott functional $\mathcal{A}^\C_{L, L}$ to a Morse action functional is non-trivial. The basic challenge is the control of the energy \eqref{eq:energy-introduction} of the associated Fueter trajectories, which are needed to have any hope of compactness of resulting moduli spaces.

\begin{problem}
\label{problem:nondeg-perturbations-of-complex-action-functional}
Find a way of perturbing $\mathcal{A}^\C_{L, L}$ through action functionals to one with nondegenerate critical points such there still exist energy bounds for the associated perturbations of Fueter trajectories. 
\end{problem}

In the setting of Lagrangian Floer cohomology, when $L$ is compact, one can perturb $\mathcal{A}_{L, L}$ to a Morse action functional by adding a Hamiltonian term to the symplectic action, or equivalently by perturbing $L$ via an exact deformation to a new Lagrangian $\tilde{L}$ which intersects $L$ transversely. Unfortunately, while it is easy to prove energy bounds for the corresponding perturbations of $\mathcal{A}^\C_{L, L}$ in the complexified setup, there are not enough of these perturbations:
\begin{itemize}
    \item A complex submanifold may not admit a deformation to a nearby transversely intersecting complex submanifold. For example, every curve on a complex symplectic surface is a complex Lagrangian. An easy computation with the adjunction formula shows that any rational curve $C$ on a K3 surface $S$ has normal bundle $N_{C/S} \simeq O(-2)$, and is thus rigid, while rational curves on $K3$ surfaces are plentiful \cite{Li2011}.
    \item Complex symplectic manifolds may have few global holomorphic functions, and the global holomorphic functions that do exist when viewed as complex Hamiltonians may be insufficient to perturb a complex Lagrangian $L$ to a transversely intersecting submanifold. A natural example is the zero-section of $T^*\mathbb{P}^1$.
    \item If such perturbations existed, one would expect that taking their real parts would give perturbations of the real symplectic action functionals $\Re(e^{-i\theta}\mathcal{A}^\C_{L, L})$. But the latter functionals have Morse homology equal to the homology of $L$, and so have critical points of odd or even index and the associated Morse complexes even have differentials. In finite dimensions, the critical points of a holomorphic Morse function are all of the same index with no differentials, and it is not clear how infinite dimensional analogs of this may be reconciled with known facts about the real Morse homology of $\Re(e^{-i\theta}\mathcal{A}^\C_{L, L})$.
\end{itemize}

An alternative approach is to follow the analogy with finite-dimensional Fukaya--Seidel categories, where one leaves the holomorphic category entirely and perturbs $F$ through symplectic rather than holomorphic maps. The difficulty in this case is that energy bounds, and the construction of symplectic perturbations of the holomorphic action functional $\cA_\C$ become more challenging. One may alternatively attempt to perturb $\cA_\C$ to an $S^1$-family of real  functionals  $\widetilde{\mathcal{A}}^{\theta}_{L_0, L_1}$ which are Morsifications of $\Re(e^{-i\theta}, \cA_\C)$; such perturbations are straightforward to construct, but modifying the categorical framework to work with such perturbations is nontrivial. 

Alternatively, it may be possible to define endomorphism hom-categories without resorting to perturbations of $\mathcal{A}^\C_{L, L}$. To this end, we note that there is a well known ansatz for Fukaya--Seidel categories with Morse--Bott critical locus. Namely, the critical locus of a holomorphic function on a K\"ahler manifold is itself a symplectic manifold. Thus, especially in the setting where all manifolds in sight are exact, one makes the ansatz that 
\begin{equation}
    \label{eq:A-side-knorrer}
    \mathrm{FS}(X,F) = \mathrm{Fuk}(\mathrm{Crit}(W))
\end{equation}
for an appropriate version of the Fukaya category of the critical locus. This is meant to be an $A$-side version of a variant of the $B$-side phenomenon of \emph{Knorrer periodicity}, which posits that for certain Morse--Bott algebraic functions $F$, one has a correponding identity for the category of matrix factorizations \cite{TelemanMorseBott}. 
\begin{equation}
\label{eq:knorrer-periodicity}
    \mathrm{MF}(F) \simeq D^b_{coh}(\mathrm{Crit}(F))
\end{equation}
On objects, the map \eqref{eq:A-side-knorrer} is meant to be to given by taking the Lagrangian thimble on a Lagrangian in the critical locus. It is folklore that the topology of the normal bundle to $\mathrm{Crit}(F)$ in $F$ can obstruct the definition of this map for reasons of grading; this is analogous to the $B$-side phenomenon that algebraic triviality of the normal bundle to $\mathrm{Crit}(F)$ can invalidate the putative isomorphism \eqref{eq:knorrer-periodicity} \cite{TelemanMorseBott}. Under a condition on the normal bundle and other geometric conditions, a version of \eqref{eq:A-side-knorrer} was announced in \cite{Abouzaid-Ganatra}. 

It seems unlikely that one can always take $\mathrm{Fuet}_M(L, L) = \mathrm{Fuk}(L)$ in any reasonable sense. Indeed, any K\"ahler manifold $X$ is a complex Lagrangian $T^*X$. Thus, applying this ansatz $\mathrm{Fuk}(X) = \mathrm{Fuet}_{T^*X}(X, X)$, one would conclude by the discussion of \autoref{sec:cotangent-bundles-conjecture} that the Fukaya category of any K\"ahler manifold admits a preferred monoidal structure. But this strongly contradicts expectations from $2d$ mirror symmetry. Indeed, a symplectic manifold $X$ may have many mirror partners, which loosely speaking should correspond to different choices of the SYZ fibration on $X$ \cite{kontsevich2000homological}. Moreover, different SYZ fibrations on $X$ are associated to different choices of monoidal structure on $\mathrm{Fuk}(X)$ via adddition in the fibers \cite{bottman-thesis}. Finally, given $D^bCoh(\check{X})$ as a triangulated category, we may not be able to reconstruct $\check{X}$ due to non-trivial derived equivalences; but specifying a monoidal structure on $D^bCoh(\check{X})$ uniquely determines $\check{X}$ via its Balmer spectrum \cite{Balmer2005}. Thus, the Fukaya category seems to be essentially not monoidal, and different monoidal structures on $\mathrm{Fuk}(X)$ single out preferred 2d mirror partners $\check{X}$ to $X$.

Nonetheless, in some cases an algebraic variety determined by its derived category purely as a triangulated category, without specifying to a monoidal structure \cite{Orlov04}. To that end, we ask:

\begin{problem}
    Under what conditions does the $3d$ $A$-model endow $\mathrm{Fuk}(X)$ with a monoidal structure by viewing $\mathrm{Fuk}(X)$ as $\mathrm{Fuet}_{T^*X}(X,X)$?
\end{problem}

By analogy to the finite-dimensional setting, the basic condition to study may be the topology and index theory of the normal bundle to $X$ inside $\mathcal{P}_{X,X}T^*X$. 

\paragraph{Wrapping.}
In the context of Lagrangian Floer cohomology of non-compact Lagrangians in non-compact symplectic manifolds, there is a well-developed theory of \emph{wrapping} which is essential for computations and examples in symplectic topology. Just as the Morse homology of a Morse function $f: Q \to \R$ on a non-compact manifold $Q$ depends on the geometry of $f$ near the ends of $Q$, the Floer cohomology of the Hamiltonian-perturbed symplectic action functionals
\begin{equation}
     \cA_{L_0, L_1}(\gamma) = \int_\gamma \lambda - h_1(\gamma(1)) + h_0(\gamma(0)) - \int_0^1 H(\gamma(\tau))\rd \tau 
\end{equation}
depends on the dynamics of the perturbing real Hamiltonian $H: M \to \R$ near the ends of $M$, although it is invariant under compact perturbations of $H$. By choosing the perturbing Hamiltonian be the same for all pairs of Lagrangians in the Fukaya category and prescribing certain kinds of dynamics at infinity, one recovers the theory of \emph{partially wrapped Fukaya categories} \cite{Sylvan2019}, which includes the theory of Fukaya-Seidel categories as a special case. There are powerful, sheaf-theoretic computational methods for partially-wrapped Fukaya categories \cite{GPS} which lead to examples of importance in $2d$ mirror symmetry  \cite{gammage-shende}. 

In the case of holomorphic symplectic geometry, a similar complex Hamiltonian deformation of the holomorphic symplectic action functional is possible. Namely, one can modify the holomorphic symplectic action functional to \eqref{eq:exact-holomorphic-symplectic-action-functional} by the addition of a Hamiltonian term $H: M \to \C$ via
\begin{equation}
\label{eq:wrapped-complex-action-functional}
    \mathcal{A}^\C_{L_0, L_1; H}(\gamma) = \int_{\gamma} \Lambda - H_0(\gamma_0) + H_1(\gamma_1) - \int_0^1 H(\gamma(\tau)) \rd\tau.
\end{equation}
The real parts of this functional $\Re(e^{-i\theta}\mathcal{A}^\C_{L_0, L_1; H})$ are then the deformations by the real Hamiltonians $\Re(e^{-i \theta} H)$ of $\mathcal{A}_\theta$. Thus, all these real functionals have the same critical points: indeed, this is because the Hamiltonian flows of $\Re(e^{-i \theta} H)$ with respect to $\omega_{\theta}$ for all $\theta$ agree the holomorphic Hamiltonian flow of $H$. 

Complex gradient trajectories of the perturbed functionals $\mathcal{A}^\C_{L_0, L_1; H}$ give a natural class of perturbations of the Fueter equation which admit topological energy identities.  Indeed the Fueter hom-category $\mathrm{Fuet}_{T^*X}(L_0, L_1)$ of  \autoref{conj:categorical-floer-theorem} is simply the Fukaya--Seidel category of the functional $\mathcal{A}^\C_{L_0, L_0}$, where $L_0 \subset T^*X$ is the zero section, perturbed by the pullback of the complex Hamiltonian $F: X \to \C$ to the cotangent bundle. We implicitly use this fact and the associated perturbations of the Fueter equations repeatedly throughout this paper. 

One should have ``wrapped'' versions of the Fueter $2$-category $\mathrm{Fuet}_M$ where we perturb \emph{all} functionals $\cA_\C$ by some fixed background holomorphic Hamiltonian $H: M \to \C$. Given an appropriate geometry of $H$ at infinity in $M$ (where we note that interesting $H$ will not exist for compact $M$), this wraped Fueter $2$-category should categorify the corresponding Fukaya--Seidel category $\mathrm{FS}(M, H)$. This may produce new examples of tractable Fueter $2$-categories, as such examples are often connected to representation theory. One may even speculate about complexifications of the Abbondondalo--Schwartz isomorphism for complex cotangent bundles \cite{Abbondandolo2005}. We note further that flowing by a holomorphic Hamiltonian tends to make intersections between complex Lagrangians transverse, which may partially alleviate the transversality challenges of the previous section in important examples. 

\paragraph{Transversality and non-integrability.}
Related to the transversality of the critical locus, discussed in \autoref{sec:morse-bott-action-functionals}, one wishes to achieve transversality for solutions the Fueter equation in order to define the the category itself, e.g. the coefficients of the differential in \eqref{eq:fueter-differential}.

\begin{problem}
Develop a basic transversality theory for the Fueter equation by analogy to \cite{McDuff2012}. Find natural classes of perturbations of the Fueter equation that still admit energy bounds. 
\end{problem}

The simplest class of such perturbations is given by wrapping Hamiltonians as above, but these do not often exist on general hyperk\"ahler manifolds; thus, this problem is related to  \autoref{problem:nondeg-perturbations-of-complex-action-functional} above. Another class of perturbations of the Fueter equation, analogous to perturbing an integrable complex structure on a K\"ahler manifold $(X, J, \omega)$ to an $\omega$-tamed almost-complex structure $\tilde{J}$, can be obtained by perturbing the almost-quaternionic triples $(I, J, K)$. 

\begin{problem}
Find weaker conditions besides the condition that $(I, J, K)$ are a part of a hyperk\"ahler structure which allow one to make sense of the Fueter $2$-category $\mathrm{Fuet}_M$. 
\end{problem}

Our paper makes some contribution to this problem. In \autoref{sec:taming-triples}), we show that there is a notion of taming for triples $(I, J, K)$ for which energy bounds for Fueter maps without boundary still exist; and that there is a another, rather restrictive notion, which allows for energy bounds when the boundary conditions are complex Lagrangians.
We also establish a convexity theory for almost quaternionic triples which allows for a flexible construction of tame triples, see  \autoref{Sec_Convexity}). Indeed, our paper focuses on cotangent bundles of K\"ahler manifolds, which are \emph{not} complete hyperk\"ahler; yet, the Fueter equations seem to behave well in this setting, in part because cotangent bundles admit $IJK$-convex functions; see \autoref{sec:cotangent-convexity}. One can think of symplectic topology as what remains of K\"ahler geometry when the integrability condition on $J$ is dropped and one only considers $J$ up to small perturbations. By analogy, it is natural to ask what is the most general context geometric context on which the Fueter $2$-category, or its closed string sector, is sensible. The convexity theory outlines in this paper suggests that such notions may exist.
\begin{problem}
\label{problem:weak-hyperkahler-geometry}
Develop a theory of weak or almost hyperkahler geometry, for which the Fueter equations behave in tractable ways. 
\end{problem}

We suggest that a natural starting point is the following geometric setting. Define a \emph{quaternionic Weinstein domain} to be an almost quaternionic  manifold with boundary $M$ admitting a proper $IJK$-convex function $\rho: M \to (-\infty, 0]$, see  \autoref{def:ijk-convex}), such that $\rho^{-1}(0) = \partial M$ is a regular level set of $\rho$. Such domains are a natural quaternionic analog of the Weinstein domains central to symplectic geometry \cite{GPS}, with both the intrinsic geometry of quaternionic Weinstein domains $M$ and the analytic properties of Fueter maps into such domains deserving of further study.

\subsection{Further geometric aspects}
\label{sec:further-geometric-aspects}

We hope that the previous sections have convinced the reader that the study of categorical aspects of the $3d$ Fueter equation is an exciting problem at the intersection of mathematical physics, differential geometry, algebra, and topology. While the definition of the Fueter $2$-category, not to mention its computation, poses significant challenges, connections to other perspectives on the $3d$ $N=4$ $\sigma$-model should offer many examples, guiding principles, and predictions.

In this final section, we outline how in examples of interest to mathematical physics, the Fueter $2$-category should be connected to computations in more traditional enumerative geometry and representation theory. 

\paragraph{Compact examples.}
\label{sec:compact-examples}
There are not many known examples of compact hyperk\"ahler manifolds. Up to deformation equivalence, there are only two compact hyperk\"ahler $4$-manifolds, $K3$ surfaces and abelian surfaces. All known higher-dimensional compact hyperk\"ahler manifolds are quotients by finite groups of products of manifolds which are   deformation equivalent to: Hilbert schemes of points on $K3$ surfaces or abelian surfaces,  higher-dimensional abelian varieties, or one of two exceptional examples of O'Grady \cite{Huybrechts2011}. As such, it is challenging to find examples of compact hyperk\"ahler manifolds for which one can study the Fueter equation and $3d$ mirror symmetry.

Nonetheless, the case of abelian surfaces $M$ may prove to be tractable and interesting. Indeed, if $M$ is an abelian surface, then the Fukaya category of $M$ is known in fairly explicit form \cite{fukaya2002mirror}, which is a significant first step towards the computation of $\mathrm{Fuet}_M$ due to equation \eqref{eq:decategorification-1}. The computation of these Fukaya categories is possible because the all pseudoholomorphic polygons are related to the linear geometry of these tori, and the corresponding counts of polygons are given by elaborate combinatorial formulae involving theta functions. It may be the case that just like these holomorphic polygons,  the Fueter maps relevant to $\mathrm{Fuet}_M$ are also computed by counting certain polyhedra in the universal covers of these tori. Since in this setting one can exclude all possible bubbling for the Fueter equations \cite{Hohloch2009}, the case of tori may be comparatively tractable.  

The case of $K3$ surfaces is related to Donaldson--Thomas theory and is the subject of upcoming work by Bousseau \cite{Bousseau2022}. It is possible that satisfactory solutions to  \autoref{problem:weak-hyperkahler-geometry} may give rise to a more flexible geometric world of compact weak or almost hyperk\"ahler manifolds.

\paragraph{Physics and representation theory.}

There are many more examples of non-compact hyperk\"ahler manifolds which should have interesting Fueter $2$-categories. The already discussed cotangent bundles of K\"ahler manifolds admit incomplete hyperk\"ahler metrics near the zero section \cite{kaledin-hyperkahler}. This paper can be seen as an exploration of this example, using the convexity theory studied in  \autoref{Sec_Convexity} to define a weak analog of a hyperk\"ahler structure on cotangent bundles, which should suffice to develop the theory discussed in this paper.
 
There is a richly studied set of non-compact complete hyperk\"ahler manifolds arising from mathematical physics which seem to recur in gauge theory and representation theory. Many of these can be fit into the framework of resolutions of symplectic singularities \cite{BPLW}, which give a series of examples of conjecturally $3d$ mirror manifolds. These manifolds have a beautiful three-Sasakian geometry at infinity \cite{Boyer2007}, as well as certain complex Lagrangian skeleta which control their Fukaya categories and are intimately related to categories of interest to representation theory \cite{mcbreen2018homological, gammage2019homological}. Another framework understanding these examples in the context of $3d$ mirror symmetry has been proposed by Teleman \cite{teleman-icm}. 

We cannot do any justice in this review section to the ideas of $3d$ mirror symmetry. We note only that studying wrapped variants of Fueter $2$-categories of these examples should give rise to categorifications of the corresponding partially wrapped Fukaya catgories, which by \cite{GPS} are equivalent to categories of sheaves. These, in turn, are known to be related to problems of representation theory. As such, this is a setting where there are very helpful physical and representation-theoretic predictions for Fueter $2$-categories. A primary challenge in this setting is to relate the existing representation-theoretic examples to Fukaya-Seidel categories of explicit holomorphic functions, as opposed to the more general framework of partially wrapped Fukaya categories and constructible sheaves, where they naturally reside. Studying the Fueter $2$-category with wrapping given by the given holomorphic function as a holomorphic Hamiltonian should then categorify the given representation-theoretic invariant.  A natural and tantalizing question, one of many, is the following.
\begin{problem}
Using the Fueter $2$-category, give a geometric categorification of categories $\mathcal{O}$, or one of its other variants, by making sense of the appropriate wrapping on $T^*G/B$, where $G$ is a reductive algebraic group and $B$ is the Borel such that the naive decategorification of the wrapped variant of $Fuet_{T^*G/B}$ is the category of constructible sheaves on the Bruhat stratification which is equivalent to category $\mathcal{O}^{\mathfrak{b}}$ defined in \cite{beilinson1996koszul}. 
\end{problem}

\paragraph{Complex Atiyah--Floer conjecture. } 
Another source of interesting complex Lagrangian submanifolds is given by holomorphic Chern--Simons theory. Given a $3$ manifold $Y$ and a Heegaard splitting into two handlebodies $Y = H_1 \cup_\Sigma H_2$, the complex flat connections on $\Sigma$ which extend to either $H_1$ or $H_2$ give a pair of intersecting complex Lagrangians $(L_0, L_1)$ in the complex character variety $\cM_\Sigma$ of $\Sigma$. By modifying this construction, one can desingularize the complex character variety such that the resulting Lagrangians support perverse sheaves as in \autoref{Sec_RecoveringFukaya} which have hypercohomology independent of the choice of Heegaard splitting, giving rise to a well defined invariant of $Y$ \cite{abouzaid2020sheaf}. This invariant should agree with the putative Floer homology of the holomorphic Chern-Simons functional, related to recent work of Taubes \cite{Taubes2013}.  By a complex and categorical version of the Atiyah--Floer conjecture, the Fukaya--Seidel category of the holomorphic Chern--Simons functional should be equivalent to $\mathrm{Fuet}_{\cM_\Sigma}(L_0, L_1)$; this is supported by a formal adiabatic limit computation \cite[Remark~4.5]{Haydys2015}.

\paragraph{Geometric Langlands program.} Important examples of non-compact complete hyperk\"ahler manifolds are the Hitchin moduli spaces \cite{hitchin1987self} and their variants. These moduli spaces have interesting singular complex Lagrangians associated to representation theory: the components of the nilpotent cone \cite{ben2018betti}. Since the moduli space of solutions to Hitchin's equations on a surface $\Sigma$ agrees with the character variety $\cM_\Sigma$ discussed in the previous paragraph, the Lagrangians $L_0, L_1$ associated to Heegaard splittings also sit inside these spaces. The  Hitchin moduli space is the base of the Hitchin fibration, whose large part resembles the cotangent bundle of the moduli space of stable bundles on $\Sigma$. Thus, a version of \autoref{Thm_FloerCorrespondence} may apply to this setting, which may be connected to other adiabatic limit problems in gauge theory. In a different direction, there are physical conjectures that the categories $\mathrm{Fuet}_M(S^1)$, (\autoref{eq:loop-space-complex-functional}, \autoref{eq:categorical-open-closed-map})  for certain Coloumb and Higgs branches $M$ should be related to the work of Ben--Zvi--Sakellaridis--Venkatesh \cite{bsv}.

%% file: FloerTheorem.tex
\section{Floer's  theorem revisited}
\label{Sec_FloerRevisited}

Let $\cX$ be a smooth manifold and let $\cA \colon \cX \to \R$ be a Morse function. A Riemannian metric on $\cX$ induces an almost complex structure $J$ on $T^*\cX$ compatible with the canonical symplectic form.  Floer's adiabatic limit theorem concerns solutions of the Floer equation
\begin{gather}
    u \colon \R\times[0,1] \to T^*\cX  \nonumber \\
    \label{Eq_StandardFloer}
    \partial_t u + J(u)\partial_\tau u + \nabla \cA(u) = 0
\end{gather}
with boundary on the zero section and converging to Hamiltonian chords of $\cA \colon T^*\cX\to \R$ as $t\to\pm\infty$. Provided that the $C^2$ norm of $\cA$ is sufficiently small, Floer proved that all such solutions satisfy $\partial_\tau u = 0$ and therefore correspond to gradient trajectories of $f$ contained in the zero section.

There is a different perturbation of the $J$-holomorphic map equation, whose solutions correspond to gradient trajectories of $f$ without the assumption that the perturbation is small. Let $V = \nabla \cA$ be the gradient vector field of $\cA$ on $\cX$. There is a canonical lift of $V$ to a Hamiltonian vector field $X_\cH$ on $T^*\cX$. One way to see this is to consider the embedding
\begin{equation*}
    \Diff(\cX) \to \Symp(T^*\cX)
\end{equation*}
defined by pushing-forward one-forms by diffeomorphisms. The induced homomorphism of Lie algebra associates with every vector field $V$ on $\cX$ a Hamiltonian vector field $X_\cH$ generated by the Hamiltonian
\begin{gather*}
    \cH \colon T^*X \to \R, \\
    \cH (\xi) =  \xi(V).
\end{gather*}
Note that critical points of $g$ are all contained in the zero section and correspond to critical point of $\cA$. We consider a version of the Floer equation
\begin{gather}
    \nonumber
    u \colon \R\times[0,1] \to T^*\cX \\
    \label{Eq_NonStandardFloer}
    \partial_t u + J(u)\partial_\tau u + X_\cH(u) = 0.
\end{gather}
Note that this is not a usual Hamiltonian perturbation of the pseudo-holomorphic map equation, which would be of the form \eqref{Eq_StandardFloer}. There is a version of Floer's theorem for this equation which does not rely on $\cA$ being $C^2$ small.

\begin{proposition}
    \label{Prop_NonStandardFloerCorrespondence}
    Any solution $u$ of the Floer equation \eqref{Eq_NonStandardFloer} with boundary on the zero section and converging to critical points of $\cH$ as $t\to\pm\infty$ satisfies $\partial_\tau u = 0$, is contained in the zero section, and corresponds to a gradient trajectory of $\cA$.
\end{proposition}

\begin{proof}
    Using the metric, we identify $T\cX$ with $T^*\cX$. Denote by $\pi \colon T^*\cX \to \cX$ the projection map. The Levi--Civita connection on $T\cX$ induces decomposition
    \begin{equation*}
        T(T^*\cX) = \pi^*T\cX \oplus \pi^*T\cX
    \end{equation*}
    with respect to which 
    \begin{equation*}
        X_\cH(\xi) = \nabla \cA \oplus - (\nabla^2 \cA) \xi \quad\text{for } \xi \in T^*\cX
    \end{equation*}
    where $\nabla \cA$, $\nabla^2 \cA$ are the gradient and Hessian of $\cA$ at point $\pi(\zeta)$. Therefore if we write $u$ as a pair $(v, \xi)$ where $v \colon \R\times[0,1]\to \cX$ and $\xi \in \Gamma(\R\times[0,1], v^*T\cX)$, equation \eqref{Eq_NonStandardFloer} becomes
    \begin{gather*}
        \nabla_\tau\xi + \partial_t v + \nabla \cA(v) =0, \\
        \partial_\tau v - \nabla_t \xi + (\nabla^2 \cA)\xi = 0.
    \end{gather*}
    Set $L_v = \nabla_t + \nabla^2 \cA(v)$. Applying $\nabla_\tau$ to the first equation and using the second equation yields
    \begin{equation*}
        \nabla_\tau^*\nabla_\tau\xi + L_vL_v^*\xi = 0,
    \end{equation*}
    which, after integrating by parts, shows that  $\nabla_\tau\xi=0$. The boundary conditions imply that $\xi=0$ and the proposition follows. 
\end{proof}

\begin{remark}
    Floer's original theorem can be easily recovered from \autoref{Prop_NonStandardFloerCorrespondence} using the following two observations. 
    
    First, the proof of \autoref{Prop_NonStandardFloerCorrespondence} shows that, in fact, the correspondence between solutions and gradient trajectories of $\cA$ holds for any perturbation of \eqref{Eq_NonStandardFloer} of the form
    \begin{equation*}
            \partial_t u + J(u)\partial_\tau u + X_\cH(u) + P(u) = 0,
    \end{equation*}
    as long as the perturbation term $P(u)$ is sufficiently small. For example, it suffices that it satisfies $|P(u)| \leq \epsilon\mathrm{dist}(u,\cX)$, where $\epsilon$ is small and $\mathrm{dist}(\cdot,\cX)$ denotes the distance from the zero section in $T^*\cX$.
    
    The second observation is that the difference between Floer's original equation \eqref{Eq_StandardFloer} and \eqref{Eq_NonStandardFloer} is a perturbation term $P(u)$ satisfying 
    \begin{equation*}
        |P(u)| \leq | \nabla^2 \cA(u) |\mathrm{dist}(u,\cX)
    \end{equation*}
    so, by the first observation, \autoref{Prop_NonStandardFloerCorrespondence} applies to \eqref{Eq_StandardFloer} as long as $\cA$ is $C^2$ small. Another way of seeing the relationship between \eqref{Eq_StandardFloer} and \eqref{Eq_NonStandardFloer} is that these equations would agree if and only if the function $\cF = \cH + i\cA$ on $T^*\cX$ were $J$-holomorphic. This is not the case. However, the difference $P(u)$ between these two equations can be identified with $\delbar_J \cF(u)$, whose magnitude is proportional to the distance from the zero section and the $C^2$ norm of $\cA$. 
\end{remark}

\begin{remark}
    From an infinite-dimensional viewpoint, the $J$-holomorphic strip equation is the gradient flow equation for the symplectic action functional on the space of paths $\gamma \colon [0,1] \to T^*\cX$ with boundary on the zero section. This functional is Morse--Bott with critical locus consisting of constant paths in the zero section. Floer's equation \eqref{Eq_StandardFloer} is obtained by perturbing the symplectic action functional by the function
    \begin{equation*}
        \gamma \mapsto \int_0^1 f(\gamma(\tau))\rd\tau,
    \end{equation*}
    which restricts to a Morse function on the critical locus.
    
    The non-standard equation \eqref{Eq_NonStandardFloer} has a similar interpretation, as a non-gradient perturbation of the $J$-holomorphic strip equation by a vector field induced by $X_\cH$. This vector field has the property that it is gradient-like on the space of paths and it restricts to a gradient vector field of $\cA$ along the critical locus. Therefore, it leads to well-defined Morse homology which recovers the usual intersection Floer homology of the zero section in $T^*\cX$ with itself. In general, equation \eqref{Eq_NonStandardFloer} is a well-behaved perturbation of the $J$-holomorphic strip equation with boundary on a Lagrangian $L$ in a symplectic manifold as long as the Hamiltonian function $\cH$ is constant on $L$ and its Hamiltonian vector field is gradient-like along $L$. If this is the case, modulo the usual bubbling issues, counting solutions to \eqref{Eq_NonStandardFloer} leads to an alternative definition of the intersection Floer homology of $L$ with itself. 
\end{remark}

%% file: Conventions.tex
\section{Notation and conventions}

There are many sign conventions in symplectic and hyperkähler geometry. Since signs play a crucial role in our discussion, we summarize below the notation and conventions used throughout this paper. 

\subsection{Coordinates}

The standard coordinates on $[0,1]$ and $\R^2$ will be denoted by $\tau$ and $(s,t)$ respectively, so that with natural orientations $[0,1]\times\R^2$ has coordinates $(\tau,s,t)$ and $\R\times[0,1]$ has coordinates $(t,\tau)$.  We identify $\R^2$ with $\C$ using the holomorphic coordinate $z = s+it$, that is: the complex structure is given by $i\partial_s = \partial_t$. Similarly, we will equip $\R\times[0,1]$ with the complex structure given by $i\partial_t = \partial_\tau$. 

\subsection{Symplectic geometry}
Let $(M,\omega)$ be a symplectic manifold. Given a function $H: M \to \R$, the Hamiltonian vector field $X_H$ is given by 
\begin{equation*}
    \omega(\cdot, X_H) = \rd H
\end{equation*}
A compatible almost complex structure $J$ on $M$ defines a metric $g$ via 
\begin{equation*}
    g(v, w) = \omega(v, Jw)
\end{equation*}
Thus, $X_H = J\nabla H$ where $\nabla H$ is the metric dual of $\rd H$; indeed,
\[ \omega(v, J\nabla H) = g(v, \nabla H) = \rd H(v). \]

Suppose now that $M = T^*X$ is the cotangent bundle of a smooth manifold $X$. Let $\pi \colon T^*X \to X$ be the projection map. The tautological $1$-form $\lambda$ on $M$ is defined by
\begin{equation*}
    \lambda_\xi = \pi^*\xi \quad\text{for } \xi \in T^*X. 
\end{equation*}
If $q_i$ are local coordinates on  $X$ and $p_i$ are the dual coordinates on $T^*X$, then
\begin{equation*}
    \lambda = \sum_i p_i \rd q_i. 
\end{equation*}
The canonical symplectic form is 
\begin{equation*}
    \omega = \rd\lambda = \sum_i \rd p_i \wedge \rd q_i.
\end{equation*}
These conventions reproduce Hamilton's equations from classical mechanics. If $M = T^*\R^n$, then the Hamiltonian flow equation is
\begin{equation*}
    \frac{\rd q_i}{\rd t} = \frac{\partial H}{\partial p_i}, \quad
    \frac{\rd p_i}{\rd t} = -\frac{\partial H}{\partial q_i}.
\end{equation*}
Note that in this example, the standard almost complex structure on $T^*\R^n$ is $J\partial_{p_i} = \partial_{q_i}$, so $T^*\R^n$ is identified with $\C^n = \R^n \oplus i\R^n$ with the opposite complex structure.

\subsection{Lagrangian Floer theory}
\label{sec:conventions_lagrangian_floer_theory}

Let $(M,\omega)$ is a symplectic manifold as before and let $L_0,L_1 \subset M$ be Lagrangian submanifolds. Suppose for simplicity that the symplectic form $\omega = \rd\lambda$ on $M$ is exact and $L_0,L_1$ are exact in the sense that there are functions $h_i \colon L_i \to \R$ such that $\lambda|_{L_i} = \rd h_i$. Let $\cP = \cP(M,L_0,L_1)$ be the space of paths $\gamma \colon [0,1] \to M$ starting on $L_0$ and ending on $L_1$. The symplectic action functional perturbed by a Hamiltonian $H \colon M \to \R$ is given by
\begin{gather*}
    \cA \colon \cP \to \R, \\
    \cA(\gamma) = \int_\gamma \lambda - h_1(\gamma(1)) + h_0(\gamma(0)) - \int_0^1 H(\gamma(\tau))\rd \tau. 
\end{gather*}
The differential of $\cA$ at $\gamma \in \cP$ is
\begin{equation*}
    \rd\cA_\gamma(v) = \int_0^1 \omega\left(v, \frac{\rd\gamma}{\rd\tau} - X_H(\gamma) \right) \rd \tau \quad\text{for } v \in \Gamma([0,1],\gamma^*TM).
\end{equation*}
Therefore, the critical points of $\cA$ are Hamiltonian chords starting on $L_0$ and ending on $L_1$. Let $J$ be an almost complex structure on $M$ compatible with $\omega$. The gradient of $\cA$ with respect to the Riemannian metric induced by $J$ is
\begin{equation*}
    \nabla\cA(\gamma) = -J(\gamma)\frac{\rd \gamma}{\rd\tau} - \nabla H(\gamma).
\end{equation*}
The gradient flow equation for $\cA$ is therefore the Floer equation:
\begin{gather*}
    u \colon \R\times[0,1] \to M, \\
    \partial_t u + J(u)\partial_\tau u + \nabla H(u) = 0, 
\end{gather*}
or, equivalently,
\begin{equation*}
    \partial_t u + J(u)(\partial_\tau u - X_H(u)) = 0.
\end{equation*}

\subsection{Complex gradient flow}

Let $(X,I)$ be an almost complex manifold. Let $F \colon X \to \C$ be an $I$-holomorphic function.  The complex gradient flow equation is
\begin{gather*}
    u \colon \C \to X, \\
    \delbar_I u  - i\nabla F(u) = 0,
\end{gather*}
where $\delbar_I u$ is the $I$-antilinear part of the differential $\rd u$, which, after trivializing the bundle of $(0,1)$ form on $\C$, we can identify with a $(0,1)$ vector field on $X$. Here $\nabla F$ is the complex vector field of type $(0,1)$ dual to the $(1,0)$ form $\rd F$. If $F = f+ig$, with $f$ and $g$ real, then
\begin{equation*}
    \nabla F = \nabla f + i \nabla g.
\end{equation*}
In coordinates $(s,t)$ on $\C$,  interpreting $\delbar_I u$ as a $(0,1)$ vector field:
\begin{equation*}
    \delbar_I u = (\partial_s u + I(u) \partial_t u) + i(I(u) \partial_s u - \partial_t u)
\end{equation*}
so the complex gradient flow equation is
\begin{equation*}
    \partial_s u + I(u)\partial_t u + \nabla g(u) = 0.
\end{equation*}
If $X$ admits a symplectic from compatible with $I$, this is the Floer equation with the Hamiltonian perturbation by $g$. Let $X_g$ be the Hamiltonian vector field. With our conventions, we have $X_g = I \nabla g = -\nabla f$, so the Hamiltonian trajectories of $g$ are the gradient trajectories of $-f$. As $s\to\pm\infty$, we require that $u$ converges to such trajectories. 

Alternatively, we could consider the conjugate equation
\begin{equation*}
    \partial_I u + i\nabla \overline F(u) = 0.
\end{equation*}
In coordinates,  
\begin{equation*}
    \partial_s u - I(u)\partial_t u - \nabla g(u) =0,
\end{equation*}
or, equivalently,
\begin{equation*}
    \partial_s u - I(u)(\partial_t u + \nabla f(u)) = 0.
\end{equation*}
The antiholomorphic map $(s,t) \mapsto (-s,t)$ induces a bijection between the sets of solutions to the gradient flow equation and its conjugate. 

The complex gradient equation has a $\C^*$ symmetry in the following sense. For $\lambda \in \C^*$ define $u_\lambda(z) = u(\lambda z)$; then $u$ is a complex gradient trajectory of $F$ if and only if $u_\lambda$ is a complex gradient trajectory of $F_\lambda = \lambda F$:
\begin{equation*}
    \delbar_I u_\lambda - i\nabla F_\lambda (u_\lambda) = 0.
\end{equation*} 
Similarly, $u$ is an anti-complex gradient trajectory of $F$ if and only if $u_\lambda$ is an anti-complex gradient trajectory of $F_{\bar\lambda}$.

\subsection{Holomorphic Floer theory}

Let $(M,I,J,K)$ be a hyperk\"ahler manifold with the corresponding symplectic forms $\omega_I,\omega_J,\omega_K$. Let $L_0,L_1 \subset M$ be submanifolds which are complex with respect to $I$ and Lagrangian with respect to $\omega_J$ and $\omega_K$. For simplicity, assume that all three symplectic forms are exact and that $L_0, L_1$ are exact Lagrangians.

Let $\cP = \cP(M, L_0,L_1)$ be the space of paths in $M$ starting on $L_0$ and ending on $L_1$. This is an $I$-complex manifold. Let $\cA_J$ and $\cA_K$ be the symplectic action functionals with respect to $\omega_J$ and $\omega_K$. The function
\begin{equation*}
    \sA = \cA_J + i\cA_K \colon \cP \to \C
\end{equation*}
is $I$-holomorphic. The anti-complex gradient flow equation for maps $U \colon \C \to \cP$
\begin{equation*}
    \partial_I U + i\nabla \overline \sA(U) = 0
\end{equation*}
is 
\begin{equation*}
   \partial_s U - I(U)(\partial_t U + \nabla \cA_J(U)) = 0, 
\end{equation*}
which is equivalent to the Fueter equation for maps $ U \colon [0,1]\times\R^2 \to M$, 
\begin{equation*}
    \partial_s U - I(U)(\partial_t U - J(U)\partial_\tau U) =0,
\end{equation*}
or
\begin{equation*}
    I(U)\partial_\tau U + J(U)\partial_s U + K(U)\partial_t U  = 0.
\end{equation*}
As $s\to\pm\infty$ we require that $U$ converges to $J$-antiholomorphic strips from $\R\times[0,1]$ with coordinates $(t,\tau)$ to $M$. 

The Fueter equation for maps $U \colon \R^3 \to M$ has an $\SO(3)$ symmetry in the following sense. For $a \in \SO(3)$ define $U_a = U\circ a^{-1}$ and let $(I_a,J_a,K_a)$ be the quaternionic triple on $M$ obtained by rotating $(I,J,K)$ by $a$. Then the Fueter equation for $U$ with respect to $(I,J,K)$ is equivalent to the Fueter equation for $U_a$ with respect to $(I_a,J_a,K_a)$. For maps with domain $[0,1]\times \R^2$, the same holds for symmetries $a$ in the subgroup $\U(1) \subset \SO(3)$ preserving the $(s,t)$ plane in $\R^3$. Explicitly, for $\theta \in [0,2\pi)$, define $z_\theta = s_\theta + i t_\theta$ in terms of $z = s + it$ by $z_\theta = e^{-i\theta}z$ and let
\begin{gather*}
    J_\theta + i K_\theta = e^{i\theta}(J + iK).
\end{gather*}
Then, the Fueter equation in coordinates $(\tau,s_\theta,t_\theta)$ with respect to the triple $(I,J_\theta,K_\theta)$ is equivalent to the Fueter equation in coordinates $(\tau,s,t)$ with respect to the triple $(I,J,K)$.

We can perturb the holomorphic action functional $\cA_\C$ by an $I$-holomorphic function $F \colon M \to \C$:
\begin{equation*}
    \sA_F (\gamma) = \sA(\gamma) + \int_0^1 F(\gamma(\tau)) \rd\tau. 
\end{equation*}
The resulting perturbed Fueter equation is
\begin{equation*}
    \partial_s U - I(U)(\partial_t U - J(U)\partial_\tau U + \nabla f (U)) = 0,
\end{equation*}
where $f = \Re(F)$. As before, $J$ and $F$ can be replaced by $J_\theta$ and $F_\theta$ for any $\theta \in [0,2\pi)$ by writing the equation in coordinates $z_\theta = e^{-i\theta}z$.

%% file: references.bib
@article {Harvey2009,
    AUTHOR = {Harvey, F. Reese and Lawson, Jr., H. Blaine},
     TITLE = {Duality of positive currents and plurisubharmonic functions in
              calibrated geometry},
   JOURNAL = {Amer. J. Math.},
  FJOURNAL = {American Journal of Mathematics},
    VOLUME = {131},
      YEAR = {2009},
    NUMBER = {5},
     PAGES = {1211--1239},
      ISSN = {0002-9327},
   MRCLASS = {32U40 (32U15 53C38)},
  MRNUMBER = {2555839},
MRREVIEWER = {Romain Dujardin},
       DOI = {10.1353/ajm.0.0074},
       URL = {https://doi.org/10.1353/ajm.0.0074},
}

@article {Alesker2006,
    AUTHOR = {Alesker, Semyon and Verbitsky, Misha},
     TITLE = {Plurisubharmonic functions on hypercomplex manifolds and
              {HKT}-geometry},
   JOURNAL = {J. Geom. Anal.},
  FJOURNAL = {The Journal of Geometric Analysis},
    VOLUME = {16},
      YEAR = {2006},
    NUMBER = {3},
     PAGES = {375--399},
      ISSN = {1050-6926},
   MRCLASS = {32U05 (53C26)},
  MRNUMBER = {2250051},
MRREVIEWER = {Julien Keller},
       DOI = {10.1007/BF02922058},
       URL = {https://doi.org/10.1007/BF02922058},
}

@book {McDuff2012,
    AUTHOR = {McDuff, Dusa and Salamon, Dietmar},
     TITLE = {{$J$}-holomorphic curves and symplectic topology},
    SERIES = {American Mathematical Society Colloquium Publications},
    VOLUME = {52},
   EDITION = {Second},
 PUBLISHER = {American Mathematical Society, Providence, RI},
      YEAR = {2012},
     PAGES = {xiv+726},
      ISBN = {978-0-8218-8746-2},
   MRCLASS = {53D45 (32Q65 53D35)},
  MRNUMBER = {2954391},
MRREVIEWER = {Mark Alan Branson},
}

@article{anselmi-fre,
title = {Topological $\sigma$-models in four dimensions and triholomorphic maps},
journal = {Nuclear Physics B},
volume = {416},
number = {1},
pages = {255-300},
year = {1994},
issn = {0550-3213},
doi = {https://doi.org/10.1016/0550-3213(94)90585-1},
url = {https://www.sciencedirect.com/science/article/pii/0550321394905851},
author = {Damiano Anselmi and Pietro Frè},
%abstract = {It is well known that topological σ-models in two dimensions constitute a path-integral approach to the study of holomorphic maps from a Riemann surface Σ to an almost complex manifold K, the most interesting case being that were K is a Kähler manifold. We show that, in the same way, topological σ-models in four dimensions introduce a path-integral approach to the study of triholomorphic maps q: M → N between a four-dimensional riemannian manifold M and an almost quaternionic manifold N. The most interesting cases are those where M, N are hyper-Kähler or quaternionic Kähler. BRST-cohomology translates into intersection theory in the moduli-space of this new class of instantonic maps, that are named hyperinstantons by us. The definition of triholomorphicity that we propose is expressed by the equation q∗ − Ju ∘ q∗ ∘ ju = 0, where {ju, u = 1,2,3} is an almost quaternionic structure on M and {Ju, u = 1,2,3} is an almost quaternionic structure on N. This is a generalization of the Cauchy-Fueter equations. For M, N hyper-Kähler, this generalization naturally arises by obtaining the topological σ-model as a twisted version of the N = 2 globally supersymmetric σ-model. We discuss various examples of hyperinstantons, in particular on the torus and the K3 surface. We also analyze the coupling of the topological σ-model to topological gravity. The classification of triholomorphic maps and the analysis of their moduli-space is a new and fully open mathematical problem that we believe deserves the attention of both mathematicians and physicists.}
}

@article{beilinson1996koszul,
  title={Koszul duality patterns in representation theory},
  author={Beilinson, Alexander and Ginzburg, Victor and Soergel, Wolfgang},
  journal={Journal of the American Mathematical Society},
  volume={9},
  number={2},
  pages={473--527},
  year={1996}
}

@misc{bsv,
author={D. Ben-Zvi}, 
title={Quantization and Duality for Hyperspherical Varieties},
year=2022,
howpublished={Talk at ISTA. Available at \url{https://ist.ac.at/en/news-events/event/?eid=3323}.}
}

@misc{kaledin-hyperkahler,
Author = {D. Kaledin},
Title = {Hyperkaehler structures on total spaces of holomorphic cotangent bundles},
Year = {1997},
Eprint = {arXiv:alg-geom/9710026},
}

@article{Intriligator1996,
  doi = {10.1016/0370-2693(96)01088-x},
  url = {https://doi.org/10.1016/0370-2693(96)01088-x},
  year = {1996},
  month = oct,
  publisher = {Elsevier {BV}},
  volume = {387},
  number = {3},
  pages = {513--519},
  author = {K. Intriligator and N. Seiberg},
  title = {Mirror symmetry in three dimensional gauge theories},
  journal = {Physics Letters B}
}

@article{fukaya2002mirror,
  title={Mirror symmetry of abelian varieties and multi-theta functions},
  author={Fukaya, Kenji},
  journal={Journal of Algebraic Geometry},
  volume={11},
  number={3},
  pages={393--512},
  year={2002},
  publisher={Citeseer}
}

@article{ben2018betti,
  title={Betti geometric langlands},
  author={Ben-Zvi, David and Nadler, David},
  journal={Algebraic geometry: Salt Lake City 2015},
  volume={97},
  pages={3--41},
  year={2018}
}

@article{hitchin1987self,
  title={The self-duality equations on a Riemann surface},
  author={Hitchin, Nigel J},
  journal={Proceedings of the London Mathematical Society},
  volume={3},
  number={1},
  pages={59--126},
  year={1987},
  publisher={Citeseer}
}

@article{abouzaid2020sheaf,
  title={A sheaf-theoretic model for SL (2, C) Floer homology},
  author={Abouzaid, Mohammed and Manolescu, Ciprian},
  journal={Journal of the European Mathematical Society},
  volume={22},
  number={11},
  pages={3641--3695},
  year={2020}
}

@article{mcbreen2018homological,
  title={Homological mirror symmetry for hypertoric varieties I},
  author={McBreen, Michael and Webster, Ben},
  journal={arXiv preprint arXiv:1804.10646},
  year={2018}
}

@article{gammage2019homological,
  title={Homological mirror symmetry for hypertoric varieties II},
  author={Gammage, Benjamin and McBreen, Michael and Webster, Ben},
  journal={arXiv preprint arXiv:1903.07928},
  year={2019}
}

@incollection{Boyer2007,
  doi = {10.1093/acprof:oso/9780198564959.003.0014},
  url = {https://doi.org/10.1093/acprof:oso/9780198564959.003.0014},
  year = {2007},
  month = oct,
  publisher = {Oxford University Press},
  pages = {473--528},
  author = {Charles P. Boyer and Krzysztof Galicki},
  title = {3-Sasakian Manifolds},
  booktitle = {Sasakian Geometry}
}

@inproceedings{Huybrechts2011,
  doi = {10.1142/9789814324359_0059},
  url = {https://doi.org/10.1142/9789814324359_0059},
  year = {2011},
  month = jun,
  publisher = {Published by Hindustan Book Agency ({HBA}),  India. {WSPC} Distribute for All Markets Except in India},
  author = {Daniel Huybrechts},
  title = {Hyperk\"{a}hler Manifolds and Sheaves},
  booktitle = {Proceedings of the International Congress of Mathematicians 2010 ({ICM} 2010)}
}

@article{Bullimore2016,
  doi = {10.1007/jhep10(2016)108},
  url = {https://doi.org/10.1007/jhep10(2016)108},
  year = {2016},
  month = oct,
  publisher = {Springer Science and Business Media {LLC}},
  volume = {2016},
  number = {10},
  author = {Mathew Bullimore and Tudor Dimofte and Davide Gaiotto and Justin Hilburn},
  title = {Boundaries,  mirror symmetry,  and symplectic duality in 3d N = 4 gauge theory},
  journal = {Journal of High Energy Physics}
}

@misc{GMW,
Author = {Davide Gaiotto and Gregory W. Moore and Edward Witten},
Title = {Algebra of the Infrared: String Field Theoretic Structures in Massive ${\cal N}=(2,2)$ Field Theory In Two Dimensions},
Year = {2015},
Eprint = {arXiv:1506.04087},
}

@misc{bottman-carmel,
Author = {Nathaniel Bottman and Shachar Carmeli},
Title = {$(A_\infty,2)$-categories and relative 2-operads},
Year = {2018},
Eprint = {arXiv:1811.05442},
}

@misc{gammage-hilburn-mazel-gee-2022,
Author = {Benjamin Gammage and Justin Hilburn and Aaron Mazel-Gee},
Title = {Perverse schobers and 3d mirror symmetry},
Year = {2022},
Eprint = {arXiv:2202.06833},
}

@BOOK{Kock2003-lj,
  title     = "London mathematical society student texts: Frobenius algebras
               and {2-D} topological quantum field theories series number 59",
  author    = "Kock, Joachim",
  publisher = "Cambridge University Press",
  series    = "London Mathematical Society student texts",
  month     =  dec,
  year      =  2003,
  address   = "Cambridge, England",
  language  = "en"
}

@article {Fueter1935,
    AUTHOR = {Fueter, Rudolf},
     TITLE = {\"{U}ber die analytische {D}arstellung der regul\"{a}ren {F}unktionen
              einer {Q}uaternionenvariablen},
   JOURNAL = {Comment. Math. Helv.},
  FJOURNAL = {Commentarii Mathematici Helvetici},
    VOLUME = {8},
      YEAR = {1935},
    NUMBER = {1},
     PAGES = {371--378},
      ISSN = {0010-2571},
   MRCLASS = {DML},
  MRNUMBER = {1509533},
       DOI = {10.1007/BF01199562},
       URL = {https://doi.org/10.1007/BF01199562},
}

@article {Walpuski2017c,
    AUTHOR = {Walpuski, Thomas},
     TITLE = {A compactness theorem for {F}ueter sections},
   JOURNAL = {Comment. Math. Helv.},
  FJOURNAL = {Commentarii Mathematici Helvetici. A Journal of the Swiss
              Mathematical Society},
    VOLUME = {92},
      YEAR = {2017},
    NUMBER = {4},
     PAGES = {751--776},
      ISSN = {0010-2571},
   MRCLASS = {58E20 (53C26)},
  MRNUMBER = {3718486},
MRREVIEWER = {Ling He},
       DOI = {10.4171/CMH/423},
       URL = {https://doi.org/10.4171/CMH/423},
}

@article {Hohloch2009,
    AUTHOR = {Hohloch, Sonja and Noetzel, Gregor and Salamon, Dietmar A.},
     TITLE = {Hypercontact structures and {F}loer homology},
   JOURNAL = {Geom. Topol.},
  FJOURNAL = {Geometry \& Topology},
    VOLUME = {13},
      YEAR = {2009},
    NUMBER = {5},
     PAGES = {2543--2617},
      ISSN = {1465-3060},
   MRCLASS = {53D40 (53C26)},
  MRNUMBER = {2529942},
MRREVIEWER = {Hansj\"{o}rg Geiges},
       DOI = {10.2140/gt.2009.13.2543},
       URL = {https://doi.org/10.2140/gt.2009.13.2543},
}

@article {Walpuski2017a,
    AUTHOR = {Walpuski, Thomas},
     TITLE = {{$\rm Spin(7)$}-instantons, {C}ayley submanifolds and {F}ueter
              sections},
   JOURNAL = {Comm. Math. Phys.},
  FJOURNAL = {Communications in Mathematical Physics},
    VOLUME = {352},
      YEAR = {2017},
    NUMBER = {1},
     PAGES = {1--36},
      ISSN = {0010-3616},
   MRCLASS = {53C29 (53C38 70S15)},
  MRNUMBER = {3623252},
MRREVIEWER = {Selman U\u{g}uz},
       DOI = {10.1007/s00220-016-2724-6},
       URL = {https://doi.org/10.1007/s00220-016-2724-6},
}

@article {Walpuski2017b,
    AUTHOR = {Walpuski, Thomas},
     TITLE = {{$G_2$}-instantons, associative submanifolds and {F}ueter
              sections},
   JOURNAL = {Comm. Anal. Geom.},
  FJOURNAL = {Communications in Analysis and Geometry},
    VOLUME = {25},
      YEAR = {2017},
    NUMBER = {4},
     PAGES = {847--893},
      ISSN = {1019-8385},
   MRCLASS = {53C38 (53C29)},
  MRNUMBER = {3731643},
MRREVIEWER = {Kotaro Kawai},
       DOI = {10.4310/CAG.2017.v25.n4.a4},
       URL = {https://doi.org/10.4310/CAG.2017.v25.n4.a4},
}

@incollection {Donaldson2011,
    AUTHOR = {Donaldson, Simon and Segal, Ed},
     TITLE = {Gauge theory in higher dimensions, {II}},
 BOOKTITLE = {Surveys in differential geometry. {V}olume {XVI}. {G}eometry
              of special holonomy and related topics},
    SERIES = {Surv. Differ. Geom.},
    VOLUME = {16},
     PAGES = {1--41},
 PUBLISHER = {Int. Press, Somerville, MA},
      YEAR = {2011},
   MRCLASS = {53C07 (14J32 53C38 53D12 57R58 58D27)},
  MRNUMBER = {2893675},
MRREVIEWER = {Andrew Swann},
       DOI = {10.4310/SDG.2011.v16.n1.a1},
       URL = {https://doi.org/10.4310/SDG.2011.v16.n1.a1},
}

@article {Taubes2013,
    AUTHOR = {Taubes, Clifford Henry},
     TITLE = {{${\rm PSL}(2;\Bbb C)$} connections on 3-manifolds with {${\rm
              L}^2$} bounds on curvature},
   JOURNAL = {Camb. J. Math.},
  FJOURNAL = {Cambridge Journal of Mathematics},
    VOLUME = {1},
      YEAR = {2013},
    NUMBER = {2},
     PAGES = {239--397},
      ISSN = {2168-0930},
   MRCLASS = {53B15 (53C05 53C21)},
  MRNUMBER = {3272050},
MRREVIEWER = {Maria Falcitelli},
       DOI = {10.4310/CJM.2013.v1.n2.a2},
       URL = {https://doi.org/10.4310/CJM.2013.v1.n2.a2},
}

@article {Taubes2015,
    AUTHOR = {Taubes, Clifford Henry},
     TITLE = {Corrigendum to ``{$PSL(2;\Bbb C)$} connections on 3-manifolds
              with {$L^2$} bounds on curvature'' [{C}ambridge {J}ournal of
              {M}athematics 1(2013) 239--397] [ {MR}3272050]},
   JOURNAL = {Camb. J. Math.},
  FJOURNAL = {Cambridge Journal of Mathematics},
    VOLUME = {3},
      YEAR = {2015},
    NUMBER = {4},
     PAGES = {619--631},
      ISSN = {2168-0930},
   MRCLASS = {53B15 (53C05 53C21)},
  MRNUMBER = {3435274},
MRREVIEWER = {Nikolai N. Saveliev},
       DOI = {10.4310/CJM.2015.v3.n4.a2},
       URL = {https://doi.org/10.4310/CJM.2015.v3.n4.a2},
}

@misc{Taubes2016,
  doi = {10.48550/ARXIV.1610.07163},
  url = {https://arxiv.org/abs/1610.07163},
  author = {Taubes, Clifford Henry},
  title = {On the behavior of sequences of solutions to U(1) Seiberg-Witten systems in dimension 4},
  publisher = {arXiv},
  year = {2016},
  copyright = {arXiv.org perpetual, non-exclusive license}
}

@article {Haydys2015a,
    AUTHOR = {Haydys, Andriy and Walpuski, Thomas},
     TITLE = {A compactness theorem for the {S}eiberg-{W}itten equation with
              multiple spinors in dimension three},
   JOURNAL = {Geom. Funct. Anal.},
  FJOURNAL = {Geometric and Functional Analysis},
    VOLUME = {25},
      YEAR = {2015},
    NUMBER = {6},
     PAGES = {1799--1821},
      ISSN = {1016-443X},
   MRCLASS = {53C27 (57R57)},
  MRNUMBER = {3432158},
MRREVIEWER = {Fortun\'{e} Massamba},
       DOI = {10.1007/s00039-015-0346-3},
       URL = {https://doi.org/10.1007/s00039-015-0346-3},
}

@article {Walpuski2021,
    AUTHOR = {Walpuski, Thomas and Zhang, Boyu},
     TITLE = {On the compactness problem for a family of generalized
              {S}eiberg-{W}itten equations in dimension 3},
   JOURNAL = {Duke Math. J.},
  FJOURNAL = {Duke Mathematical Journal},
    VOLUME = {170},
      YEAR = {2021},
    NUMBER = {17},
     PAGES = {3891--3934},
      ISSN = {0012-7094},
   MRCLASS = {53C07 (35J70 57K31)},
  MRNUMBER = {4340726},
       DOI = {10.1215/00127094-2021-0005},
       URL = {https://doi.org/10.1215/00127094-2021-0005},
}

@article{abouzaid2009morse,
  title={Morse homology, tropical geometry, and homological mirror symmetry for toric varieties},
  author={Abouzaid, Mohammed},
  journal={Selecta Mathematica},
  volume={15},
  number={2},
  pages={189--270},
  year={2009},
  publisher={Springer}
}

@article{fang2008t,
  title={T-duality and homological mirror symmetry of toric varieties},
  author={Fang, Bohan and Liu, Chiu-Chu Melissa and Treumann, David and Zaslow, Eric},
  journal={arXiv preprint arXiv:0811.1228},
  year={2008}
}

@misc{shendetoric,
Author = {Vivek Shende},
Title = {Toric mirror symmetry revisited},
Year = {2021},
Eprint = {arXiv:2103.05386},
}

@article{klyachko1990equivariant,
  title={Equivariant bundles on toral varieties},
  author={Klyachko, Alexander A},
  journal={Mathematics of the USSR-Izvestiya},
  volume={35},
  number={2},
  pages={337},
  year={1990},
  publisher={IOP Publishing}
}

@article{thaddeus1996geometric,
  title={Geometric invariant theory and flips},
  author={Thaddeus, Michael},
  journal={Journal of the American Mathematical Society},
  volume={9},
  number={3},
  pages={691--723},
  year={1996}
}

@article{hanlon2022aspects,
  title={Aspects of functoriality in homological mirror symmetry for toric varieties},
  author={Hanlon, A and Hicks, J},
  journal={Advances in Mathematics},
  volume={401},
  pages={108317},
  year={2022},
  publisher={Elsevier}
}

@article{Li2011,
  doi = {10.1007/s00222-011-0359-y},
  url = {https://doi.org/10.1007/s00222-011-0359-y},
  year = {2011},
  month = oct,
  publisher = {Springer Science and Business Media {LLC}},
  volume = {188},
  number = {3},
  pages = {713--727},
  author = {Jun Li and Christian Liedtke},
  title = {Rational curves on K3 surfaces},
  journal = {Inventiones mathematicae}
}

@article {Floer1988,
    AUTHOR = {Floer, Andreas},
     TITLE = {The unregularized gradient flow of the symplectic action},
   JOURNAL = {Comm. Pure Appl. Math.},
  FJOURNAL = {Communications on Pure and Applied Mathematics},
    VOLUME = {41},
      YEAR = {1988},
    NUMBER = {6},
     PAGES = {775--813},
      ISSN = {0010-3640},
   MRCLASS = {58F05},
  MRNUMBER = {948771},
MRREVIEWER = {Holger Kantz},
       DOI = {10.1002/cpa.3160410603},
       URL = {https://doi.org/10.1002/cpa.3160410603},
}

@article {Floer1988a,
    AUTHOR = {Floer, Andreas},
     TITLE = {Morse theory for {L}agrangian intersections},
   JOURNAL = {J. Differential Geom.},
  FJOURNAL = {Journal of Differential Geometry},
    VOLUME = {28},
      YEAR = {1988},
    NUMBER = {3},
     PAGES = {513--547},
      ISSN = {0022-040X},
   MRCLASS = {58F05 (35J65 58E05)},
  MRNUMBER = {965228},
MRREVIEWER = {Jean-Claude Sikorav},
       URL = {http://projecteuclid.org/euclid.jdg/1214442477},
}

@article{witten-susy-morse-theory,
author = {Edward Witten},
title = {{Supersymmetry and Morse theory}},
volume = {17},
journal = {Journal of Differential Geometry},
number = {4},
publisher = {Lehigh University},
pages = {661 -- 692},
year = {1982},
doi = {10.4310/jdg/1214437492},
URL = {https://doi.org/10.4310/jdg/1214437492}
}

@misc{freed-cobordism,
Author = {Daniel S. Freed},
Title = {The cobordism hypothesis},
Year = {2012},
Eprint = {arXiv:1210.5100},
}

@article{Schwarz1997,
  doi = {10.1007/bf02506415},
  url = {https://doi.org/10.1007/bf02506415},
  year = {1997},
  month = jan,
  publisher = {Springer Science and Business Media {LLC}},
  volume = {183},
  number = {2},
  pages = {463--476},
  author = {Albert Schwarz and Oleg Zaboronsky},
  title = {Supersymmetry and localization},
  journal = {Communications in Mathematical Physics}
}

@article{witten-tqft,
author = {Edward Witten},
title = {{Topological quantum field theory}},
volume = {117},
journal = {Communications in Mathematical Physics},
number = {3},
publisher = {Springer},
pages = {353 -- 386},
year = {1988},
doi = {cmp/1104161738},
URL = {https://doi.org/}
}

@article{floer-infinite-dimensional-morse-theory,
author = {Andreas Floer},
title = {{Witten's complex and infinite-dimensional Morse theory}},
volume = {30},
journal = {Journal of Differential Geometry},
number = {1},
publisher = {Lehigh University},
pages = {207 -- 221},
year = {1989},
doi = {10.4310/jdg/1214443291},
URL = {https://doi.org/10.4310/jdg/1214443291}
}

@misc{ahsan-khan,
author= {Ahsan Khan},
title = {Upcoming publication}
}

@article{Baez1995,
  doi = {10.1063/1.531236},
  url = {https://doi.org/10.1063/1.531236},
  year = {1995},
  month = nov,
  publisher = {{AIP} Publishing},
  volume = {36},
  number = {11},
  pages = {6073--6105},
  author = {John C. Baez and James Dolan},
  title = {Higher-dimensional algebra and topological quantum field theory},
  journal = {Journal of Mathematical Physics}
}

@Inbook{Hofer1995,
author="Hofer, Helmut
and Salamon, Dietmar A.",
editor="Hofer, Helmut
and Taubes, Clifford H.
and Weinstein, Alan
and Zehnder, Eduard",
title="Floer homology and Novikov rings",
bookTitle="The Floer Memorial Volume",
year="1995",
publisher="Birkh{\"a}user Basel",
address="Basel",
pages="483--524",
isbn="978-3-0348-9217-9",
doi="10.1007/978-3-0348-9217-9_20",
url="https://doi.org/10.1007/978-3-0348-9217-9_20"
}

@misc{Moore-Segal,
Author = {Gregory W. Moore and Graeme Segal},
Title = {D-branes and K-theory in 2D topological field theory},
Year = {2006},
Eprint = {arXiv:hep-th/0609042},
}

@InProceedings{BlumbergCohenTeleman,
  author       = "Andrew J. Blumberg, Ralph L. Cohen, Constantin Teleman",
  title        = "Open-closed field theories, string topology, and Hochschild homology",
  booktitle    = "Alpine Perspectives on Algebraic Topology",
  year         = "2009",
  editor       = "C. Ausoni, K. Hess, and J. Scherer",
  volume       = "504",
  number       = "",
  series       = "Contemporary Mathematics",
  pages        = "",
  month        = "",
  address      = "",
  organization = "",
  publisher    = "American Mathematical Society",
  note         = ""
}

@misc{lurie-tqft,
Author = {Jacob Lurie},
Title = {On the Classification of Topological Field Theories},
Year = {2009},
Eprint = {arXiv:0905.0465},
}

@article{atiyah-tqft,
     author = {Atiyah, Michael F.},
     title = {Topological quantum field theory},
     journal = {Publications Math\'ematiques de l'IH\'ES},
     pages = {175--186},
     publisher = {Institut des Hautes \'Etudes Scientifiques},
     volume = {68},
     year = {1988},
     zbl = {0692.53053},
     mrnumber = {90e:57059},
     language = {en},
     url = {http://www.numdam.org/item/PMIHES_1988__68__175_0/}
}

@book {Seidel2008,
    AUTHOR = {Seidel, Paul},
     TITLE = {Fukaya categories and {P}icard-{L}efschetz theory},
    SERIES = {Zurich Lectures in Advanced Mathematics},
 PUBLISHER = {European Mathematical Society (EMS), Z\"{u}rich},
      YEAR = {2008},
     PAGES = {viii+326},
      ISBN = {978-3-03719-063-0},
   MRCLASS = {53D40 (16E45 32Q65 53D12)},
  MRNUMBER = {2441780},
MRREVIEWER = {Timothy Perutz},
       DOI = {10.4171/063},
       URL = {https://doi.org/10.4171/063},
}

@article {Kapranov2016,
    AUTHOR = {Kapranov, M. and Kontsevich, M. and Soibelman, Y.},
     TITLE = {Algebra of the infrared and secondary polytopes},
   JOURNAL = {Adv. Math.},
  FJOURNAL = {Advances in Mathematics},
    VOLUME = {300},
      YEAR = {2016},
     PAGES = {616--671},
      ISSN = {0001-8708},
   MRCLASS = {17B55 (16E45 18E30 52B20 53D37 81R05)},
  MRNUMBER = {3534842},
MRREVIEWER = {Leandro R. Cagliero},
       DOI = {10.1016/j.aim.2016.03.028},
       URL = {https://doi.org/10.1016/j.aim.2016.03.028},
}

@InProceedings{SeidelVanishingCycles,
author="Seidel, Paul",
editor="Casacuberta, Carles
and Mir{\'o}-Roig, Rosa Maria
and Verdera, Joan
and Xamb{\'o}-Descamps, Sebasti{\`a}",
title="Vanishing Cycles and Mutation",
booktitle="European Congress of Mathematics",
year="2001",
publisher="Birkh{\"a}user Basel",
address="Basel",
pages="65--85",
isbn="978-3-0348-8266-8"
}

@misc{Witten-Analytic-Chern-Simons,
Author = {Edward Witten},
Title = {Analytic Continuation Of Chern-Simons Theory},
Year = {2010},
Eprint = {arXiv:1001.2933},
}

@article{witten1993algebraic,
  title={Algebraic geometry associated with matrix models of two-dimensional gravity},
  author={Witten, Edward},
  journal={Topological methods in modern mathematics (Stony Brook, NY, 1991)},
  volume={235},
  year={1993}
}

@article{Fan2013,
  doi = {10.4007/annals.2013.178.1.1},
  url = {https://doi.org/10.4007/annals.2013.178.1.1},
  year = {2013},
  month = jul,
  publisher = {Annals of Mathematics},
  volume = {178},
  number = {1},
  pages = {1--106},
  author = {Huijun Fan and Tyler Jarvis and Yongbin Ruan},
  title = {The Witten equation,  mirror symmetry,  and quantum singularity theory},
  journal = {Annals of Mathematics}
}

@article{fan2018fukaya,
  title={Fukaya category of Landau-Ginzburg model},
  author={Fan, Huijun and Jiang, Wenfeng and Yang, Dingyu},
  journal={arXiv preprint arXiv:1812.11748},
  year={2018}
}

@article {Haydys2015,
    AUTHOR = {Haydys, Andriy},
     TITLE = {Fukaya-{S}eidel category and gauge theory},
   JOURNAL = {J. Symplectic Geom.},
  FJOURNAL = {The Journal of Symplectic Geometry},
    VOLUME = {13},
      YEAR = {2015},
    NUMBER = {1},
     PAGES = {151--207},
      ISSN = {1527-5256},
   MRCLASS = {53D37},
  MRNUMBER = {3338833},
MRREVIEWER = {Hai-Long Her},
       DOI = {10.4310/JSG.2015.v13.n1.a5},
       URL = {https://doi.org/10.4310/JSG.2015.v13.n1.a5},
}

@misc{Gaiotto2015,
  doi = {10.48550/ARXIV.1506.04087},
  url = {https://arxiv.org/abs/1506.04087},
  author = {Gaiotto, Davide and Moore, Gregory W. and Witten, Edward},
  title = {Algebra of the Infrared: String Field Theoretic Structures in Massive ${\cal N}=(2,2)$ Field Theory In Two Dimensions},
  publisher = {arXiv},
  year = {2015},
  copyright = {arXiv.org perpetual, non-exclusive license}
}

@article{kapustin-rozansky-saulinas,
title = {Three-dimensional topological field theory and symplectic algebraic geometry I},
journal = {Nuclear Physics B},
volume = {816},
number = {3},
pages = {295-355},
year = {2009},
issn = {0550-3213},
doi = {https://doi.org/10.1016/j.nuclphysb.2009.01.027},
url = {https://www.sciencedirect.com/science/article/pii/S0550321309000583},
author = {Anton Kapustin and Lev Rozansky and Natalia Saulina}
}

@misc{kapustin-rozansky,
Author = {Anton Kapustin and Lev Rozansky},
Title = {Three-dimensional topological field theory and symplectic algebraic geometry II},
Year = {2009},
Eprint = {arXiv:0909.3643},
}

@misc{gaitsgory-sheaves-of-categories,
Author = {Dennis Gaitsgory},
Title = {Sheaves of categories and the notion of 1-affineness},
Year = {2013},
Eprint = {arXiv:1306.4304},
}

@misc{arinkin-singular-support, 
author={D. Arinkin}, 
title={Singular support of categories}, 
year=2020,
howpublished={Talk at M-seminar. Notes/video available at
\url{https://www.math.k-state.edu/research/m-center/past_seminars/2020.html}.}
}

@article{BPLW,
  title={Quantizations of conical symplectic resolutions II: category O  and symplectic duality},
  author={Braden, Tom and Licata, Anthony and Proudfoot, Nicholas and Webster, Ben},
  journal={arXiv preprint arXiv:1407.0964},
  year={2014}
}

@article{NadlerZaslow,
  title={Constructible sheaves and the Fukaya category},
  author={Nadler, David and Zaslow, Eric},
  journal={Journal of the American Mathematical Society},
  volume={22},
  number={1},
  pages={233--286},
  year={2009}
}

@article{kontsevich2000homological,
  title={Homological mirror symmetry and torus fibrations},
  author={Kontsevich, Maxim and Soibelman, Yan},
  journal={arXiv preprint math/0011041},
  year={2000}
}

@phdthesis{bottman-thesis,
  author       = {Nathaniel Bottman}, 
  title        = {A Monoidal Structure for the Fukaya Category},
  school       = {Harvard University},
  year         = 2010,
}

@misc{gammage-shende,
Author = {Benjamin Gammage and Vivek Shende},
Title = {Mirror symmetry for very affine hypersurfaces},
Year = {2017},
Eprint = {arXiv:1707.02959},
}

@article{Abbondandolo2005,
  doi = {10.1002/cpa.20090},
  url = {https://doi.org/10.1002/cpa.20090},
  year = {2005},
  publisher = {Wiley},
  volume = {59},
  number = {2},
  pages = {254--316},
  author = {Alberto Abbondandolo and Matthias Schwarz},
  title = {On the Floer homology of cotangent bundles},
  journal = {Communications on Pure and Applied Mathematics}
}

@article{Balmer2005,
  doi = {10.1515/crll.2005.2005.588.149},
  url = {https://doi.org/10.1515/crll.2005.2005.588.149},
  year = {2005},
  month = nov,
  publisher = {Walter de Gruyter {GmbH}},
  volume = {2005},
  number = {588},
  pages = {149--168},
  author = {Paul Balmer},
  title = {The spectrum of prime ideals in tensor triangulated categories},
  journal = {Journal f\"{u}r die reine und angewandte Mathematik (Crelles Journal)}
}

@article{TelemanMorseBott,
  doi = {10.1215/00127094-2019-0048},
  url = {https://doi.org/10.1215/00127094-2019-0048},
  year = {2020},
  month = feb,
  publisher = {Duke University Press},
  volume = {169},
  number = {3},
  author = {Constantin Teleman},
  title = {Matrix factorization of Morse{\textendash}Bott functions},
  journal = {Duke Mathematical Journal}
}

@misc{teleman-icm,
Author = {Constantin Teleman},
Title = {Gauge theory and mirror symmetry},
Year = {2014},
Eprint = {arXiv:1404.6305},
}

@InCollection{Orlov04,
 Author = {Orlov, D. O.},
 Title = {Triangulated categories of singularities and {D}-branes in {Landau}-{Ginzburg} models},
 BookTitle = {Algebraic geometry. Methods, relations, and applications. Collected papers. Dedicated to the memory of Andrei Nikolaevich Tyurin.},
 Pages = {227--248},
 Year = {2004},
 Publisher = {Moscow: Maik Nauka/Interperiodica},
 Language = {English},
 zbMATH = {5071248},
 Zbl = {1101.81093}
}

@article {Gaiotto2017,
    AUTHOR = {Gaiotto, Davide and Moore, Gregory W. and Witten, Edward},
     TITLE = {An introduction to the web-based formalism},
   JOURNAL = {Confluentes Math.},
  FJOURNAL = {Confluentes Mathematici},
    VOLUME = {9},
      YEAR = {2017},
    NUMBER = {2},
     PAGES = {5--48},
   MRCLASS = {81T30 (37K40 53D37 81Q60 81T40)},
  MRNUMBER = {3745160},
       DOI = {10.5802/cml.40},
       URL = {https://doi.org/10.5802/cml.40},
}

@misc{ebeling-singularities-survey,
Author = {Wolfgang Ebeling},
Title = {Homological mirror symmetry for singularities},
Year = {2016},
Eprint = {arXiv:1601.06027},
}

@article{Khovanov2001,
  doi = {10.1090/s0894-0347-01-00374-5},
  url = {https://doi.org/10.1090/s0894-0347-01-00374-5},
  year = {2001},
  month = sep,
  publisher = {American Mathematical Society ({AMS})},
  volume = {15},
  number = {1},
  pages = {203--271},
  author = {Mikhail Khovanov and Paul Seidel},
  title = {Quivers,  Floer cohomology,  and braid group actions},
  journal = {Journal of the American Mathematical Society}
}

@article{Berglund1993,
  doi = {10.1016/0550-3213(93)90250-s},
  url = {https://doi.org/10.1016/0550-3213(93)90250-s},
  year = {1993},
  month = mar,
  publisher = {Elsevier {BV}},
  volume = {393},
  number = {1-2},
  pages = {377--391},
  author = {Per Berglund and Tristan H\"{u}bsch},
  title = {A generalized construction of mirror manifolds},
  journal = {Nuclear Physics B}
}

@incollection {Donaldson1998,
    AUTHOR = {Donaldson, S. K. and Thomas, R. P.},
     TITLE = {Gauge theory in higher dimensions},
 BOOKTITLE = {The geometric universe ({O}xford, 1996)},
     PAGES = {31--47},
 PUBLISHER = {Oxford Univ. Press, Oxford},
      YEAR = {1998},
   MRCLASS = {57R57 (14J32 32J18 53C07 57R58 58D27)},
  MRNUMBER = {1634503},
MRREVIEWER = {Krzysztof Galicki},
}

@misc{Wang2022,
  doi = {10.48550/ARXIV.2209.02810},
  url = {https://arxiv.org/abs/2209.02810},
  author = {Wang, Donghao},
  title = {The Complex Gradient Flow Equation and Seidel's Spectral Sequence},
  publisher = {arXiv},
  year = {2022},
  copyright = {Creative Commons Attribution 4.0 International}
}

@misc{nakajima-coloumb-1,
Author = {Hiraku Nakajima},
Title = {Towards a mathematical definition of Coulomb branches of $3$-dimensional $\mathcal N=4$ gauge theories, I},
Year = {2015},
Eprint = {arXiv:1503.03676},
}

@book{LurieHA, 
author={Jacob Lurie},
title={Higher Algebra}, 
year={2017}}

@book{gaitsgory-rozenblym,
    title = {A Study in Derived Algebraic Geometry},
    author = {Dennis Gaitsgory and Nick Rozenblyum},
    year = {2017},
    month = {08},
    day = {30},
    pagecount = {1016},
    isbn = {9781470435684}
}

@misc{gaitsgory-on-traces,
Author = {D. Gaitsgory and D. Kazhdan and N. Rozenblyum and Y. Varshavsky},
Title = {A toy model for the Drinfeld-Lafforgue shtuka construction},
Year = {2019},
Eprint = {arXiv:1908.05420},
}

@Inbook{vanderPut2003,
author="van der Put, Marius
and Singer, Michael F.",
title="Monodromy, the Riemann-Hilbert Problem, and the Differential Galois Group",
bookTitle="Galois Theory of Linear Differential Equations",
year="2003",
publisher="Springer Berlin Heidelberg",
address="Berlin, Heidelberg",
pages="143--155",
isbn="978-3-642-55750-7",
doi="10.1007/978-3-642-55750-7_5",
url="https://doi.org/10.1007/978-3-642-55750-7_5"
}

@misc{MasseyPerverseSheafNotes,
Author = {David B. Massey},
Title = {Notes on Perverse Sheaves and Vanishing Cycles},
Year = {1999},
Eprint = {arXiv:math/9908107},
}

@article{Brav2015,
  doi = {10.5427/jsing.2015.11e},
  url = {https://doi.org/10.5427/jsing.2015.11e},
  year = {2015},
  publisher = {Journal of Singularities},
  author = {Christopher Brav and Vittoria Bussi and Delphine Dupont and Dominic Joyce and Bal{\'{a}}zs Szendroi and J\"{o}rg Sch\"{u}rmann},
  title = {Symmetries and stabilization for sheaves of vanishing cycles},
  journal = {Journal of Singularities}
}

@article{Solomon2019,
  doi = {10.1112/s0010437x1900753x},
  url = {https://doi.org/10.1112/s0010437x1900753x},
  year = {2019},
  month = sep,
  publisher = {Wiley},
  volume = {155},
  number = {10},
  pages = {1924--1958},
  author = {Jake P. Solomon and Misha Verbitsky},
  title = {Locality in the Fukaya category of a~hyperk\"{a}hler manifold},
  journal = {Compositio Mathematica}
}

@misc{Abouzaid-Ganatra, 
author={M. Abouzaid and S. Ganatra.}, 
title = {Generating Fukaya categories of Landau-Ginzburg models.}, 
howpublished={Talks at the MIT workshop on Lefschetz fibrations}, 
year={2015} 
}

@misc{PIRSA-khan,
author={Tudor Dimofte and Benjamin Gammage and Justin Hilburn and Ahsan Khan},
title={Discussion on 3d A-model},
howpublished={Videotaped discussion at the Perimiter Institute, \url{https://pirsa.org/22060086}},
year={2022}
}

@article{GPS1,
  doi = {10.1007/s10240-019-00112-x},
  url = {https://doi.org/10.1007/s10240-019-00112-x},
  year = {2019},
  month = aug,
  publisher = {Springer Science and Business Media {LLC}},
  volume = {131},
  number = {1},
  pages = {73--200},
  author = {Sheel Ganatra and John Pardon and Vivek Shende},
  title = {Covariantly functorial wrapped Floer theory on Liouville sectors},
  journal = {Publications math{\'{e}}matiques de l{\textquotesingle}{IH}{\'{E}}S}
}

@article{AbouzaidGeneration,
  doi = {10.1007/s10240-010-0028-5},
  url = {https://doi.org/10.1007/s10240-010-0028-5},
  year = {2010},
  month = oct,
  publisher = {Springer Science and Business Media {LLC}},
  volume = {112},
  number = {1},
  pages = {191--240},
  author = {Mohammed Abouzaid},
  title = {A geometric criterion for generating the Fukaya category},
  journal = {Publications math{\'{e}}matiques de l{\textquotesingle}{IH}{\'{E}}S}
}

@misc{SeidelSymplecticHomology,
Author = {Paul Seidel},
Title = {Symplectic homology as Hochschild homology},
Year = {2006},
Eprint = {arXiv:math/0609037},
}

@misc{Kontsevich2020,
  doi = {10.48550/ARXIV.2005.10651},
  url = {https://arxiv.org/abs/2005.10651},
  author = {Kontsevich, Maxim and Soibelman, Yan},
  title = {Analyticity and resurgence in wall-crossing formulas},
  publisher = {arXiv},
  year = {2020},
  copyright = {arXiv.org perpetual, non-exclusive license}
}

@misc{GPS,
Author = {Sheel Ganatra and John Pardon and Vivek Shende},
Title = {Microlocal Morse theory of wrapped Fukaya categories},
Year = {2018},
Eprint = {arXiv:1809.08807},
}

@article{Sylvan2019,
  doi = {10.1112/topo.12088},
  url = {https://doi.org/10.1112/topo.12088},
  year = {2019},
  month = jan,
  publisher = {Wiley},
  volume = {12},
  number = {2},
  pages = {372--441},
  author = {Zachary Sylvan},
  title = {On partially wrapped Fukaya categories},
  journal = {Journal of Topology}
}

@misc{Esfahani2022,  
  author = {Esfahani, Saman}, 
  title = {Monopole Fueter Floer homology (in preparation)}, 
}

@book{loday2013cyclic,
  title={Cyclic homology},
  author={Loday, Jean-Louis},
  volume={301},
  year={2013},
  publisher={Springer Science \& Business Media}
}

@article{ganatra2019cyclic,
  title={Cyclic homology, $S^1$-equivariant floer cohomology, and calabi-yau structures},
  author={Ganatra, Sheel},
  journal={arXiv preprint arXiv:1912.13510},
  year={2019}
}

@article{gammage2022betti,
  title={Betti Tate's thesis and the trace of perverse schobers},
  author={Gammage, Benjamin and Hilburn, Justin},
  journal={arXiv preprint arXiv:2210.06548},
  year={2022}
}

@misc{Bousseau2022,
Author = {Pierrick Bousseau},
Title = {Holomorphic Floer theory and Donaldson-Thomas invariants},
Year = {2022},
Eprint = {arXiv:2210.17001},
}

@article {Doan2021,
    AUTHOR = {Doan, Aleksander and Walpuski, Thomas},
     TITLE = {On the existence of harmonic {$\rm Z_2$} spinors},
   JOURNAL = {J. Differential Geom.},
  FJOURNAL = {Journal of Differential Geometry},
    VOLUME = {117},
      YEAR = {2021},
    NUMBER = {3},
     PAGES = {395--449},
      ISSN = {0022-040X},
   MRCLASS = {53C27 (57K41)},
  MRNUMBER = {4255067},
MRREVIEWER = {Nicolas Ginoux},
       DOI = {10.4310/jdg/1615487003},
       URL = {https://doi.org/10.4310/jdg/1615487003},
}

@misc{Taubes2014,
  doi = {10.48550/ARXIV.1407.6206},
  url = {https://arxiv.org/abs/1407.6206},
  author = {Taubes, Clifford Henry},
  title = {The zero loci of Z/2 harmonic spinors in dimension 2, 3 and 4},
  publisher = {arXiv},
  year = {2014},
  copyright = {arXiv.org perpetual, non-exclusive license}
}

@article {Takahashi2018,
    AUTHOR = {Takahashi, Ryosuke},
     TITLE = {Index theorem for {$\Bbb Z/2$}-harmonic spinors},
   JOURNAL = {Math. Res. Lett.},
  FJOURNAL = {Mathematical Research Letters},
    VOLUME = {25},
      YEAR = {2018},
    NUMBER = {5},
     PAGES = {1645--1671},
      ISSN = {1073-2780},
   MRCLASS = {58J20 (53C27 58D27 70S15)},
  MRNUMBER = {3917743},
MRREVIEWER = {Sebastian Goette},
       DOI = {10.4310/MRL.2018.v25.n5.a13},
       URL = {https://doi.org/10.4310/MRL.2018.v25.n5.a13},
}

@article {Donaldson2021,
    AUTHOR = {Donaldson, Simon},
     TITLE = {Deformations of multivalued harmonic functions},
   JOURNAL = {Q. J. Math.},
  FJOURNAL = {The Quarterly Journal of Mathematics},
    VOLUME = {72},
      YEAR = {2021},
    NUMBER = {1-2},
     PAGES = {199--235},
      ISSN = {0033-5606},
   MRCLASS = {53C43 (31C12 42B10 47F10 58H15)},
  MRNUMBER = {4271385},
MRREVIEWER = {Kamil Niedzia\l omski},
       DOI = {10.1093/qmath/haab018},
       URL = {https://doi.org/10.1093/qmath/haab018},
}

@article {Haydys2012,
    AUTHOR = {Haydys, Andriy},
     TITLE = {Gauge theory, calibrated geometry and harmonic spinors},
   JOURNAL = {J. Lond. Math. Soc. (2)},
  FJOURNAL = {Journal of the London Mathematical Society. Second Series},
    VOLUME = {86},
      YEAR = {2012},
    NUMBER = {2},
     PAGES = {482--498},
      ISSN = {0024-6107},
   MRCLASS = {53C38 (53C07 53C29)},
  MRNUMBER = {2980921},
       DOI = {10.1112/jlms/jds008},
       URL = {https://doi.org/10.1112/jlms/jds008},
}

@article {Haydys2022,
    AUTHOR = {Haydys, Andriy},
     TITLE = {Seiberg-{W}itten monopoles and flat {${\rm PSL}(2, \Bbb
              R)$}-connections},
   JOURNAL = {Adv. Math.},
  FJOURNAL = {Advances in Mathematics},
    VOLUME = {409},
      YEAR = {2022},
     PAGES = {Paper No. 108686},
      ISSN = {0001-8708},
   MRCLASS = {Prelim},
  MRNUMBER = {4483244},
       DOI = {10.1016/j.aim.2022.108686},
       URL = {https://doi.org/10.1016/j.aim.2022.108686},
}

@article {Mazzeo2020,
    AUTHOR = {Mazzeo, Rafe and Witten, Edward},
     TITLE = {The {KW} equations and the {N}ahm pole boundary condition with
              knots},
   JOURNAL = {Comm. Anal. Geom.},
  FJOURNAL = {Communications in Analysis and Geometry},
    VOLUME = {28},
      YEAR = {2020},
    NUMBER = {4},
     PAGES = {871--942},
      ISSN = {1019-8385},
   MRCLASS = {57K41 (53C07 57K14)},
  MRNUMBER = {4165312},
MRREVIEWER = {Derek G. Harland},
       DOI = {10.4310/CAG.2020.v28.n4.a4},
       URL = {https://doi.org/10.4310/CAG.2020.v28.n4.a4},
}

@misc{Daemi2021,
  doi = {10.48550/ARXIV.2109.07038},
  url = {https://arxiv.org/abs/2109.07038},
  author = {Daemi, Aliakbar and Fukaya, Kenji and Lipyanskiy, Maksim},
  title = {Lagrangians, SO(3)-instantons and mixed equation},
  publisher = {arXiv},
  year = {2021},
  copyright = {arXiv.org perpetual, non-exclusive license}
}

@book {Cieliebak2012,
    AUTHOR = {Cieliebak, Kai and Eliashberg, Yakov},
     TITLE = {From {S}tein to {W}einstein and back},
    SERIES = {American Mathematical Society Colloquium Publications},
    VOLUME = {59},
      NOTE = {Symplectic geometry of affine complex manifolds},
 PUBLISHER = {American Mathematical Society, Providence, RI},
      YEAR = {2012},
     PAGES = {xii+364},
      ISBN = {978-0-8218-8533-8},
   MRCLASS = {53-02 (32E10 32Q65 53C15 53D05 58E05)},
  MRNUMBER = {3012475},
MRREVIEWER = {Chris M. Wendl},
       DOI = {10.1090/coll/059},
       URL = {https://doi.org/10.1090/coll/059},
}
